\documentclass[11pt]{amsart}
\usepackage[hmarginratio=1:1]{geometry}                
\usepackage{graphicx}
\usepackage{amsmath}%
\usepackage{amsfonts}%
\usepackage{amssymb}
\usepackage{amsthm}
\usepackage{caption}
\usepackage{subcaption}
\usepackage{stmaryrd}
\usepackage{mathtools}
\usepackage{cite}
\usepackage{tikz}
\usetikzlibrary{matrix,arrows,calc,shapes.geometric}

\theoremstyle{plain}
\newtheorem{Theorem}{Theorem}[section]
\newtheorem{Lemma}{Lemma}[section]

\theoremstyle{definition}
\newtheorem{Definition}{Definition}[section]

\newtheorem{Example}{Example}[section]

\theoremstyle{remark}
\newtheorem{remark}{Remark}[section]

\numberwithin{equation}{section}

\newcommand{\la}{\langle}
\newcommand{\ra}{\rangle}
\newcommand{\mf}{\mathfrak}
\newcommand{\mbb}{\mathbb}
\newcommand{\mbf}{\mathbf}
\newcommand{\mc}{\mathcal}
\newcommand{\on}{\operatorname}

\newcommand{\ad}{\mathbf{ad}}

\newcommand{\g}{\mathfrak{g}}

\newcommand{\h}{\mathfrak{h}}

\newcommand{\ann}{\on{ann}}

\newcommand{\Mod}{{\mathrm{-Mod}}}
\newcommand{\Alg}{{\mathrm{-Alg}}}
\newcommand{\CAlg}{{\mathrm{-CAlg}}}
\newcommand{\CPoiss}{{\mathrm{-CPoiss}}}
\newcommand{\Top}{{\mathrm{Top}}}
\newcommand{\Hom}{\operatorname{Hom}}
\newcommand{\MCat}{\operatorname{MonCat}}
\newcommand{\id}{\operatorname{id}}
\newcommand{\K}{\mbb{K}}
\newcommand{\Col}{{\mc{CS}\mathrm{urf}}}
\newcommand{\sCol}{{\mathrm{sk}\mc{CS}\mathrm{urf}}}
\newcommand{\FSet}{{\on{FSet}}}
\providecommand{\abs}[1]{\lvert#1\rvert}

\newcommand{\mmat}[2][3em]{\matrix (#2) [matrix of math nodes, row sep=#1,
  column sep=#1, text height=1.5ex, text depth=0.25ex]}
\tikzset{node distance=2cm, auto}

\title{Coherent Quantization using Coloured Surfaces}
\author{David Li-Bland and Pavol \v{S}evera}

\begin{document}
\begin{abstract}
In this note, we revisit the quantization of Lie bialgebras described by the second author in \cite{Severa:2014te}, placing it in the more general framework of quantizing moduli spaces developed in \cite{LiBland:2014da}. In particular, we show that embeddings of quilted surfaces (which are compatible with the choice of skeleton) induce morphisms between the corresponding quantized moduli spaces of flat connections. As an application, we describe quantizations of both the variety of Lagrangian subalgebras and the de-Cocini Procesi wonderful compactification, which are compatible with the action of the (quantized) Poisson Lie group.
\end{abstract}
\maketitle
\section{Introduction}


\subsection{Coloured Surfaces, and Moduli Spaces}\label{sec:GeomColSurf}
Suppose that $\Sigma$ is a compact oriented surface with boundary and $\mbf{Walls}:=\{\mbf{w}_i\}_{i\in I}\subset\partial \Sigma$ is a finite collection of disjoint segments in the boundary, which we call \emph{domain walls}. We require that every connected component of $\Sigma$ contains a domain wall. 
\begin{center}
\begingroup%
  \makeatletter%
  \providecommand\color[2][]{%
    \errmessage{(Inkscape) Color is used for the text in Inkscape, but the package 'color.sty' is not loaded}%
    \renewcommand\color[2][]{}%
  }%
  \providecommand\transparent[1]{%
    \errmessage{(Inkscape) Transparency is used (non-zero) for the text in Inkscape, but the package 'transparent.sty' is not loaded}%
    \renewcommand\transparent[1]{}%
  }%
  \providecommand\rotatebox[2]{#2}%
  \ifx\svgwidth\undefined%
    \setlength{\unitlength}{44.22585449bp}%
    \ifx\svgscale\undefined%
      \relax%
    \else%
      \setlength{\unitlength}{\unitlength * \real{\svgscale}}%
    \fi%
  \else%
    \setlength{\unitlength}{\svgwidth}%
  \fi%
  \global\let\svgwidth\undefined%
  \global\let\svgscale\undefined%
  \makeatother%
  \begin{picture}(1,0.98857799)%
    \put(0,0){\includegraphics[width=\unitlength,page=1]{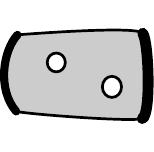}}%
    \put(3.43464069,0.44540333){\color[rgb]{0,0,0}\makebox(0,0)[rb]{\smash{}}}%
  \end{picture}%
\endgroup%
\\
	Fig 1) The domain walls on this surface are indicated by thickened segments.
\end{center}

Now suppose that $\g$ is a (finite dimensional) Lie algebra over a field $\mbb{K}$ of characteristic zero, and $t\in(S^2\g)^\g$ is a chosen $\ad$-invariant element. We observe that $t$ defines a $\g$-invariant bilinear form on the dual space $\g^*$; and we say that a subalgebra $\mf{c}\subseteq\g$ is \emph{coisotropic} if the restriction of this bilinear form to the annihilator $\ann(\mf{c})\subseteq \g^*$ of $\mf{c}$ vanishes.\footnote{Equivalently, the image of $t$ under the projection $S^2\g\to S^2(\g/\mf{c})$ is zero.} Similarly, suppose that $G$ is an algebraic group\footnote{By an \emph{algebraic group} we mean an affine group scheme of finite type over $\mbb{K}$; in particular, since $\mbb{K}$ is of characteristic zero, an algebraic group over $\mbb{K}$ is smooth.} with Lie algebra $\g$ and that $t\in(S^2\g)$ is $G$-invariant, we say that an algebraic subgroup $C\subseteq G$ is \emph{coisotropic} if the corresponding Lie subalgebra of $\g$ is coisotropic.
 
A \emph{colouring} of the domain walls is an assignment $\mbf{w}\mapsto C_\mbf{w}$ of a coisotropic subgroup of $G$ to each domain wall. We call the data $(\Sigma,\mbf{Walls},C_\cdot)$ a \emph{coloured surface}. A typical example will be depicted as follows:
\begin{center}
\begingroup%
  \makeatletter%
  \providecommand\color[2][]{%
    \errmessage{(Inkscape) Color is used for the text in Inkscape, but the package 'color.sty' is not loaded}%
    \renewcommand\color[2][]{}%
  }%
  \providecommand\transparent[1]{%
    \errmessage{(Inkscape) Transparency is used (non-zero) for the text in Inkscape, but the package 'transparent.sty' is not loaded}%
    \renewcommand\transparent[1]{}%
  }%
  \providecommand\rotatebox[2]{#2}%
  \ifx\svgwidth\undefined%
    \setlength{\unitlength}{177.69462891bp}%
    \ifx\svgscale\undefined%
      \relax%
    \else%
      \setlength{\unitlength}{\unitlength * \real{\svgscale}}%
    \fi%
  \else%
    \setlength{\unitlength}{\svgwidth}%
  \fi%
  \global\let\svgwidth\undefined%
  \global\let\svgscale\undefined%
  \makeatother%
  \begin{picture}(1,0.24604403)%
    \put(0,0){\includegraphics[width=\unitlength,page=1]{ColSurfA.pdf}}%
    \put(0.38219965,0.1655566){\color[rgb]{0,0,0}\makebox(0,0)[lb]{\smash{Domain walls}}}%
    \put(0.43213677,0.11331394){\color[rgb]{0,0,0}\makebox(0,0)[lb]{\smash{$C_+\subseteq G$}}}%
    \put(0.43213677,0.05928867){\color[rgb]{0,0,0}\makebox(0,0)[lb]{\smash{$C_-\subseteq G$}}}%
    \put(0.85483687,0.11085501){\color[rgb]{0,0,0}\makebox(0,0)[rb]{\smash{}}}%
  \end{picture}%
\endgroup%
\end{center}

We define a morphism between coloured surfaces $(\Sigma,\mbf{Walls},C_\cdot)\to (\Sigma',\mbf{Walls}',C_\cdot')$
	to be an orientation preserving embedding of surfaces $\Sigma\to\Sigma'$ which maps domain walls coloured by a given coisotropic subgroup of $G$ into domain walls coloured by the same coisotropic subgroup.
	
In \cite{LiBland:2012vo,LiBland:2013ue}, we studied the moduli space which classifies flat $G$-bundles on $\Sigma$ with reductions of structure along the domain walls dictated by the corresponding coisotropic subgroups. We showed that the moduli space carries a natural Poisson structure generalizing the Atiyah-Bott Poisson structure on moduli spaces (cf. \cite{Atiyah:1983dt,Goldman:1986eh}). In particular, we constructed a functor from the category of coloured surfaces to the category of Poisson algebras.

Subsequently,  we described a deformation-quantization procedure for the Poisson algebras arising from coloured surfaces \cite{LiBland:2014da}. That deformation quantization procedure depends upon a choice of combinatorial decomposition of the surface (which we call a \emph{skeleton} for the surface see \S~\ref{sec:Skeleta}).\footnote{This is closely related to a framing of the surface.} The main result in the current paper is to show that this deformation-quantization procedure is functorial with respect to embeddings of coloured surfaces which are compatible with the chosen skeleta. We apply this result to
\begin{itemize}
	\item deform finite dimensional Poisson algebraic/Lie groups to quantum groups (following ideas of \cite{Severa:2014te}),
	\item  to equivalently quantize various classical phase spaces such as the variety of Lagrangian subalgebras \cite{Drinfeld:1993il,Evens:2001ue,Evens:2006kk} and the de Concini-Procesi wonderful compactification of a Poisson algebraic group \cite{DeConcini:1983ki,Evens:2001ue,Evens:2006kk,LiBland:2009hx}, and
	\item to show that (formally) any Poisson homogeneous space $M$ for a finite dimensional Poisson algebraic group $H$ can be canonically $H$-equivariantly embedded as a Poisson submanifold of a larger Poisson $H$-space $\tilde M$, which can be equivariantly quantized (this builds on the ideas of Enriquez and Kosmann-Schwarzbach \cite{Enriquez:r0OMuB1e}).  
\end{itemize}

To our knowledge, this procedure provides the first equivariant quantizations of the variety of Lagrangian subalgebras and the deConcini-Procesi wonderful compactification of a Poisson Lie group. While Poisson Lie groups have been quantized by Etingof and Kazhdan in a celebrated series of  papers \cite{Etingof:1996bc,Etingof:1998ix,Etingof:2000gc,Etingof:2000gj,Etingof:2008dm,Etingof:1998vf}, our approach to this problem provides a different perspective which is closely related to topological field theories. Moreover, it describes the geometric picture underlying \v{S}evera's quantization of Lie bialgebras \cite{Severa:2014te}.

Allow us to sketch our deformation quantization of Poisson Lie/algebraic groups to Hopf algebras in some more detail. Suppose $H$ is a Poisson Lie group, $G$ is the double, and $H^*$ is the dual Poisson Lie group. For simplicity we will assume that the product map $$H^*\times H\xrightarrow{h,h'\mapsto h\cdot h'}G$$ is a diffeomorphism. Let $\mc{M}_\Sigma$ denote the moduli space  of flat $G$ bundles over the following coloured surface
\begin{center}
\begingroup%
  \makeatletter%
  \providecommand\color[2][]{%
    \errmessage{(Inkscape) Color is used for the text in Inkscape, but the package 'color.sty' is not loaded}%
    \renewcommand\color[2][]{}%
  }%
  \providecommand\transparent[1]{%
    \errmessage{(Inkscape) Transparency is used (non-zero) for the text in Inkscape, but the package 'transparent.sty' is not loaded}%
    \renewcommand\transparent[1]{}%
  }%
  \providecommand\rotatebox[2]{#2}%
  \ifx\svgwidth\undefined%
    \setlength{\unitlength}{228.18520508bp}%
    \ifx\svgscale\undefined%
      \relax%
    \else%
      \setlength{\unitlength}{\unitlength * \real{\svgscale}}%
    \fi%
  \else%
    \setlength{\unitlength}{\svgwidth}%
  \fi%
  \global\let\svgwidth\undefined%
  \global\let\svgscale\undefined%
  \makeatother%
  \begin{picture}(1,0.33248722)%
    \put(0,0){\includegraphics[width=\unitlength,page=1]{PoissonLie.pdf}}%
    \put(0.67488036,0.16124026){\color[rgb]{0,0,0}\makebox(0,0)[lb]{\smash{$H\subset G$}}}%
    \put(0.67488036,0.09112178){\color[rgb]{0,0,0}\makebox(0,0)[lb]{\smash{$H^*\subset G$}}}%
    \put(0.60476188,0.23135874){\color[rgb]{0,0,0}\makebox(0,0)[lb]{\smash{Domain walls}}}%
    \put(0.23207915,0.15442469){\color[rgb]{0,0,0}\makebox(0,0)[rb]{\smash{$\Sigma=$}}}%
  \end{picture}%
\endgroup%

\end{center}
with the indicated reductions along the coloured domain walls. Then $\mc{M}_\Sigma$ is canonically isomorphic to the Poisson manifold $H$ (cf. \cite{LiBland:2013ue}). However, the identification of the moduli space $\mc{M}_{\Sigma_{G,H,H^*}}$ with a Poisson Lie group is more than coincidental: the following cobordism (from the top horizontal slice to the bottom horizontal slice)
\begin{center}
\begingroup%
  \makeatletter%
  \providecommand\color[2][]{%
    \errmessage{(Inkscape) Color is used for the text in Inkscape, but the package 'color.sty' is not loaded}%
    \renewcommand\color[2][]{}%
  }%
  \providecommand\transparent[1]{%
    \errmessage{(Inkscape) Transparency is used (non-zero) for the text in Inkscape, but the package 'transparent.sty' is not loaded}%
    \renewcommand\transparent[1]{}%
  }%
  \providecommand\rotatebox[2]{#2}%
  \ifx\svgwidth\undefined%
    \setlength{\unitlength}{315.30371094bp}%
    \ifx\svgscale\undefined%
      \relax%
    \else%
      \setlength{\unitlength}{\unitlength * \real{\svgscale}}%
    \fi%
  \else%
    \setlength{\unitlength}{\svgwidth}%
  \fi%
  \global\let\svgwidth\undefined%
  \global\let\svgscale\undefined%
  \makeatother%
  \begin{picture}(1,0.37479702)%
    \put(0,0){\includegraphics[width=\unitlength,page=1]{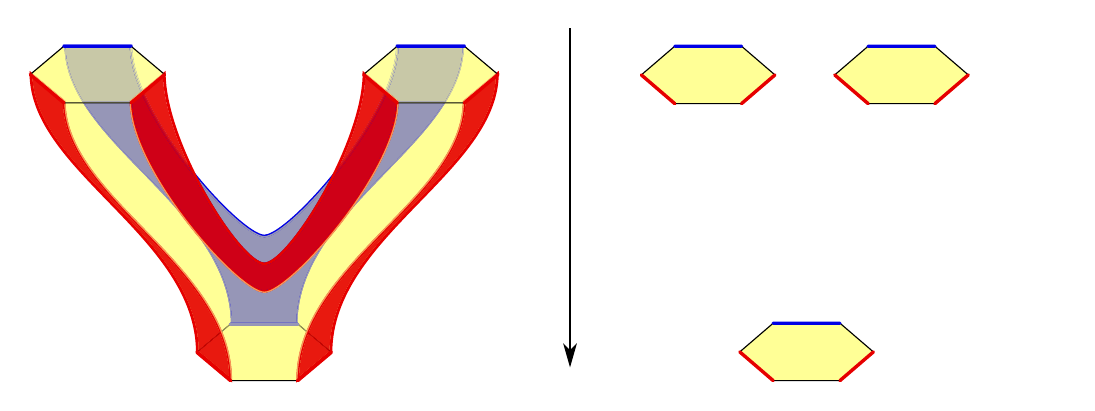}}%
    \put(0.54930483,0.04698375){\color[rgb]{0,0,0}\makebox(0,0)[b]{\smash{$t$}}}%
    \put(0.80588161,0.04698375){\color[rgb]{0,0,0}\makebox(0,0)[lb]{\smash{$t=1$}}}%
    \put(0.89722211,0.30070738){\color[rgb]{0,0,0}\makebox(0,0)[lb]{\smash{$t=0$}}}%
  \end{picture}%
\endgroup%

\end{center}
 corresponds to a group product on the moduli space, \begin{equation}\label{eq:IntroMult}\mc{M}_{\Sigma_{G,H,H^*}}\times \mc{M}_{\Sigma_{G,H,H^*}}\to \mc{M}_{\Sigma_{G,H,H^*}},\end{equation} and further cobordisms correspond to the group unit and the various structural axioms for a group. To define \eqref{eq:IntroMult} rigourously, we need to be careful: although we are unable to handle these cobordisms directly, we reinterpret them as sequences of embeddings (essentially by slicing them along the critical values for a Morse function):
\begin{center}
\begingroup%
  \makeatletter%
  \providecommand\color[2][]{%
    \errmessage{(Inkscape) Color is used for the text in Inkscape, but the package 'color.sty' is not loaded}%
    \renewcommand\color[2][]{}%
  }%
  \providecommand\transparent[1]{%
    \errmessage{(Inkscape) Transparency is used (non-zero) for the text in Inkscape, but the package 'transparent.sty' is not loaded}%
    \renewcommand\transparent[1]{}%
  }%
  \providecommand\rotatebox[2]{#2}%
  \ifx\svgwidth\undefined%
    \setlength{\unitlength}{337.15371094bp}%
    \ifx\svgscale\undefined%
      \relax%
    \else%
      \setlength{\unitlength}{\unitlength * \real{\svgscale}}%
    \fi%
  \else%
    \setlength{\unitlength}{\svgwidth}%
  \fi%
  \global\let\svgwidth\undefined%
  \global\let\svgscale\undefined%
  \makeatother%
  \begin{picture}(1,0.40249562)%
    \put(0,0){\includegraphics[width=\unitlength]{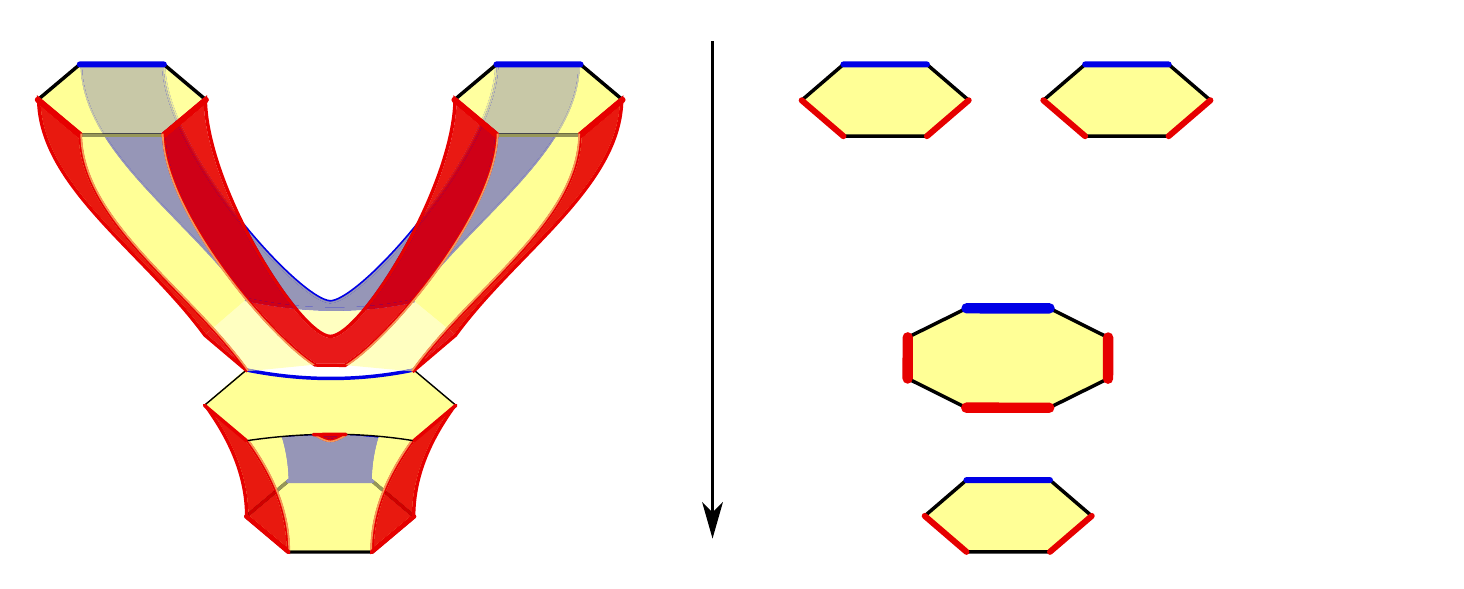}}%
    \put(0.5089603,0.0724125){\color[rgb]{0,0,0}\makebox(0,0)[b]{\smash{$t$}}}%
    \put(0.75365465,0.04393884){\color[rgb]{0,0,0}\makebox(0,0)[lb]{\smash{$t=1$}}}%
    \put(0.77263709,0.15308784){\color[rgb]{0,0,0}\makebox(0,0)[lb]{\smash{$t=\frac{1}{2}$}}}%
    \put(0.83907563,0.32867544){\color[rgb]{0,0,0}\makebox(0,0)[lb]{\smash{$t=0$}}}%
  \end{picture}%
\endgroup%

\end{center}
This results in the following sequence of embeddings:
\begin{center}
\begingroup%
  \makeatletter%
  \providecommand\color[2][]{%
    \errmessage{(Inkscape) Color is used for the text in Inkscape, but the package 'color.sty' is not loaded}%
    \renewcommand\color[2][]{}%
  }%
  \providecommand\transparent[1]{%
    \errmessage{(Inkscape) Transparency is used (non-zero) for the text in Inkscape, but the package 'transparent.sty' is not loaded}%
    \renewcommand\transparent[1]{}%
  }%
  \providecommand\rotatebox[2]{#2}%
  \ifx\svgwidth\undefined%
    \setlength{\unitlength}{444.86015625bp}%
    \ifx\svgscale\undefined%
      \relax%
    \else%
      \setlength{\unitlength}{\unitlength * \real{\svgscale}}%
    \fi%
  \else%
    \setlength{\unitlength}{\svgwidth}%
  \fi%
  \global\let\svgwidth\undefined%
  \global\let\svgscale\undefined%
  \makeatother%
  \begin{picture}(1,0.23917552)%
    \put(0,0){\includegraphics[width=\unitlength,page=1]{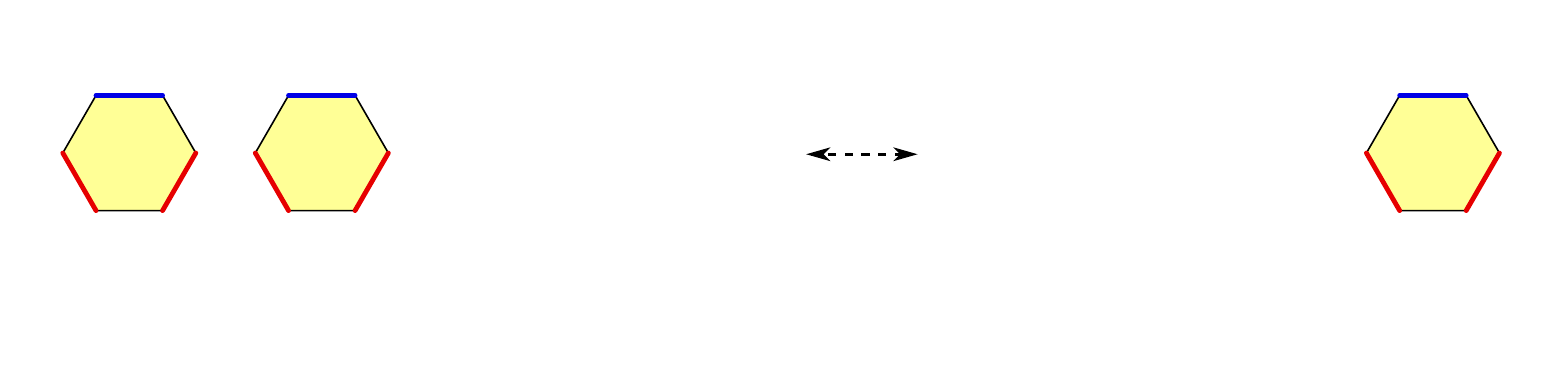}}%
    \put(0.55897301,0.14604063){\color[rgb]{0,0,0}\makebox(0,0)[b]{\smash{$=$}}}%
    \put(0.43118712,0.02169749){\color[rgb]{0,0,0}\makebox(0,0)[b]{\smash{$\Sigma^{(2)}$}}}%
    \put(0.68654827,0.02169749){\color[rgb]{0,0,0}\makebox(0,0)[b]{\smash{$\Sigma^{(2)}$}}}%
    \put(0.92752287,0.02169749){\color[rgb]{0,0,0}\makebox(0,0)[b]{\smash{$\Sigma$}}}%
    \put(0.14705288,0.02169749){\color[rgb]{0,0,0}\makebox(0,0)[b]{\smash{$\Sigma\sqcup\Sigma$}}}%
    \put(0,0){\includegraphics[width=\unitlength,page=2]{PoissonLieMult.pdf}}%
  \end{picture}%
\endgroup%
	
\end{center}
Finally, using the fact that our deformation quantization of $\mc{M}_\Sigma$ is  functorial with respect to embeddings (Theorem~\ref{thm:CohQuant}) yields a compatible coproduct on the deformation quantization of $\mc{M}_\Sigma$, i.e. it deforms $H\cong \mc{M}_\Sigma$ to a Hopf algebra.

In future work we hope to show that a variation of our deformation quantization procedure doesn't depend upon the choice of skeleta for our coloured surfaces. The close resemblence of our deformation quantization procedure with the quantization developed by Ben-Zvi, Brochier, and Jordan \cite{BenZvi:Rwmg8j55}, provides evidence that this should be possible. Moreover, we hope the resulting functor will extend naturally to cobordisms between coloured surfaces and not just embeddings.

\subsection{Acknowledgements}
The authors would like to thank Anton Alekseev, David Ben-Zvi, Yuri Berest, Adrien Brochier, Pavel Etingof, and David Jordan, for helpful discussions and suggestions. 

David Li-Bland was supported by NSF Grant DMS-1204779. Pavol Severa was supported by the grant MODFLAT of the European Research Council and the NCCR SwissMAP of the Swiss National Science Foundation.

\section{Preliminaries}

\subsection{Combinatorics}
An \emph{ordered morphism} $p:I\to J$ between two finite sets $I$ and $J$ is a map $p:I\to J$ of the underlying sets equipped with a linear order on every fibre $p^{-1}(j)$. To compose two ordered morphisms, one first composes the underlying maps of sets, and then equips the resulting fibres with the lexicographic composite of the linear orders.

A convenient way to picture an ordered morphism $p:I\to J$ is to use a \emph{polygonal representation}: for each $i\in I$ one draws a planar polygon with $\abs{p^{-1}(i)}+1$ edges (we allow `polygons' with one, two, or more edges), and one labels the bottom edge with $i$, and the remaining edges with the elements of $p^{-1}(i)$ in the counter-clockwise order. For example, the ordered morphism 
\begin{subequations}\label{eq:ordMorphP}\begin{equation}p:\{1_a,2_a,1_b,2_b,3_b,1_c\}\xrightarrow{i_\alpha\mapsto \alpha}\{a,b,c\}\end{equation}
(with the obvious linear order on the fibres) can be pictured as follows:
\begin{equation}\begin{tikzpicture}
	\coordinate (a) at (0,0);
	\coordinate (1a) at (2,0);
	\coordinate (2a) at (1,1.62);
	\draw[-,very thick] (a) -- node [swap] {$a$} (1a);
	\draw[-] (1a) -- node [swap] {$1_a$} (2a);
	\draw[-] (2a) -- node [swap] {$2_a$} (a);
	\begin{scope}[shift={(4,0)}]
	\coordinate (b) at (0,0);
	\coordinate (1b) at (2,0);
	\coordinate (2b) at (2,2);
	\coordinate (3b) at (0,2);
	\draw[-,very thick] (b) -- node [swap] {$b$} (1b);
	\draw[-] (1b) -- node [swap] {$1_b$} (2b);
	\draw[-] (2b) -- node [swap] {$2_b$} (3b);
	\draw[-] (3b) -- node [swap] {$3_b$} (b);
	\end{scope}
	\begin{scope}[shift={(8,0)}]
	\coordinate (c) at (0,0);
	\coordinate (1c) at (2,0);
	\draw[-,very thick] (c) -- node [swap] {$c$} (1c);
	\draw[-] (1c) edge [bend right] node [swap] {$1_c$} (c);
	\end{scope}
\end{tikzpicture}\end{equation}
\end{subequations}
Composition then corresponds to gluing along the indicated edges. For example the composite of $p$ with $$q:\{a,b,c\}\underset{b\mapsto y}{\xrightarrow{a,c\mapsto x}}\{x,y\}$$
$$\begin{tikzpicture}
	\coordinate (x) at (0,0);
	\coordinate (a) at (2,0);
	\coordinate (c) at (1,1.62);
	\draw[-,very thick] (x) -- node [swap] {$x$} (a);
	\draw[-] (a) -- node [swap] {$a$} (c);
	\draw[-] (c) -- node [swap] {$c$} (x);
	\begin{scope}[shift={(4,0)}]
	\coordinate (y) at (0,0);
	\coordinate (b) at (2,0);
	\draw[-,very thick] (y) -- node [swap] {$y$} (b);
	\draw[-] (b) edge [bend right] node [swap] {$b$} (y);
	\end{scope}
\end{tikzpicture}$$
is pictured as
$$\begin{tikzpicture}
	\coordinate (x) at (0,0);
	\coordinate (1a) at (2,0);
	\coordinate (2a) at (2,2);
	\coordinate (1c) at (0,2);
	\draw[-,very thick] (x) -- node [swap] {$x$} (1a);
	\draw[-] (1a) -- node [swap] {$1_a$} (2a);
	\draw[-] (2a) -- node [swap] {$2_a$} (1c);
	\draw[-] (1c) -- node [swap] {$1_c$} (x);
	\draw[-,dotted] (1c) -- node {$a$} (1a);
	\draw[-,dotted] (1c) edge [bend left] node {$c$} (x);
	\begin{scope}[shift={(4,0)}]
	\coordinate (b) at (0,0);
	\coordinate (1b) at (2,0);
	\coordinate (2b) at (2,2);
	\coordinate (3b) at (0,2);
	\draw[-,very thick] (b) -- node [swap] {$y$} (1b);
	\draw[-,dotted] (b) edge [bend left] node {$b$} (1b);
	\draw[-] (1b) -- node [swap] {$1_b$} (2b);
	\draw[-] (2b) -- node [swap] {$2_b$} (3b);
	\draw[-] (3b) -- node [swap] {$3_b$} (b);
	\end{scope}
\end{tikzpicture}$$
where we have glued along the dotted lines.


 We let $\FSet^<$ denote the category consisting of finite sets and ordered morphisms. There is an obvious forgetful functor $\FSet^<\to \FSet$ to the category of finite sets.
 
Extending the concept somewhat, a \emph{parenthesized ordered morphism} $p:I\to J$ between two finite sets $I$ and $J$ is an ordered morphism $p:I\to J$,  together with a parenthesization of the ordered set $p^{-1}(j)$ for each $j\in J$. For example, the ordered morphism $\{1,2,3\}\to \{x\}$ (with $1<2<3$), can be equipped with two different parenthesizations: either $(1\quad 2)\quad 3$ or $1\quad (2\quad 3)$. Equivalently, equipping an ordered morphism with a parenthesization is equivalent to triangulating the corresponding polygonal representation. For example, the two possible parentesizations of $\{1,2,3\}\to \{x\}$ are:
\begin{subequations}
\begin{equation}\label{eq:TwoBasicParenth}\begin{tikzpicture}
	\coordinate (b) at (0,0);
	\coordinate (1b) at (2,0);
	\coordinate (2b) at (2,2);
	\coordinate (3b) at (0,2);
	\draw[-,very thick] (b) -- node [swap] {$x$} (1b);
	\draw[-] (1b) -- node [swap] {$1$} (2b);
	\draw[-] (2b) -- node [swap] {$2$} (3b);
	\draw[-] (3b) -- node [swap] {$3$} (b);
	\draw[-,thin] (3b) -- (1b);
	\node (m) at (1,-1) {$(1\quad 2)\quad 3$};
	\begin{scope}[shift={(6,0)}]
	\coordinate (b) at (0,0);
	\coordinate (1b) at (2,0);
	\coordinate (2b) at (2,2);
	\coordinate (3b) at (0,2);
	\draw[-,very thick] (b) -- node [swap] {$x$} (1b);
	\draw[-] (1b) -- node [swap] {$1$} (2b);
	\draw[-] (2b) -- node [swap] {$2$} (3b);
	\draw[-] (3b) -- node [swap] {$3$} (b);
	\draw[-,thin] (2b) -- (b);
	\node (m) at (1,-1) {$1\quad (2\quad 3)$};
	\end{scope}
\end{tikzpicture}\end{equation}
Similarly, the five possible parenthesizations of $\{1,2,3,4\}\to \{y\}$ (with $1<2<3<4$) are:
\begin{equation}\label{eq:PentParenth}\begin{tikzpicture}
	\coordinate (y) at (0,0);
	\coordinate (1) at (2,0);
	\coordinate (2) at (2.5,1.8);
	\coordinate (3) at (1,3);
	\coordinate (4) at (-.5,1.8);
	\draw[-,very thick] (y) -- node [swap] {$x$} (1);
	\draw[-] (1) -- node [swap] {$1$} (2);
	\draw[-] (2) -- node [swap] {$2$} (3);
	\draw[-] (3) -- node [swap] {$3$} (4);
	\draw[-] (4) -- node [swap] {$4$} (y);
	\draw[-,thin] (3) -- (1);
	\draw[-,thin] (3) -- (y);
	\node (m) at (1,-1) {$(1\quad 2)\quad (3\quad 4)$};
	
\begin{scope}[shift={(4.5,-3.6)}]
	\coordinate (y) at (0,0);
	\coordinate (1) at (2,0);
	\coordinate (2) at (2.5,1.8);
	\coordinate (3) at (1,3);
	\coordinate (4) at (-.5,1.8);
	\draw[-,very thick] (y) -- node [swap] {$x$} (1);
	\draw[-] (1) -- node [swap] {$1$} (2);
	\draw[-] (2) -- node [swap] {$2$} (3);
	\draw[-] (3) -- node [swap] {$3$} (4);
	\draw[-] (4) -- node [swap] {$4$} (y);
	\draw[-,thin] (2) -- (y);
	\draw[-,thin] (3) -- (y);	
	\node (m) at (1,-1) {$1\quad (2\quad (3\quad 4))$};
\end{scope}
	
\begin{scope}[shift={(-4.5,-3.6)}]
	\coordinate (y) at (0,0);
	\coordinate (1) at (2,0);
	\coordinate (2) at (2.5,1.8);
	\coordinate (3) at (1,3);
	\coordinate (4) at (-.5,1.8);
	\draw[-,very thick] (y) -- node [swap] {$x$} (1);
	\draw[-] (1) -- node [swap] {$1$} (2);
	\draw[-] (2) -- node [swap] {$2$} (3);
	\draw[-] (3) -- node [swap] {$3$} (4);
	\draw[-] (4) -- node [swap] {$4$} (y);
	\draw[-,thin] (3) -- (1);
	\draw[-,thin] (4) -- (1);
	\node (m) at (1,-1) {$((1\quad 2)\quad 3)\quad 4$};
\end{scope}
	
\begin{scope}[shift={(3,-9)}]
	\coordinate (y) at (0,0);
	\coordinate (1) at (2,0);
	\coordinate (2) at (2.5,1.8);
	\coordinate (3) at (1,3);
	\coordinate (4) at (-.5,1.8);
	\draw[-,very thick] (y) -- node [swap] {$x$} (1);
	\draw[-] (1) -- node [swap] {$1$} (2);
	\draw[-] (2) -- node [swap] {$2$} (3);
	\draw[-] (3) -- node [swap] {$3$} (4);
	\draw[-] (4) -- node [swap] {$4$} (y);
	\draw[-,thin] (4) -- (2);
	\draw[-,thin] (2) -- (y);
	\node (m) at (1,-1) {$1\quad ((2\quad 3)\quad 4)$};
\end{scope}
	
\begin{scope}[shift={(-3,-9)}]
	\coordinate (y) at (0,0);
	\coordinate (1) at (2,0);
	\coordinate (2) at (2.5,1.8);
	\coordinate (3) at (1,3);
	\coordinate (4) at (-.5,1.8);
	\draw[-,very thick] (y) -- node [swap] {$x$} (1);
	\draw[-] (1) -- node [swap] {$1$} (2);
	\draw[-] (2) -- node [swap] {$2$} (3);
	\draw[-] (3) -- node [swap] {$3$} (4);
	\draw[-] (4) -- node [swap] {$4$} (y);
	\draw[-,thin] (4) -- (2);
	\draw[-,thin] (4) -- (1);
	\node (m) at (1,-1) {$(1\quad (2\quad 3))\quad 4$};
\end{scope}

\end{tikzpicture}\end{equation}
\end{subequations}

As before, triangulated polygonal representations of parenthesized ordered morphisms are composed by gluing.

While the choice of parenthesization which one equips an ordered morphism with will be important, any two parenthesizations should be equivalent in a unique way. To encode this, we define a 2-category $\FSet^{(<)}$:
\begin{itemize}
	\item whose objects are finite sets,
	\item whose 1-morphisms $p:I\to J$ between finite sets $I$ and $J$ are parenthesized ordered morphisms,
	\item with a unique 2-morphism $\alpha:p\to p'$ between any two 1-morphisms $p,p':I\to J$ which both have the same underlying ordered morphism.
\end{itemize}
Note that $\FSet^{(<)}$ and $\FSet^{<}$ are naturally biequivalent as 2-categories. 
Indeed, all 2-morphisms are invertible, and for any two finite sets $I$ and $J$, the groupoid\footnote{Recall that a groupoid is a category all of whose morphisms are invertible} $\Hom_{\FSet^{(<)}}(I,J)$ of 1-morphisms between $I$ and $J$ is equivalent to the set $\Hom_{\FSet^{<}}(I,J)$ of ordered morphisms between $I$ and $J$. In particular, the forgetful functor $\FSet^{(<)}\to \FSet^{<}$ which is the identity on objects, sends any parenthesized ordered morphism to the underlying ordered morphism, and any 2-morphism to the identity, defines a biequivalence of 2-categories. 

Finally, we note that there is a canonical `order reversal' endofunctor on $\FSet^<$ (and $\FSet^{(<)}$) which fixes the objects but reverses the linear orders for every ordered morphism; we denote this order reversal by $p\to \bar p$ for any ordered morphism $p:I\to J$ (in terms of their (triangulated) polygonal representations, this corresponds to reversing the orientation of the polygons - i.e. flipping them over).



\subsection{The Drinfel'd Category}
Let $\g$ be a Lie algebra over a field $\mbb{K}$ of characteristic zero, and suppose that  $t\in(S^2\g)^\g$ is a chosen $\ad$-invariant element.

Let $\Phi\in\mathbb{K}\la\!\la x,y\ra\!\ra$ be a Drinfeľd associator. The element $t\in(S^2\g)^G$ and the associator $\Phi$ may be used to deform the symmetric monoidal structure on the category $U(\g)\Mod$ to a braided monoidal structure:
 More precisely, let $U(\g)\Mod^\Phi_\hbar$ be the category with the same objects as $U(\g)\Mod$, and with
$$\Hom_{U(\g)\Mod^\Phi_\hbar}(X,Y)=\Hom_{U(\g)\Mod}(X,Y)[\![\hbar]\!].$$
The tensor product of objects and morphisms are the same in both categories, but the braiding in $U(\g)\Mod^\Phi_\hbar$ is the symmetry in $U(\g)\Mod$ composed with the action of $\exp(\hbar t^{1,2}/2)\in U(\g\oplus\g)[\![\hbar]\!]$, and the associativity constraint is given by the action of $\Phi(\hbar t^{1,2},\hbar t^{2,3})\in  U(\g\oplus\g\oplus\g)[\![\hbar]\!]$. See \cite{Drinfeld:1989tu} for details.

\subsection{Quantum Fusion}
Suppose that $\g$ and $\h$ are Lie algebras with chosen elements $t_\g\in(S^2\g)^\g$, $t_\h\in(S^2\h)^\h$.

Given a finite set $J$, we use the shorthand $\g^J:=\bigoplus_{j\in J}\g\cong \on{Maps}(J,\g)$, and given a map $p:I\to J$ between finite sets, we let $$p^!:\g^J\cong \on{Maps}(J,\g)\xrightarrow{(\xi_\cdot:J\to \g)\mapsto(\xi_{p(\cdot)}:I\to \g)} \on{Maps}(I,\g)\cong\g^I$$ denote the corresponding pullback. Since $p^!:\g^J\to \g^I$ is a morphism of Lie algebras, we get a corresponding functor
\begin{subequations}\label{eq:qFusionForp}
\begin{equation}\label{eq:FpUnderFunct}F^p:U(\g^I\oplus\h)\Mod^\Phi_\hbar\to U(\g^J\oplus\h)\Mod^\Phi_\hbar\end{equation} between the (deformed) categories of modules.

If $p:I\to J$ is a parenthesized ordered morphism, then \eqref{eq:FpUnderFunct} can be upgraded to a monoidal functor (cf. \cite[Theorem~2]{LiBland:2014da}) as follows: For each $j\in J$, let $I_j=p^{-1}(j)$; we choose an order preserving identification $I_j=(1_j,2_j,\dots,n_j)$, where $n_j=\abs{I_j}$. We may construct a new parenthesized ordered set
$$I_j^{(a)(b)}:=(1_j^a\quad \cdots\quad n_j^a)(1_j^b\quad \cdots\quad n_j^b),$$
where the terms in each half are parenthesized as in $I_j$, and a second parenthesized ordered set
$$I_j^{(ab)}:=((1_j^a\quad 1_j^b)\quad \cdots\quad (n_j^a\quad n_j^b)),$$
where the pairs are parenthesized as in $I_j$. We let $B_j$ denote the parenthesized braid from $I_j^{(a)(b)}$ to $I_j^{(ab)}$ which identifies the corresponding elements, all the strands connecting elements of the form $i_j^a$ move strictly to the right, the strands connected elements of the form $i_j^b$ move strictly to the left, and the rightward moving strands pass over the leftward moving strands. For example, if $I_j=(1\quad 2)$, then $B_j$ is the parenthesized braid 
$$\begin{tikzpicture}[baseline=1cm]
\coordinate (diff) at (0.65,0);
\coordinate (dy) at (0,0.5);
\node(X1) at (0,0) {$(1^a$};
\node(Y1) at ($(X1)+(diff)$) {$2^a)$};
\node(Z1) at (2,0) {$(1^b$};
\node(W1) at ($(Z1)+(diff)$) {$2^b)$};
\node(X2) at (0,2) {$(1^a$};
\node(Z2) at ($(X2)+(diff)$) {$1^b)$};
\node(Y2) at (2,2) {$(2^a$};
\node(W2) at ($(Y2)+(diff)$) {$2^b)$};
\draw(X1)--(X2);
\draw(W1)--(W2);
\draw (Z1)..controls +(0,1) and +(0,-1)..(Z2);
\draw[line width=1ex,white] (Y1)..controls +(0,1) and +(0,-1)..(Y2);
\draw(Y1)..controls +(0,1) and +(0,-1)..(Y2);
\end{tikzpicture}$$
if $I_j=((1\quad 2)\quad 3)$, then $B_j$ is the parenthesized braid 
$$\begin{tikzpicture}
[baseline=1cm]
\coordinate (diff) at (0.75,0);
\coordinate (dy) at (0,0.5);
\node(X1b) at (0,0) {$((1^a$};
\node(X2b) at ($(X1b)+(diff)$) {$2^a)$};
\node(X3b) at ($(X2b)+(diff)$) {$3^a)$};
\node(Y1b) at (3,0) {$((1^b$};
\node(Y2b) at ($(Y1b)+(diff)$) {$2^b)$};
\node(Y3b) at ($(Y2b)+(diff)$) {$3^b)$};
\node(X1t) at (0,3) {$((1^a$};
\node(Y1t) at ($(X1t)+(diff)$) {$1^b)$};
\node(X2t) at ($(Y1t)+(diff)$) {$(2^a$};
\node(Y2t) at  ($(X2t)+(diff)$) {$2^b))$};
\node(X3t) at (3.75,3) {$(3^a$};
\node(Y3t) at  ($(X3t)+(diff)$) {$3^b)$};
\draw (Y1b)..controls +(0,1) and +(0,-1)..(Y1t);
\draw (Y2b)..controls +(0,1) and +(0,-1)..(Y2t);
\draw(Y3b)--(Y3t);
\draw(X1b)--(X1t);
\draw[line width=1ex,white] (X2b)..controls +(0,1) and +(0,-1)..(X2t);
\draw (X2b)..controls +(0,1) and +(0,-1)..(X2t);
\draw[line width=1ex,white] (X3b)..controls +(0,1) and +(0,-1)..(X3t);
\draw (X3b)..controls +(0,1) and +(0,-1)..(X3t);

\end{tikzpicture}$$

We let $K_j\in \big(U(\g)^{\otimes I_j}\big)^{\otimes 2}[\![\hbar]\!]\cong U(\g^{I_j})^{\otimes 2}[\![\hbar]\!]$ denote the element corresponding to the parenthesized braid $B_j$. Then the functor \eqref{eq:FpUnderFunct} can be equipped with a monoidal structure via the coherence map 
\begin{equation}\label{eq:FpCoherence}F^p(X^a)\otimes F^p(X^b)\xrightarrow{\nu^p\cdot}F^p(X^a\otimes X^b),\quad X^a,X^b\in U(\g^I\oplus\h)\Mod^\Phi_\hbar
	\end{equation}
given by the action of \begin{equation}\nu^p:=\prod_{j\in J}K_j \in U(\g^I)^{\otimes 2}[\![\hbar]\!]\end{equation}
	\end{subequations} (cf. \cite[Theorem~2]{LiBland:2014da}). The unit coherence map is the identity, as before. We call the corresponding monoidal functor the \emph{(quantum) fusion} functor associated to the parenthesized ordered morphism $p:I\to J$.

\begin{Theorem}\label{thm:CohAss}
	There is a  strict 2-functor (called \emph{(quantum) fusion}) \begin{equation}\label{eq:QuantFus2Funct}F^{(-)}:\FSet^{(<)}\to \MCat\end{equation} from the 2-category  of parenthesized ordered morphisms betweeen finite sets to the 2-category of monoidal categories, monoidal functors, and monoidal natural transformations, which sets any finite set $I$ to $U(\g^I\oplus\h)\Mod^\Phi_\hbar$, and any parenthesized ordered morphism $p:I\to J$ to the quantum fusion functor \eqref{eq:qFusionForp} described above.
\end{Theorem}
\begin{proof}
	The fact that the natural transformation \eqref{eq:FpCoherence} satisfies the axioms for a monoidal structure on \eqref{eq:FpUnderFunct} follows from \cite[Theorem~2]{LiBland:2014da} by iterated fusion using the given parenthesization.\footnote{Of course, the parenthesization doesn't give a canonical order in which to perform the iterated fusion, but the final result doesn't depend on the various choices. Indeed, quantum fusion along disjoint ordered pairs commutes: see \cite[Lemma~3.5]{BarNatan:1998he}, for example, where this fact is expressed as `locality in space'.} Moreover, it is clear that the underlying functor \eqref{eq:FpUnderFunct} depends naturally on the underlying map of sets, $p:I\to J$, while the parenthesized braid (and hence the coherence map \eqref{eq:FpCoherence}) depends naturally on the parenthesized ordered morphism $p:I\to J$. Thus, \eqref{eq:QuantFus2Funct} certainly restricts to an ordinary functor between the underlying categories.
	
	It remains to define \eqref{eq:QuantFus2Funct} on the 2-morphisms, and to show that it is compatible with both horizontal and vertical composition of 2-morphisms. We begin by considering the two basic parenthesized ordered morphisms pictured in \eqref{eq:TwoBasicParenth}, which we denote by $((\cdot\cdot)\cdot)$ and $(\cdot(\cdot\cdot))$, respectively. Since the underlying maps of sets coincide, the underlying functors \eqref{eq:FpUnderFunct} are identical $F^{((\cdot\cdot)\cdot)}=D^{(3)}=F^{(\cdot(\cdot\cdot))}$, and only the coherence maps \eqref{eq:FpCoherence} differ. However, we claim that there is a natural monoidal isomorphism $(D^{(3)},\nu^{((\cdot\cdot)\cdot)})\dasharrow (D^{(3)},\nu^{(\cdot(\cdot\cdot))})$
 between the two functors, 
 \begin{equation}\label{eq:MonIso}D^{(3)}(X)\xrightarrow{(\Phi^{1,2,3}\cdot)} D^{(3)}(X),\end{equation}
given by the natural action of the associator $\Phi^{1,2,3}\in U(\g\oplus\g\oplus\g)[\![\hbar]\!]$ on $X\in U(\g\oplus\g\oplus\g\oplus \h)\Mod^\Phi_\hbar$. More precisely, \eqref{eq:MonIso} makes the following diagrams commute:
\begin{equation}\label{eq:CohAssSq}\begin{tikzpicture}
\mmat[4em]{m}{ D^{(3)}(X^a)\otimes D^{(3)}(X^b) & &D^{(3)}(X^a)\otimes D^{(3)}(X^b)\\
D^{(3)}(X^a\otimes X^b) & &D^{(3)}(X^a\otimes X^b)\\};
\draw[->] (m-1-1) edge node {$(\Phi^{1a,2a,3a}\Phi^{1b,2b,3b}\cdot)$} (m-1-3)
(m-1-1) edge node {$\nu^{((\cdot\cdot)\cdot)}_{X,Y}$} (m-2-1)
(m-1-3) edge node {$\nu^{(\cdot(\cdot\cdot))}_{X,Y}$} (m-2-3)
(m-2-1) edge node {$(\Phi^{1a1b,2a2b,3a3b}\cdot)$} (m-2-3);
\end{tikzpicture}\end{equation}
Indeed, both composites in this diagram correspond to the parenthesized braid
$$\begin{tikzpicture}
[baseline=1cm]
\coordinate (diff) at (0.75,0);
\coordinate (dy) at (0,0.5);
\node(X1b) at (0,0) {$((1a$};
\node(X2b) at ($(X1b)+(diff)$) {$2a)$};
\node(X3b) at ($(X2b)+(diff)$) {$3a)$};
\node(Y1b) at (3,0) {$((1b$};
\node(Y2b) at ($(Y1b)+(diff)$) {$2b)$};
\node(Y3b) at ($(Y2b)+(diff)$) {$3b)$};
\node(X1t) at (0,3) {$(1a$};
\node(Y1t) at ($(X1t)+(diff)$) {$1b)$};
\node(X2t) at (2.25,3) {$((2a$};
\node(Y2t) at  ($(X2t)+(diff)$) {$2b)$};
\node(X3t) at ($(Y2t)+(diff)$) {$(3a$};
\node(Y3t) at  ($(X3t)+(diff)$) {$3b))$};
\draw (Y1b)..controls +(0,1) and +(0,-1)..(Y1t);
\draw (Y2b)..controls +(0,1) and +(0,-1)..(Y2t);
\draw(Y3b)--(Y3t);
\draw(X1b)--(X1t);
\draw[line width=1ex,white] (X2b)..controls +(0,1) and +(0,-1)..(X2t);
\draw (X2b)..controls +(0,1) and +(0,-1)..(X2t);
\draw[line width=1ex,white] (X3b)..controls +(0,1) and +(0,-1)..(X3t);
\draw (X3b)..controls +(0,1) and +(0,-1)..(X3t);
\end{tikzpicture}$$


Next, we consider the five parenthesized ordered morphisms pictured in \eqref{eq:PentParenth}. In each case, the underlying map of sets is the same, and hence the underlying functors \eqref{eq:FpUnderFunct} $D^{(4)}:U(\g^4\oplus \h)\Mod^\Phi_\hbar\to U(\g\oplus \h)\Mod^\Phi_\hbar$ coincide. Also notice that any two adjacent parenthesizations differ by a `basic reparenthesization' of the form $((\cdot\cdot)\cdot) \leftrightarrow(\cdot(\cdot\cdot))$,\footnote{In terms of triangulations of the corresponding polyonal representation, a `basic reparenthesization' corresponds to a type 2-2 pachner move.} and we claim that the following diagram commutes
\begin{equation}\label{eq:CohAssPent}\begin{tikzpicture}
\mmat{m}{\big(D^{(4)},\nu^{(((\cdot\cdot)\cdot)\cdot)}\big) & \big(D^{(4)},\nu^{((\cdot(\cdot\cdot))\cdot)}\big) & \big(D^{(4)},\nu^{(\cdot((\cdot\cdot)\cdot))}\big)\\
\big(D^{(4)},\nu^{((\cdot\cdot)(\cdot\cdot))}\big)&&\big(D^{(4)},\nu^{(\cdot(\cdot(\cdot\cdot)))}\big)\\};
\draw[->] (m-1-1) edge node {$(\Phi^{1,2,3}\cdot)$} (m-1-2)
(m-1-2) edge node {$(\Phi^{1,23,4}\cdot)$} (m-1-3)
(m-1-1) edge node {$(\Phi^{12,3,4}\cdot)$} (m-2-1)
(m-1-3) edge node {$(\Phi^{1,23,4}\cdot)$} (m-2-3)
(m-2-1) edge node {$(\Phi^{1,2,34}\cdot)$} (m-2-3);
\end{tikzpicture}\end{equation}
(the pentagon identity). In fact, this is the case since on the level of objects \eqref{eq:CohAssPent} is just the usual pentagon identity for the Drinfel'd associator $\Phi$.

Therefore, by Mac Lane's Coherence theorem, it follows that for any two parenthesizations of the same underlying ordered morphism $p,p':I\to J$, there is a canonical `higher associator' $\Phi^{p,p'}$ (which can be written as an iterated composite of the Drinfel'd associator, see \cite{MacLane:1998tv} for details) defining a natural isomorphism 
$$(F^{p},\nu^p)\xrightarrow{\Phi^{p,p'}\cdot } (F^{p'},\nu^{p'}),$$
such that for any three parenthesizations of the same underlying ordered morphism $p,p',p'':I\to J$, the following diagram commutes:
$$\begin{tikzpicture}
\mmat{m}{
&(F^{p},\nu^p)&\\
(F^{p'},\nu^{p'})&&(F^{p''},\nu^{p''})\\};
\draw[->,dashed]
(m-1-2) edge node [swap] {$\Phi^{p,p'}$} (m-2-1)
(m-1-2) edge node  {$\Phi^{p,p''}$} (m-2-3)
(m-2-1) edge node [swap] {$\Phi^{p',p''}$} (m-2-3);
\end{tikzpicture}$$

In particular, this defines \eqref{eq:QuantFus2Funct} on the 2-morphisms and shows that it is compatible with vertical composition.

The horizontal composition of monoidal natural transformations is just their Godement product, and the higher associators $\Phi^{p,p'}$ are compatible with the Godement product (see \cite[Lemma~3.5]{BarNatan:1998he}, for example, where this fact is expressed as `locality in scale').
\end{proof}

\subsection{Reduction}
Suppose that $\g$ and $\h$ are Lie algebras with chosen elements $t_\g\in(S^2\g)^\g$, $t_\h\in(S^2\h)^\h$, and $\mf{c}\subseteq \g$ is a coisotropic Lie subalgebra. As shown in
\cite[Proposition~2]{LiBland:2014da}, the functor
$$\mf{c}\oplus\mf{c}\text{-invariants}:U(\g\oplus\g\oplus\h)\Mod^\Phi_\hbar\to U(\h)\Mod^\Phi_\hbar,$$
is monoidal, where the coherence map is the natural inclusion, and the unit coherence map is the identity. As shown in
  \cite[Proposition~2,Theorem~2]{LiBland:2014da} there is a second monoidal functor $$\mf{c}\text{-invariants}\circ F:U(\g\oplus\g\oplus\h)\Mod^\Phi_\hbar\to U(\h)\Mod^\Phi_\hbar,$$
where $F:U(\g\oplus\g\oplus\h)\Mod^\Phi_\hbar\to U(\g\oplus\h)\Mod^\Phi_\hbar$ is the \emph{(quantum) fusion} functor induced by the diagonal inclusion  $\g\oplus\h\to \g\oplus\g\oplus \h$ (cf. \cite[Theorem~2]{LiBland:2014da}), and the coherence map is given by the action of the universal element $J\in \bigl(U(\g\oplus\g)\bigr)^{\otimes2}[\![\hbar]\!]$, followed by the natural inclusion. (The unit coherence map is the identity, as before.)

\begin{Theorem}\label{thm:RedToRedFus}
The natural inclusion $$X^{\mf{c}\oplus\mf{c}}\to X^{\mf{c}_\Delta}=F(X)^\mf{c}$$
(here $\mf{c}_\Delta\subseteq \mf{c}\oplus\mf{c}$ denotes the diagonal),
 defines a monoidal natural transformation $$\mf{c}\oplus\mf{c}\text{-invariants}\dasharrow\mf{c}\text{-invariants}\circ F.$$ 
\end{Theorem}
\begin{proof}
Recall that the monoidal products on $U(\g\oplus\g\oplus\h)\Mod^\Phi_\hbar$ and $U(\h)\Mod^\Phi_\hbar$ are the usual ones (only the braiding and associativity are deformed).
First, we must show that, for any objects $X,Y\in U(\g\oplus\g\oplus\h)\Mod^\Phi_\hbar$, the monoidal product coherence diagram
$$\begin{tikzpicture}
\mmat{m}{
X^{\mf{c}\oplus\mf{c}}\otimes Y^{\mf{c}\oplus\mf{c}}& X^{\mf{c}_\Delta}\otimes X^{\mf{c}_\Delta}\\
(X\otimes Y)^{\mf{c}\oplus\mf{c}}&(X\otimes Y)^{\mf{c}_\Delta}\\};
\draw[->]
(m-1-1) edge (m-1-2)
(m-1-1) edge (m-2-1)
(m-1-2) edge node {$(J\cdot)$} (m-2-2)
(m-2-1) edge (m-2-2);
\end{tikzpicture}$$
commutes, where the undecorated arrows are just the natural inclusions, and the rightmost arrow denotes the action of $J\in (U\g)^{\otimes4}[\![\hbar]\!]=\bigl(U(\g\oplus\g)\bigr)^{\otimes2}[\![\hbar]\!]$, followed by the natural inclusion.

Recall that $J=J(\hbar t_\g^{1,2},\hbar t_\g^{1,3},\hbar t_\g^{1,4},\hbar t_\g^{2,3},\hbar t_\g^{2,4},\hbar t_\g^{3,4})$ is given by a formal series in the variables $\hbar t_\g^{i,j}$ (where $i,j\in\{1,2,3,4\}$ range over distinct pairs). Since $\mf{c}$ is coisotropic, the image of $t_\g\in \g\otimes\g$ in $(\g/\mf{c})\otimes (\g/\mf{c})$ is trivial. Thus, each of $\hbar t_\g^{i,j}$ (where $i,j\in\{1,2,3,4\}$ denote any distinct pair of elements) acts trivially on $X^{\mf{c}\oplus\mf{c}}\otimes Y^{\mf{c}\oplus\mf{c}}$, which implies the commutativity of the diagram above.

The fact that the monoidal unit coherence diagram,
$$\begin{tikzpicture}
\mmat{m}{
&\mbb{K}&\\
\mbb{K}&&\mbb{K}\\};
\draw[->]
(m-1-2) edge (m-2-1)
(m-1-2) edge node {$(J\cdot)$} (m-2-3)
(m-2-1) edge (m-2-3);
\end{tikzpicture}$$
(here the unlabelled arrows denote the identity map)
commutes is obvious, since $J$ acts trivially on $\mbb{K}$.
\end{proof}

%
%

\section{Coloured surfaces}\label{sec:ColSurf}
\subsection{(Geometric) Coloured Surfaces and Skeleta}\label{sec:Skeleta}
It will be useful to have a more combinatorial description of the coloured surfaces described in \S~\ref{sec:GeomColSurf}. In order to do this, 
suppose that the set of domain walls has been partitioned $\mbf{Walls}=\mbf{Walls}^-\sqcup \mbf{Walls}^+$ into `positive' and `negative' domain walls. A \emph{bone} for the coloured surface $(\Sigma,\mbf{Walls},C_\cdot)$ is an embedded oriented 1-dimensional connected manifold with non-empty boundary $\partial\gamma=\gamma_-\sqcup\gamma_+$ whose source $\gamma_-$ lies on a negative domain wall, whose target $\gamma_+$ lies on a positive domain wall, and whose interior is mapped into the interior of $\Sigma$.
\begin{figure}[h!]
\begin{center}
\begin{subfigure}{.4\textwidth}
\centering
	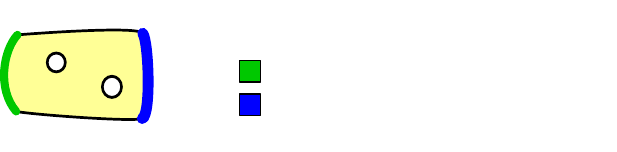
	\caption{ An example of a (geometric) coloured surface.}
\end{subfigure}
\begin{subfigure}{.4\textwidth}
\centering
	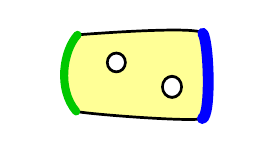
	\caption{  A skeleton.}
	\end{subfigure}
\end{center}
\caption{\label{fig:ColSurfAndSkel}}
\end{figure}
A skeleton of the coloured surface $(\Sigma,\mbf{Walls},C_\cdot)$ is a collection, $\mbf{Bones}$, of  mutually disjoint bones such that the surface $\Sigma$ deformation retracts onto the union $$\bigcup_{\gamma\in\mbf{Bones}\cup\mbf{Walls}}\gamma$$ of bones and domain walls.
\begin{figure}[h!]
\begin{center}
	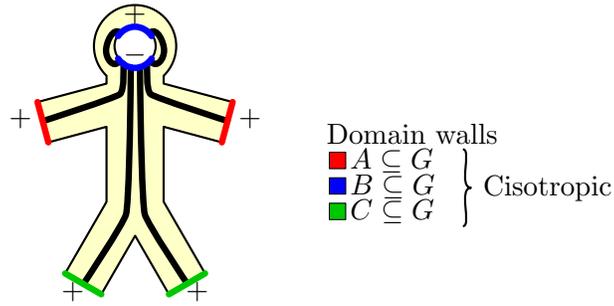
	\caption{\label{fig:SkelMan} Another example of a (geometric) coloured surface.}
\end{center}
\end{figure}

\begin{remark}
If each component of $\Sigma$ contains both a positive and negative domain wall, then a skeleton for $\Sigma$ can always be chosen. If each component of $\Sigma$ has a non-trivial boundary, then one can always add extra positive/negative domain walls (coloured by $C=G$) without changing the moduli spaces associated to the coloured surface in \cite{LiBland:2012vo}. In this way, skeletons can always be chosen for any component of $\Sigma$ with a non-trivial boundary.
\end{remark}

A skeleton for a coloured surface $(\Sigma,\mbf{Walls},C_\cdot)$ naturally defines a bipartite graph $\Gamma$ whose edges are the bones, $\mbf{E}_\Gamma:=\mbf{Bones}$, and whose vertices are the domain walls, $\mbf{V}_\Gamma:=\mbf{Walls}$, with the incidence maps defined in the obvious way. For example, in Figure~\ref{fig:SomeGraphs} we have drawn the graphs corresponding to the skeleta pictured in Figures~\ref{fig:ColSurfAndSkel} and \ref{fig:SkelMan}.

\begin{figure}[h!]
\begin{center}
\begin{subfigure}{.4\textwidth}
\centering
	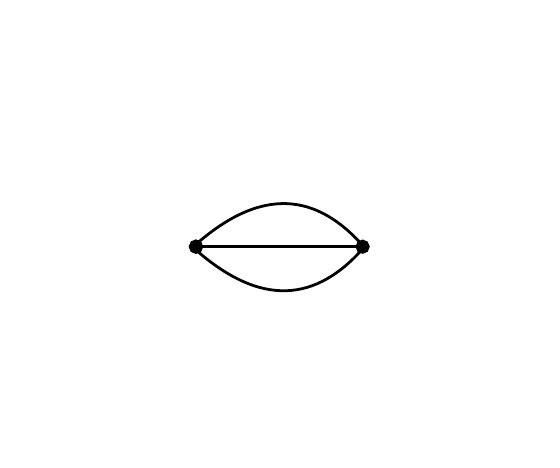 
\caption{The graph corresponding to Figure~\ref{fig:ColSurfAndSkel}.}
\end{subfigure}
\begin{subfigure}{.4\textwidth}
\centering
	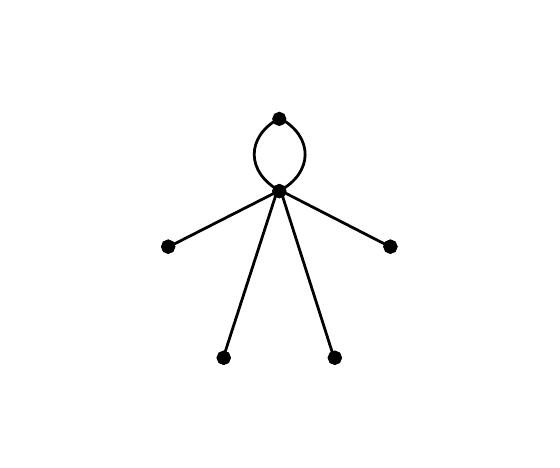 
\caption{The graph corresponding to Figure~\ref{fig:SkelMan}.}
\end{subfigure}
\end{center}
\caption{\label{fig:SomeGraphs}}
\end{figure}

While the graph constructed in this way describes the homotopy type of the surface, we cannot fully recover the surface from the bare graph. For example, notice that each domain wall $\mbf{w}\in \mbf{Walls}$ inherits the boundary orientation from $\partial\Sigma\subset\Sigma$, so the bones incident to $\mbf{w}$ inherit a linear ordering; i.e. the set of edges incident to any vertex in the graph are linearly ordered. As it turns out, in order to recover the surface from the graph it suffices to remember these linear orders (cf. \cite{Fock:1999wz}).

\subsection{Bipartite Graphs with Half-Edges}
Let $\mc{W}$ denote the category with the following objects and (non-identity) morphisms:
$$\begin{tikzpicture}
	\node (v-) at (-2,-1) {$\mbf{V}^-$};
	\node (v+) at (2,-1) {$\mbf{V}^+$};
	\node (h-) at (-2,1) {$\mbf{H}^-$};
	\node (h+) at (2,1) {$\mbf{H}^+$};
	\node (e) at (0,2) {$\mbf{E}$};
	\draw[->] (h-) -- node[swap] {$\mbf{i}_-$} (v-);
	\draw[->] (e) -- node [swap] {$\mbf{h}_-$} (h-);
	\draw[->] (h+) -- node {$\mbf{i}_+$} (v+);
	\draw[->] (e) -- node {$\mbf{h}_+$} (h+);
\end{tikzpicture}$$
A bipartite graph with half edges is a functor $\Gamma:\mc{W}\to\FSet$ which sends $\mbf{h}_\pm$ to injective maps and $\mbf{i}_\pm$ to surjective maps, i.e. it consists of two finite sets of \emph{negative} (resp. \emph{positive}) \emph{vertices}, which we denote with the following shorthand:
$$\mbf{V}^-_\Gamma:=\Gamma(\mbf{V}^-),\quad \mbf{V}^+_\Gamma:=\Gamma(\mbf{V}^+),$$
two finite sets of \emph{negative} (resp. \emph{positive}) \emph{half edges}, denoted with analogous shorthand:
$$\mbf{H}^-_\Gamma:=\Gamma(\mbf{H}^-),\quad \mbf{H}^+_\Gamma:=\Gamma(\mbf{H}^+),$$
one finite set of \emph{edges}:
$$\mbf{E}_\Gamma:=\Gamma(\mbf{E}),$$
 surjective \emph{incidence} maps:
  $$\mbf{i}_\pm:\mbf{H}^\pm_\Gamma\to \mbf{V}^\pm_\Gamma,$$
 and injective maps:
 $$\mbf{h}_\pm:\mbf{E}_\Gamma\to \mbf{H}^\pm_\Gamma,$$
 sending a given edge to its initial and final half (one can think of $\mbf{h}_+\circ\mbf{h}_-^{-1}$ as pairing the subset $\mbf{h}_-(\mbf{E})$ of negative half edges with the subset $\mbf{h}_+(\mbf{E})$ of positive half edges). We call the complement $H^{\pm,w}_\Gamma:=\mbf{H}^\pm_\Gamma\setminus \mbf{h}_\pm\big(\mbf{E}_\Gamma\big)$ the set of (positive/negative) \emph{widowed} half edges.

A morphism $\Gamma\to \Gamma'$ between bipartite graphs with half edges is a natural transformation between the respective functors.

A lift of the functor $\Gamma:\mc{W}\to\FSet$ through the category of ordered morphisms,
$$\begin{tikzpicture}
	\mmat{m}{&\FSet^<\\\mc{W}&\FSet\\};
	\draw[dotted,->] (m-2-1) -- node {$\Gamma^<$} (m-1-2);
	\draw[->] (m-2-1) -- node {$\Gamma$} (m-2-2);
	\draw[->] (m-1-2) --  (m-2-2);
\end{tikzpicture}$$
is said to \emph{equip $\Gamma$ with a linear ordering on the incidence sets}, and we call any functor $\Gamma^<:\mc{W}\to \FSet^<$ over a bipartite graph with half edges a \emph{ciliated graph} (note that this terminology differs somewhat from that of \cite{Fock:1999wz}). Morphisms between ciliated graphs are just natural transformations between the respective functors.

\subsection{(Combinatorial) Skeletized Coloured Surfaces}

Throughout this section we let $\g$ be a Lie algebra with a chosen invariant element $t\in (S^2\g)^\g$. By an \emph{algebraic group over $\mbb{K}$}, we mean an affine group scheme of finite type over $\mbb{K}$; in particular, since $\mbb{K}$ is of characteristic zero, an algebraic group over $\mbb{K}$ is smooth.

An algebraic group $C$ over $\mbb{K}$, equipped with an inclusion of its Lie algebra $\mf{c}\hookrightarrow \g$ is called $\g$-\emph{coisotropic} if $t\in S^2(\g)$ is $C$-invariant and the restriction of $t$ to $\on{ann}(\mf{c})\subseteq \g^*$ vanishes. Note that this implies that $C$ is also $\bar\g$-coisotropic.

Suppose that an algebraic group $H$ comes equipped with a map of Lie algebras $\h\to\g$, we say that a $\K$-scheme, $X$ carries an $(H,\g)$-action if both $H$ and $\g$ act on $X$, and the infinitesimal $\h$ action on $X$ factors through the $\g$ action. We say that the action has \emph{coisotropic stabilizers} if the composite map $$\Omega_{X/\K}\xrightarrow{\rho^*}\g^*\otimes\mc{O}_X\xrightarrow{t^\sharp\otimes\id}\g\otimes\mc{O}_X\xrightarrow{\rho}\Theta_X$$ is zero, where $\Omega_{X/\K}$ is the cotangent sheaf, $\Theta_X=\Hom_{\mc{O}_X}(\Omega_{X/\K},\mc{O}_X)$ is the tangent sheaf, and $\rho:\g\otimes\mc{O}_X\to \Theta_{X}$ is the action map.

As in \cite{LiBland:2014da}, we say that a commutative associative algebra $A\in U(\g)\Mod$ is \emph{$\g$-quasi-Poisson commutative} if $m\circ (t^{1,2}\cdot)=0$, where $m:A\otimes A\to A$ is the multiplication. Notice that $A$ is $\g$-quasi-Poisson-commutative only if it is $\bar\g$-quasi-Poisson-commutative. Moreover, \cite[Proposition~1]{LiBland:2014da} implies that $A$ with its original product is a commutative associative algebra in $U(\g)\Mod^\Phi_\hbar$. It follows from \cite[Proposition~1]{LiBland:2014da} that if $\g$ acts on a $\K$-scheme, $X$, with coisotropic stabilizers, then the structure ring $\mc{O}_X(U)$ is $\g$-quasi-Poisson-commutative for any open set $U\subseteq X$.

 
\begin{Definition}\label{def:SkelColSurf}
A skeletized coloured surface is a triple $(\Gamma^<,C_\cdot,X_\cdot)$, consisting of 
\begin{itemize}
\item a ciliated graph  $\Gamma^<:\mc{W}\to \FSet^<$ 
\item an assignment of a $\g$-coisotropic algebraic group over $\mbb{K}$, $C_v$, to every vertex $v\in\mbf{V}_\Gamma=\mbf{V}_\Gamma^+\sqcup \mbf{V}_\Gamma^-$, 
\item an assignment of a scheme $X_e$ to each edge $e\in \mbf{E}_\Gamma$ on which  $(C_{\mbf{i}_+\circ \mbf{h}_+(e)}\times C_{\mbf{i}_-\circ \mbf{h}_-(e)},\g\oplus\bar\g)$ acts with coisotropic stabilizers.
\item an assignment of a scheme $X_h$ to each widowed half-edge $h\in \mbf{H}^{\pm,w}_\Gamma:=\mbf{H}^\pm_\Gamma\setminus \mbf{h}_\pm\big(\mbf{E}_\Gamma\big)$ on which $(C_{\mbf{i}_\pm(h)},\g)$ acts with coisotropic stabilizers.
\end{itemize}
 
 
 A morphism of skeletized coloured surfaces $$\phi:(\Gamma^<,C_\cdot,X_\cdot)\to (\Gamma'^<,C'_\cdot,X'_\cdot)$$
 is a morphism of the underlying ciliated graphs $\phi:\Gamma^<\to \Gamma'^<$, along with  morphisms
\begin{itemize}
 \item  $C_v\leftarrow C'_{\phi_{\mbf{V}}(v)}$, for any $v\in \mbf{V}_\Gamma$, 
 \item  equivariant morphisms $X_e\leftarrow X'_{\phi_{\mbf{E}}(e)}$, for any $e\in \mbf{E}_\Gamma$, and
 \item  equivariant morphisms $X_h\leftarrow X'_{\phi_{\mbf{H}^\pm}(h)}$, for any $h\in \mbf{H}_\Gamma^{\pm,w}$.
\end{itemize}
 
 We denote the corresponding category by $\sCol$.
\end{Definition}

Given a skeletized coloured surface $(\Gamma^<,\mf{c}_\cdot,X_\cdot)$, let \begin{subequations}\label{eq:XCMGamma}\begin{align}X_\Gamma&:=\prod_{e\in \mbf{E}_\Gamma\sqcup\mbf{H}^{\pm,w}_\Gamma} X_e,\text{ and}\\ C_\Gamma&:=\prod_{v\in \mbf{V}_\Gamma} C_v.\end{align} We define the ringed space $\mc{M}_\Gamma:=\mc{M}_{(\Gamma^<,C_\cdot,X_\cdot)}$, to be the quotient (in the category of commutative-ringed spaces\footnote{The category of commutative ringed spaces is a Grothendieck bi-fibration over the category of topological spaces, the projection being the functor $X=(\abs{X},\mc{O}_{X})\mapsto \abs{X}$ which forgets the sheaf of rings. Since the base is (co)complete, and so are each of the fibres, it follows that the total category of commutative ringed spaces is (co)complete. Moreover, any (co)limit can be computed by first taking the (co)limit in the base, and then computing the (co)limit of the (co)cartesian image of the diagram in the (co)limiting fibre. Note however, that colimits of commutative locally ringed spaces may not be locally ringed.}) \begin{equation}\label{eq:MGamma0}\mc{M}_\Gamma:= X_\Gamma/C_\Gamma=\bigg(\prod_{e\in \mbf{E}_\Gamma\sqcup \mbf{H}_\Gamma} X_e\bigg)/\bigg(\prod_{v\in \mbf{V}_\Gamma} C_v\bigg),\end{equation}\end{subequations}
where each $C_v$ (for $v\in \mbf{V}_\Gamma$) acts diagonally. Explicitly, the underlying topological space is the quotient space $\abs{\mc{M}_\Gamma}:=\abs{X_\Gamma}/\abs{C_\Gamma}$ and for any open set $U\subseteq \mc{M}_\Gamma$, the structure sheaf is $\mc{O}_{\mc{M}_\Gamma}(U)=\pi_*\mc{O}_{X_\Gamma}(U)^C=\mc{O}_{X_\Gamma}\big(\pi^{-1}(U)\big)^C$, where $\pi:\abs{X_\Gamma}\to \abs{\mc{M}_\Gamma}$ is the quotient map. In particular, \eqref{eq:MGamma0} is the geometric quotient of schemes (if it exists).

This construction defines a functor \begin{equation}\label{eq:ClassModSpace}\sCol\xrightarrow{\Gamma\mapsto\mc{M}_\Gamma} \Top_{\K\CAlg}\end{equation} from (combinatorial) skeletized coloured surfaces to comutatively ringed spaces, which we call the \emph{classical moduli space functor} (see \cite{LiBland:2012vo} for the terminology).

\begin{remark}\label{rem:GeomColSurf}
As mentioned in \S~\ref{sec:Skeleta}, any skeleton for a coloured surface $(\Sigma,\mbf{Walls},C_\cdot)$ defines a skeletized coloured surface $(\Gamma^<,C_\cdot,X_\cdot)$ (in the sense of Definition~\ref{def:SkelColSurf}), where the graph $\Gamma$ has edge set $\mbf{E}_\Gamma:=\mbf{Bones}$, vertex set $\mbf{V}_\Gamma:=\mbf{Walls}$, and half edge sets $\mbf{H}^\pm_\Gamma:=\mbf{E}_\Gamma:=\mbf{Bones}$. The maps $\mbf{h}_\pm:\mbf{E}_\Gamma\to \mbf{H}^\pm_\Gamma$ are the identity maps, and for any half edge $\gamma\in \mbf{H}^\pm_\Gamma:=\mbf{Bones}$, the incidence maps are defined as 
\begin{align*}\mbf{i}_-(\gamma)&=\mbf{w}_i\Leftrightarrow \gamma_-\in\mbf{w}_i,\\
\mbf{i}_+(\gamma)&=\mbf{w}_i\Leftrightarrow \gamma_+\in\mbf{w}_i.	
 \end{align*}
 Each domain wall $\mbf{w}\in \mbf{Walls}$ inherits the boundary orientation from $\partial\Sigma\subset\Sigma$, so the bones incident to $\mbf{w}$ inherit a linear ordering.

Thus, we have a natural map
\begin{subequations}
\label{eq:CombFunc}
\begin{equation}\label{eq:SkelColSurfGeom}(\Sigma,\mbf{Walls},\mbf{Bones},C_\cdot)\mapsto(\Gamma^<,C_\cdot,X_\cdot)\end{equation}
 from coloured surfaces equipped with skeleta to (combinatorial) skeletized coloured surfaces (in the sense of Defintion~\ref{def:SkelColSurf}), defined by setting $X_e=G$ for every $e\in \mbf{E}_\Gamma$.\footnote{The Lie algebra $\g\oplus\bar\g$-acts on $G$ with coisotropic stabilizers via
$$(\xi,\eta)\cdot f:=(\eta^L-\xi^R)f,\quad f\in \mc{O}(G),\quad \xi,\eta\in \g\oplus\bar\g,$$
and $\xi^L$ (resp. $\xi^R$) denotes the left (resp. right) invariant vector field on $G$ corresponding to $\xi\in\g$.}  
 
 Recall from \S~\ref{sec:GeomColSurf} that a morphism between (geometric) coloured surfaces $(\Sigma,\mbf{Walls},C_\cdot)\to (\Sigma',\mbf{Walls}',C_\cdot')$ is defined
	to be an orientation preserving embedding of surfaces $\Sigma\to\Sigma'$ which maps domain walls coloured by a given coisotropic subgroup of $G$ into domain walls coloured by the same coisotropic subgroup. If both coloured surfaces are equipped with skeleta, we say that a morphism is compatible with the skeleta if it preserves the parity of the domain walls, and the image of any bone in $\Sigma$ is homotopic (relative the domain walls) to a bone in $\Sigma'$. We call the resulting category $\sCol^{geom}$ the category of \emph{geometric coloured surfaces equipped with skeleta}.
	
	It is clear that \eqref{eq:SkelColSurfGeom} extends to a \emph{combinatorialization} functor 
	\begin{equation}K:\sCol^{geom}\to \sCol\end{equation}
	\end{subequations} from the category  of (geometric) coloured surfaces equipped with skeleta to the category of (combinatorial) skeletized coloured surfaces (in the sense of Definition~\ref{def:SkelColSurf}). 
 
 Conversely, given a (combinatorial) skeletized coloured surface $(\Gamma^<,C_\cdot,X_\cdot)$, we have the following construction: Given an ordered morphism of finite sets, $p:I\to J$, let $\Sigma_p$ denote the surface obtained by blowing up a polygonal representation of $p$ at the corner of every polygon. We may equip $\Sigma_p$ with a choice of `half-bones' indexed by $I$: for each $i\in I$ a half bone is an embedded line segment connecting the midpoint the boundary segment of $\Sigma_p$ labelled by $i$ to the boundary segment labelled by $p(i)$; we assume that these half-bones a mutually disjoint. For example, if the ordered morphism is as pictured in \eqref{eq:ordMorphP}, then $\Sigma_p$ is
 \begin{center}
\begingroup%
  \makeatletter%
  \providecommand\color[2][]{%
    \errmessage{(Inkscape) Color is used for the text in Inkscape, but the package 'color.sty' is not loaded}%
    \renewcommand\color[2][]{}%
  }%
  \providecommand\transparent[1]{%
    \errmessage{(Inkscape) Transparency is used (non-zero) for the text in Inkscape, but the package 'transparent.sty' is not loaded}%
    \renewcommand\transparent[1]{}%
  }%
  \providecommand\rotatebox[2]{#2}%
  \ifx\svgwidth\undefined%
    \setlength{\unitlength}{328.86054688bp}%
    \ifx\svgscale\undefined%
      \relax%
    \else%
      \setlength{\unitlength}{\unitlength * \real{\svgscale}}%
    \fi%
  \else%
    \setlength{\unitlength}{\svgwidth}%
  \fi%
  \global\let\svgwidth\undefined%
  \global\let\svgscale\undefined%
  \makeatother%
  \begin{picture}(1,0.36567722)%
    \put(0.5286453,0.02935089){\color[rgb]{0,0,0}\makebox(0,0)[b]{\smash{$\Sigma_p$}}}%
    \put(0,0){\includegraphics[width=\unitlength,page=1]{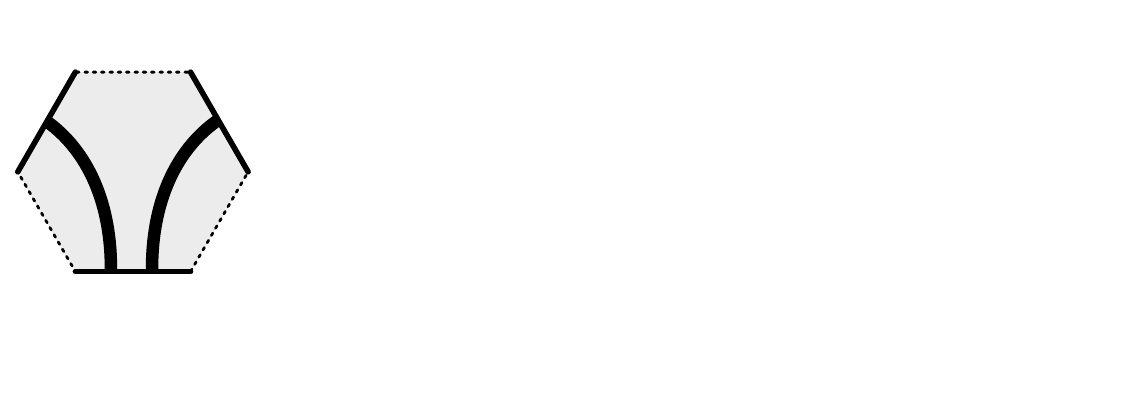}}%
    \put(0.11439048,0.10154516){\color[rgb]{0,0,0}\makebox(0,0)[b]{\smash{$a$}}}%
    \put(0.02595372,0.25858435){\color[rgb]{0,0,0}\makebox(0,0)[rb]{\smash{$2_a$}}}%
    \put(0.20332775,0.25858435){\color[rgb]{0,0,0}\makebox(0,0)[lb]{\smash{$1_a$}}}%
    \put(0,0){\includegraphics[width=\unitlength,page=2]{SigmaP.pdf}}%
    \put(0.87776129,0.09882777){\color[rgb]{0,0,0}\makebox(0,0)[b]{\smash{$c$}}}%
    \put(0.87776129,0.32263086){\color[rgb]{0,0,0}\makebox(0,0)[b]{\smash{$1_c$}}}%
    \put(0,0){\includegraphics[width=\unitlength,page=3]{SigmaP.pdf}}%
    \put(0.52979362,0.08433504){\color[rgb]{0,0,0}\makebox(0,0)[b]{\smash{$b$}}}%
    \put(0.42267148,0.20242493){\color[rgb]{0,0,0}\makebox(0,0)[rb]{\smash{$3_b$}}}%
    \put(0.64121921,0.20242493){\color[rgb]{0,0,0}\makebox(0,0)[lb]{\smash{$1_b$}}}%
    \put(0.53443827,0.32708863){\color[rgb]{0,0,0}\makebox(0,0)[b]{\smash{$2_b$}}}%
    \put(0,0){\includegraphics[width=\unitlength,page=4]{SigmaP.pdf}}%
  \end{picture}%
\endgroup%
	
 \end{center}

Gluing the surfaces $\Sigma_{\Gamma(\mbf{i}_+)}$ and $\Sigma_{\overline{\Gamma(\mbf{i}_+)}}$ along the boundary segments $\mbf{h}_+(e)$ and $\mbf{h}_-(e)$ (respectively) for each $e\in E_\Gamma$ results in an oriented surface, which we denote by $\Sigma_{\Gamma^<}$. For each vertex $v\in V_\Gamma^\pm$, we may interpret the boundary segment labelled by $v$ as a domain wall coloured by $C_v$; and for each edge $e\in E_\Gamma$, the half-bones corresponding to $\mbf{h}_-(e)$ and $\mbf{h}_+(e)$ concatenate to define a bone which connects the domain walls $\mbf{i}_-\circ\mbf{h}_-(e)$ and $\mbf{i}_+\circ\mbf{h}_+(e)$. In particular, if there are no widowed half edges, and $X_e=G$ for each $e\in E_\Gamma$, it follows that $(\Sigma_{\Gamma^<},\mbf{Walls}^\pm\cong V_\Gamma^\pm, \mbf{Bones}\cong E_\Gamma, C_\cdot)$ is a (geometric) skeletized coloured surface.

If there are widowed half edges, but $X_e=G$ for each $e\in E_\Gamma$, we may still represent the (combinatorial) skeletized coloured surface $(\Gamma^<,C_\cdot,X_\cdot)$ by colouring the boundary segments of $\Sigma_{\Gamma^<}$ corresponding to the vertices appropriately, and for each widowed half edge $h\in H_\Gamma^{\pm,w}$, we will draw a node at the free tip of the corresponding half-bone and label the node by $X_h$. For example, suppose $\Gamma^<$ is the graph with no negative half edges (and hence no edges or negative vertices), two positive half edges $(1,2)$, and one positive vertex $v$, so that $\Sigma_{\Gamma^<}$ is as in Figure~\ref{fig:SigG1}
\begin{figure}[h!]
\begin{center}
\begin{subfigure}{.4\textwidth}
\centering
	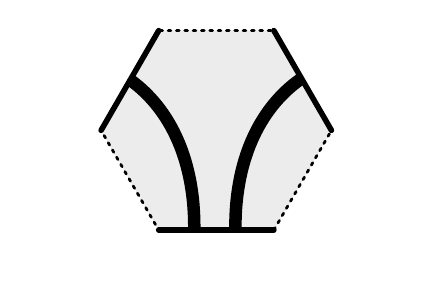 
\caption{\label{fig:SigG1}}
\end{subfigure}
\begin{subfigure}{.4\textwidth}
\centering
	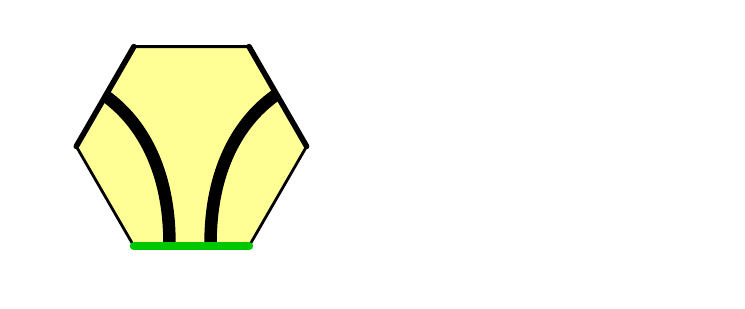 
\caption{\label{fig:SigG2}A pictorial representation of a (combinatorial) skeletized coloured surface which has widowed half-edges.}
\end{subfigure}
\end{center}
\caption{}
\end{figure}
 then we would represent $(\Gamma^<,C_\cdot,X_\cdot)$ as in Figure~\ref{fig:SigG2}. As before, morphisms of such (combinatorial) skeletized coloured surfaces correspond to appropriate embeddings of the corresponding (geometric) coloured surfaces.

\end{remark}

\subsection{The ringed space $\mc{M}_\Gamma^\hbar$  associated to a skeletized coloured surface}\label{sec:theAlgebra}

We let $\K\Alg_\hbar$ denote the category of associative deformations of commutative $\K$-algebras, formally parametrized by $\hbar$. Explicitly, it is the full subcategory of associative algebras over $\mbb{K}[\![\hbar]\!]$ which is spanned by algebras which are topologically free as modules over $\K[\![\hbar]\!]$,\footnote{That is, they are separated, complete, and torsion free.} and which become commutative at $\hbar=0$. The topological tensor product equips $\K\Alg_\hbar$ with a monoidal structure. Moreover, the forgetful functor $$\K\Alg_\hbar\xrightarrow{A_\hbar\to A_{\hbar=0}}\K\CAlg$$ is strictly monoidal. We let $\Top_{\K\Alg_\hbar}$ denote the category of topological spaces equipped with sheaves in $\K\Alg_\hbar$, it inherits a monoidal structure from the cartesian product (on the level of spaces), and the topological tensor product (on the level of sheaves of rings). Notice that the forgetful functor 
\begin{equation*}\Top_{\K\Alg_\hbar}\xhookrightarrow{X^\hbar\mapsto X^{\hbar=0}} \Top_{\K\CAlg}\end{equation*}
 is  monoidal.

	For any $X^\hbar\in \Top_{\K\Alg_\hbar}$, the underlying commutative ringed space $X:=X^{\hbar=0}$ carries a natural Poisson bracket: if $f,g\in \mc{O}_{X}(U)$ (where $U\subseteq X$ is an open set), then 
	\begin{subequations}\label{eq:ForgToPoiss}\begin{equation}\label{eq:PBrack}\{f,g\}:=\frac{\tilde f \tilde g-\tilde g\tilde f}{\hbar}\pmod\hbar,\end{equation}
	 where $\tilde f,\tilde g\in \mc{O}_{X^{\hbar}}(U)$ are any lifts of $f,g$. We say that $X\xhookrightarrow{\hbar=0} X^\hbar$ is a deformation quantization (of this Poisson structure on $X$). We let $\K\CPoiss$ denote the category of commutative Poisson algebras over $\K$, and $\Top_{\K\CPoiss}$ denote the category of Poisson ringed spaces, i.e. spaces equipped with a $\K\CPoiss$-valued presheaf of functions, whose underling $\K\CAlg$-valued presheaf is actually a sheaf. Then \eqref{eq:PBrack} defines a functor:
\begin{equation}\Top_{\K\Alg_\hbar}\xhookrightarrow{X^\hbar\mapsto (X^{\hbar=0},\{\cdot,\cdot\})} \Top_{\K\CPoiss}\end{equation}
	\end{subequations}

In this section, we construct a lift $\Gamma\to \mc{M}_\Gamma^\hbar$ of the classical moduli space functor \eqref{eq:ClassModSpace}, which we call the \emph{quantized moduli space} functor:
\begin{equation}\label{eq:QuantModuli}\begin{tikzpicture}
	\node (skg) at (-4,0) {$\sCol^{geom}$};
	\node (sk) at (0,0) {$\sCol$};
	\node (kh) at (6,2) {$\Top_{\K\Alg_\hbar}$};
	\node (k) at (6,0) {$\Top_{\K\CAlg}$};
	\draw[->] (sk) -- node[swap] {$\Gamma\to \mc{M}_\Gamma$} (k);
	\draw[->] (kh) -- (k);
	\draw[->,dashed] (sk) -- node {$\Gamma\to \mc{M}_{\Gamma}^\hbar$} (kh);
	\draw[->] (skg) -- (sk);
\end{tikzpicture}\end{equation}

We begin by noting that since $\FSet^<$ and $\FSet^{(<)}$ are biequivalent as 2-categories, we may assume that every ordered morphism is canonically and coherently equipped with a parenthesization (for example, we may choose the standard parenthesization). In particular, the strict (quantum) fusion  2-functor induces a pseudo-functor from $\FSet^<$ to $\MCat$ which we also denote by $F^{(-)}$, and which is canonically defined up to a unique pseudo-natural isomorphism. As a result, we are at liberty to implicitly assume that parenthesizations have been coherently\footnote{That is, up to canonical 2-morphisms.} chosen in what follows.

Suppose that $(\Gamma^<,C_\cdot,X_\cdot)$ is a skeletized coloured surface, and let $X_\Gamma$, $C_\Gamma$ and $\mc{M}_\Gamma$ be as in \eqref{eq:XCMGamma}. 


First, let $C^\circ_\Gamma\subseteq C_\Gamma$ denote the identity component and define $\mc{M}^\circ_\Gamma= X_\Gamma/C^\circ_\Gamma$  (in the category of commutative-ringed spaces) with the quotient map ${\pi_\circ}:X_\Gamma\to \mc{M}^\circ_\Gamma$.

Notice that for any open set $V\subseteq X_\Gamma$, the ring $\mc{O}_{X_\Gamma}(V)$ is a $\g^{\mbf{H}^+_\Gamma}\oplus\bar\g^{\mbf{H}^-_\Gamma}$-quasi-Poisson-commutative algebra. In particular, for $V={\pi_\circ}^{-1}(W)$ (where $W\subseteq \abs{\mc{M}_\Gamma^\circ}$ is open), 
$$\mc{O}_{X_\Gamma}(V)\in U\big(\g^{\mbf{H}^+_\Gamma}\oplus\bar\g^{\mbf{H}^-_\Gamma}\big)\Mod^\Phi_\hbar$$ is a commutative associative algebra, where the product is the original (undeformed) product (cf. \cite[Proposition~1]{LiBland:2014da}).

 Fusing the $\g^{\mbf{H}^+_\Gamma}$ factors along $\mbf{i}_+:\mbf{H}_\Gamma^+\to \mbf{V}_\Gamma^+$ and the $\bar\g^{\mbf{H}^-_\Gamma}$ factors along $\mbf{i}_-:\mbf{H}_\Gamma^-\to \mbf{V}_\Gamma^-$ yields an associative algebra $$F^{\mbf{i}_+}\circ F^{\mbf{i}_-}(\mc{O}_{X_\Gamma}(V))\in U\big(\g^{\mbf{V}_\Gamma^+}\oplus\bar\g^{\mbf{V}_\Gamma^-}\big)\Mod^\Phi_\hbar$$
Explicitly, the multiplication on $F^{\mbf{i}_+}\circ F^{\mbf{i}_-}(\mc{O}_{X_\Gamma}(V))=F^{\mbf{i}_+\sqcup\mbf{i}_-}(\mc{O}_{X_\Gamma}(V))$ is \begin{equation}\label{eq:multGamma}m_\Gamma^\hbar:=\big(m_\Gamma\big)\circ( \nu^{\mbf{i}_+}\otimes  \nu^{\mbf{i}_-}),\end{equation} where $m_\Gamma:\mc{O}_{X_\Gamma}(V)\otimes \mc{O}_{X_\Gamma}(V)\to \mc{O}_{X_\Gamma}(V)$ is the original multiplication on $\mc{O}_{X_\Gamma}(V)$, and $\nu^{\mbf{i}_\pm}$ are as in \eqref{eq:qFusionForp}.

Now the Lie algebra $$\mf{c}_\Gamma:=\{(\xi_\cdot)\in \g^{\mbf{V}_\Gamma}\mid \xi_v\in \mf{c}_v\}\subseteq \g^{\mbf{V}_\Gamma^+}\oplus \bar\g^{\mbf{V}_\Gamma^-}$$
of $C_\Gamma$ is a coisotropic subalgebra. Thus, the $\mf{c}_\Gamma$ invariant subspace $$\mf{c}_\Gamma\text{-invariants}\bigg(F^{\mbf{i}_+}\circ F^{\mbf{i}_-}\big(\mc{O}_{X_\Gamma}(V)\big)\bigg)$$
is an associative algebra (in the category of vector spaces), whose multiplication depends on the formal parameter $\hbar$. We define
\begin{subequations}\label{eq:AlgFormula}\begin{equation}\mc{O}_{\mc{M}_{\Gamma}^{\hbar,\circ}}(W):=\mf{c}_\Gamma\text{-invariants}\bigg(F_\g^{\mbf{i}_+}\circ F_{\bar\g}^{\mbf{i}_-}\big(\mc{O}_{X_\Gamma}({\pi_\circ}^{-1}(W))\big)\bigg)[\![\hbar]\!]\end{equation} by extending the multiplication to be $\hbar$-linear.
As a $\mbb{K}[\![\hbar]\!]$-module, \begin{equation}\label{eq:PreSheafVect}\mc{O}_{\mc{M}_{\Gamma}^{\hbar,\circ}}(W)= \big(\mc{O}_{X_\Gamma}({\pi_\circ}^{-1}(W))\big)^{\mf{c}_\Gamma}[\![\hbar]\!]=\mc{O}_{\mc{M}_\Gamma^\circ}(W)[\![\hbar]\!],\end{equation}\end{subequations} where $\mf{c}_\Gamma$ acts via $(\mbf{i}_+^!\oplus\mbf{i}_-^!):\g^{\mbf{V}_\Gamma^+\sqcup \mbf{V}_\Gamma^-}\to  \g^{\mbf{H}^+_\Gamma\sqcup \mbf{E}_\Gamma}\oplus\bar\g^{\mbf{H}^-_\Gamma\sqcup \mbf{E}_\Gamma}$, and the multiplication is given by \eqref{eq:multGamma}.\footnote{Since the $\nu^{\mbf{i}_\pm}$ appearing in \eqref{eq:multGamma} depends explicitly on the choice of parenthesization on $\mbf{i}_\pm:\mbf{H}_\Gamma^\pm\to \mbf{V}_\Gamma^\pm$, so does the algebra \eqref{eq:AlgFormula}. However, Theorem~\ref{thm:CohAss} as well as the biequivalence between $\FSet^{(<)}$ and $\FSet^<$, means that for any two choices of parenthesization, the resulting algebras are canoncially isomorphic as $\mbb{K}[\![\hbar]\!]$ algebras. Since parenthesizations can be chosen canonically (for example, one may always take the standard parenthesization), one may interpret this as follows: The algebra $\mc{O}_{\mc{M}_{\Gamma}^{\hbar,\circ}}(W)$ is canonically defined as a $\K[\![\hbar]\!]$ algebra, and any choice of parenthesization of $\mbf{i}_\pm$ defines a trivialization $\mc{O}_{\mc{M}_{\Gamma}^{\hbar,\circ}}(W)\cong \mc{O}_{\mc{M}_\Gamma^\circ}(W)[\![\hbar]\!]$ as a free $\mbb{K}[\![\hbar]\!]$ module.}


Now \eqref{eq:AlgFormula} depends naturally on the open subset $W\subseteq \abs{\mc{M}^\circ_\Gamma}$, so that $\mc{O}_{\mc{M}_{\Gamma}^{\hbar,\circ}}(W)$ defines a presheaf of associative algebras over $\abs{\mc{M}_\Gamma^\circ}$. Moreover the isomorphism \eqref{eq:PreSheafVect} implies that $\mc{O}_{\mc{M}_{\Gamma}^{\hbar,\circ}}$ is a sheaf of $\mbb{K}[\![\hbar]\!]$-modules, and hence also of algebras. Thus, $(\abs{\mc{M}_\Gamma^\circ}, \mc{O}_{\mc{M}_{\Gamma}^{\hbar,\circ}})$ is a (non-commutative) ringed space, flat over $\on{Spec}(\mbb{K}[\![\hbar]\!])$, whose fibre over $\hbar=0$ is the quotient $\mc{M}_\Gamma^\circ=X_\Gamma/C^\circ_\Gamma$. 

In the event that $C_\Gamma$ is not connected, there is a natural action of the discrete group $C_\Gamma/C^\circ_\Gamma$ on $(\abs{\mc{M}_\Gamma^\circ}, \mc{O}_{\mc{M}_{\Gamma}^{\hbar,\circ}})$, and we define $\mc{M}_\Gamma^\hbar$ as the corresponding quotient in the category of (possibly non-commutative) ringed spaces.

\begin{Theorem}\label{thm:AlgFunct}
The map $(\Gamma^<,C_\cdot,X_\cdot)\mapsto \mc{M}_\Gamma^\hbar$ extends to a contravariant monoidal functor 
$$\sCol\xrightarrow{\Gamma\to \mc{M}_\Gamma^\hbar} \Top_{\K\Alg_\hbar}$$
fitting into the diagram \eqref{eq:QuantModuli}, where the monoidal structure on $\sCol$ is given by disjoint union.
\end{Theorem}

	We start with a Lemma. Suppose that $p:I\to J$ is an ordered morphism of finite sets, for every $(i,j)\in I\times J$, $\g$ acts on given $\K$-schemes $Y_i,X_j$ with coisotropic stabilizers, and for every $i\in I$, $f_i:X_{p(i)}\to Y_i$ is a $\g$-equivariant morphism. 
	Let $X=\prod_{j\in J} X_j$ and $Y=\prod_{i\in I} Y_i$, and \begin{equation}\label{eq:fDiag}f:=\prod_{i\in I}f_i:X\xrightarrow{\{x_j\}_{j\in J}\mapsto\{y_i:=f_i(x_{p(i)})\}_{i\in I}} Y\end{equation} denote the diagonal morphism, and let $f^\sharp:\mc{O}_Y\to f_*\mc{O}_X$  denote the underlying morphism of the sheaves of rings.
	
	Theorem~\ref{thm:CohAss} implies that, for any parenthesization of the ordered morphism $p:I\to J$ and any open set $V\subseteq Y$, $F^{p}\big(\mc{O}_Y(V)\big)$ with the original multiplication precomposed with $\nu^p$ is an associative algebra in $U(\g^J)\Mod^\Phi_\hbar$. Similarly, for any open set $U\subseteq X$, $\mc{O}_X(U)$ (with the original multiplication) is also an associative algebra in $U(\g^J)\Mod^\Phi_\hbar$. In particular, $(\abs{X},\mc{O}_X)$ and $\big(\abs{Y},F^{p}\big(\mc{O}_Y\big)\big)$ are spaces equipped with sheaves of rings in $U(\g^J)\Mod^\Phi_\hbar$.
\begin{Lemma}\label{lem:DiagFusionMorph}
The map $(f,f^\sharp):(\abs{X},\mc{O}_X)\to \big(\abs{Y},F^{p}\big(\mc{O}_Y\big)\big)$ is a morphism of spaces equipped with sheaves of rings in $U(\g^J)\Mod^\Phi_\hbar$.
\end{Lemma}
\begin{proof}
Firstly, without loss of generality, we may assume that $I=\{1,\dots,n\}$, $J=\{\ast\}$ is the one point set, and $p:I\to J$ is equipped with the standard order on $I=\{1,\dots,n\}$.
Secondly, by factoring $f:X=X_\ast\to Y=Y_1\times\cdots\times Y_n$ as the diagonal map $X\to X^n$ followed by $f_1\times\cdots f_n: X^n\to Y$, we may also assume without loss of generality that $Y_i=X$ for every $i$.
Lastly, it is sufficient to suppose that $X=\on{Spec}(A_X)$ is affine. 

Now the $k$-fold diagonal map $X\to X^k$ is dual to the $k$-fold multiplication $m_{A_X}^{(k)}:A_X^{\otimes k}\to A_X$. 

Thus, we need to show that 
\begin{equation}\label{eq:CommMultToShow}m_{A_X}\circ(m_{A_X}^{(n)}\otimes m_{A_X}^{(n)})=m_{A_X}^{(n)}\circ (m_{A_X})^{\otimes n}\circ \nu^p.\end{equation}
We may simplify both sides of \eqref{eq:CommMultToShow} using the identity $$m_{A_X}\circ(m_{A_X}^{(n)}\otimes m_{A_X}^{(n)})=m_{A_X}^{(2n)}=m_{A_X}^{(n)}\circ (m_{A_X})^{\otimes n}.$$
Now, $m_{A_X}^{(2n)}\circ\nu^p=m_{A_X}^{(2n)}$ since $A_X$ is $\g$-quasi-Poisson commutative, which implies \eqref{eq:CommMultToShow}.

\end{proof}
\begin{proof}[Proof of Theorem~\ref{thm:AlgFunct}]
Let $\phi:(\Gamma^<,C_\cdot,X_\cdot)\to (\Gamma'^<,C'_\cdot,X_\cdot')$ be a morphism of skeletized coloured surfaces. For simplicity, we will assume that each of the algebraic groups $C_\cdot$ and $C'_\cdot$ is connected.  
When we need to distinguish them, we let $\mbf{i}_\pm^\Gamma$ and $\mbf{i}_\pm^{\Gamma'}$ denote the incidence maps for $\Gamma$ and $\Gamma'$ respectively.


Let $X_\Gamma:=\prod_{e\in \mbf{E}_\Gamma\sqcup \mbf{H}_\Gamma^{\pm,w}} X_e$,  and 
 $X'_{\Gamma'}:=\prod_{e\in \mbf{E}_{\Gamma'}\sqcup \mbf{H}_{\Gamma'}^{\pm,w}} X'_e$ be as in \eqref{eq:XCMGamma}. For any  $e\in \mbf{E}_{\Gamma}\sqcup \mbf{H}_{\Gamma}^{\pm,w}$, we have a map $X_e\leftarrow X_{\phi(e)}$, and 
 these fit together to define a morphism of ringed spaces $(\abs{X_\Gamma},\mc{O}_{X_\Gamma})\leftarrow (\abs{X'_{\Gamma'}},\mc{O}_{X'_{\Gamma'}}):(\tilde{f}_\phi,\tilde{f}_\phi^\sharp)$, as in \eqref{eq:fDiag}.

Notice that for any open set $V\subseteq \abs{X_\Gamma}$, 
$$\mc{O}_{X_\Gamma}(V)\in U\big(\g^{\mbf{H}^+_\Gamma}\oplus\bar\g^{\mbf{H}^-_\Gamma}\big)\Mod^\Phi_\hbar$$ is a commutative associative algebra. We let $F^{\phi_{\mbf{H}^\pm}}\big(\mc{O}_{X_\Gamma}(V)\big):=F^{\phi_{\mbf{H}^+}}\circ F^{\phi_{\mbf{H}^-}}\big(\mc{O}_{X_\Gamma}(V)\big)$ denote the result of taking the quantum fusion of the $\g^{\mbf{H}_\Gamma^+}$ factors along the ordered morphism $\phi_{\mbf{H}^+}:\mbf{H}_\Gamma^+\to\mbf{H}_{\Gamma'}^+$ and the $\bar\g^{\mbf{H}_\Gamma^-}$ factors along the ordered morphism $\phi_{\mbf{H}^-}:\mbf{H}_\Gamma^-\to\mbf{H}_{\Gamma'}^-$.

Now Lemma~\ref{lem:DiagFusionMorph} implies that we have a morphism 
$$\big(\abs{X_\Gamma},F^{\phi_{\mbf{H}^\pm}}(\mc{O}_{X_\Gamma})\big)\leftarrow( \abs{X'_{\Gamma'}},\mc{O}_{X'_{\Gamma'}}):(\tilde{f}_\phi,\tilde{f}_\phi^\sharp)$$ of spaces equipped with sheaves of rings in $U\big(\g^{\mbf{H}^+_{\Gamma'}}\oplus\bar\g^{\mbf{H}^-_{\Gamma'}}\big)\Mod^\Phi_\hbar$.

Taking the quantum fusion with respect to the ordered morphisms $\mbf{i}^{\Gamma'}_+$ for the $\g^{\mbf{H}^+_{\Gamma'}}$ factors and $\mbf{i}^{\Gamma'}_-$ for the $\bar\g^{\mbf{H}^-_{\Gamma'}}$ factors yields a morphism
$$\big(\abs{X_\Gamma},F^{\mbf{i}_\pm^{\Gamma'}\circ\phi_{\mbf{H}^\pm}}(\mc{O}_{X_\Gamma})\big)\leftarrow\big( \abs{X'_{\Gamma'}},F^{\mbf{i}_\pm^{\Gamma'}}(\mc{O}_{X'_{\Gamma'}})\big):\big(\tilde{f}_\phi,F^{\mbf{i}_\pm^{\Gamma'}}(\tilde{f}_\phi^\sharp)\big)$$ of spaces equipped with sheaves of rings in $U\big(\g^{\mbf{V}^+_{\Gamma'}}\oplus\bar\g^{\mbf{V}^-_{\Gamma'}}\big)\Mod^\Phi_\hbar$ (where as before, we have used the shorthand $F^{\mbf{i}_\pm^{\Gamma'}\circ\phi_{\mbf{H}^\pm}}= F^{\mbf{i}_+^{\Gamma'}\circ\phi_{\mbf{H}^+}}\circ F^{\mbf{i}_-^{\Gamma'}\circ\phi_{\mbf{H}^-}}$ and $F^{\mbf{i}_\pm^{\Gamma'}}=F^{\mbf{i}_+^{\Gamma'}}\circ F^{\mbf{i}_-^{\Gamma'}}$).

However, since $\mbf{i}_\pm^{\Gamma'}\circ\phi_{\mbf{H}^\pm}= \phi_{\mbf{V}^\pm}\circ \mbf{i}_\pm^{\Gamma}$ as ordered morphisms, Theorem~\ref{thm:CohAss} implies that we have a canonical isomorphism
$$\big(\abs{X_\Gamma},F^{\phi_{\mbf{V}^\pm}\circ \mbf{i}_\pm^{\Gamma}}(\mc{O}_{X_\Gamma})\big)\cong \big(\abs{X_\Gamma},F^{\mbf{i}_\pm^{\Gamma'}\circ\phi_{\mbf{H}^\pm}}(\mc{O}_{X_\Gamma})\big)$$ as spaces equipped with sheaves of rings in $U\big(\g^{\mbf{V}^+_{\Gamma'}}\oplus\bar\g^{\mbf{V}^-_{\Gamma'}}\big)\Mod^\Phi_\hbar$.

Let $C_\Gamma:=\prod_{v\in \mbf{V}_\Gamma} C_v$, and 
 $C'_{\Gamma'}:=\prod_{v\in \mbf{V}_{\Gamma'}} C'_v$  be as in \eqref{eq:XCMGamma}, and let $\mf{c}_\Gamma$ and $\mf{c}'_{\Gamma'}$ denote the respective Lie algebras. Note that $\mf{c}'_{\Gamma'}\subseteq \g^{\mbf{V}_{\Gamma'}}$ acts on $X_\Gamma$ via the pullback $$\g^{\mbf{V}_{\Gamma'}}\xleftarrow{\{\xi_v:=\eta_{\phi_{\mbf{V}}(v)}\}_{v\in \mbf{V}_\Gamma}\mapsfrom\{\eta_{v'}\}_{v'\in \mbf{V}_{\Gamma'}}} \g^{\mbf{V}_{\Gamma}}:\phi_{\mbf{V}}^*.$$

Now, applying Theorem~\ref{thm:RedToRedFus} iteratively, for each $v'\in \mbf{V}_{\Gamma'}$, yields a morphism 
\begin{multline*}r_{\phi_{\mbf{V}}}:\mf{c}_{\Gamma}\text{-invariants}\big(F^{\mbf{i}_\pm^\Gamma}(\mc{O}_{X_\Gamma})\big)
\to \mf{c}'_{\Gamma'}\text{-invariants}\big( F^{\phi_{\mbf{V}^\pm}\circ \mbf{i}_\pm^{\Gamma}}(\mc{O}_{X_\Gamma})\big) 
\\\cong\mf{c}'_{\Gamma'}\text{-invariants}\big(F^{\mbf{i}_\pm^{\Gamma'}\circ\phi_{\mbf{H}^\pm}}(\mc{O}_{X_\Gamma})\big).\end{multline*} of sheaves  over $X_\Gamma$ of $\K[\![\hbar]\!]$ rings.

Let $\pi:X_\Gamma\to X_\Gamma/C_\Gamma=\mc{M}_\Gamma$ and $\pi':X'_{\Gamma'}\to X'_{\Gamma'}/C'_{\Gamma'}=\mc{M}_{\Gamma'}$ denote the quotient maps, and $\mc{M}_\Gamma\leftarrow \mc{M}_{\Gamma'}:f_\phi$ the quotient of $\tilde{f}_\phi$. Suppose that $W\subseteq \abs{\mc{M}_{\Gamma}}$ is an open set.
Since $$\mc{O}_{\mc{M}_\Gamma^\hbar}(W)= \mf{c}_{\Gamma}\text{-invariants}\big(F^{\mbf{i}_\pm^\Gamma}(\mc{O}_{X_\Gamma}(W))\big)[\![\hbar]\!],$$  while for any open set $W'\subseteq \abs{\mc{M}_{\Gamma'}}$
$$\mc{O}_{\mc{M}_{\Gamma'}^\hbar}(W')=\mf{c}'_{\Gamma'}\text{-invariants}\big(F^{\mbf{i}_\pm^{\Gamma'}}(\mc{O}_{X'_{\Gamma'}}(W'))\big)[\![\hbar]\!],$$
the composition $$f_\phi^\sharp(W):=F^{\mbf{i}^{\Gamma'}_\pm}(\tilde{f}_\phi^\sharp)\circ r_{\phi_{\mbf{V}}}:\mc{O}_{\mc{M}_\Gamma^\hbar}(W)\to\mc{O}_{\mc{M}_{\Gamma'}^\hbar}\big(f_\phi^{-1}(W)\big)$$ is a morphism of associative $\K[\![\hbar]\!]$-algebras. Since $f_\phi^\sharp(W)$ is natural in $W$, 
$$(\mc{M}_\Gamma,\mc{O}_{\mc{M}_\Gamma^\hbar})\leftarrow (\mc{M}_{\Gamma'},\mc{O}_{\mc{M}_{\Gamma'}^\hbar}):(f_\phi,f_\phi^\sharp)$$
is a morphism in  $\Top_{\K\Alg_\hbar}$.
Since $(f_\phi,f_\phi^\sharp)$ is natural in $\phi$, it follows that the association $\Gamma\to (\mc{M}_\Gamma,\mc{O}_{\mc{M}_\Gamma^\circ}^\hbar)$ defines a contravariant functor $\sCol\xrightarrow{\Gamma\to \mc{M}_\Gamma^\hbar} \Top_{\K\Alg_\hbar}$ which, by construction, fits into the diagram \eqref{eq:QuantModuli}.

The functor $\Gamma\to \mc{M}_\Gamma^\hbar$ is monoidal, by construction.
\end{proof}

\begin{remark}
There are obvious variants of skeletized coloured surfaces in which algebraic groups $G$ and $C_v\subseteq G$ are replaced by (complex) Lie groups Lie groups, and the schemes $X_e$ assigned to (half) edges are replaced by (complex holomorphic) manifolds; subject to the condition that the quotient  $\mc{M}_\Gamma:=X_\Gamma/C_\Gamma$ defined in \eqref{eq:XCMGamma} is well behaved (in the appropriate sense). Moreover, it is clear that the same proof yields an appropriate variant of Theorem~\ref{thm:AlgFunct} in both the smooth or holomorphic settings.
\end{remark}

\section{Applications}
%

\subsection{Geometric skeletized Coloured surfaces}
Suppose that $G$ is an affine algebraic group whose Lie algebra is $\g$. 

Suppose that $(\Sigma,\mbf{Walls},C_\cdot)$ is a (geometric) coloured surface, 
and let $\mc{M}(\Sigma,\mbf{Walls},\mf{c}_\cdot; G)$ denote the moduli space classifying pairs 
$$\big(P\to \Sigma, \{Q_\mbf{w}\to \mbf{w}\}_{\mbf{w}\in \mbf{Walls}}\big)$$ where $P\to \Sigma$ is a flat $G$-bundle, and for each domain wall $\mbf{w}\in \mbf{Walls}$, $Q_\mbf{w}\to \mbf{w}$ is a $C_\mbf{w}$-invariant subbundle. Then $\mc{M}(\Sigma,\mbf{Walls},\mf{c}_\cdot; G)$ carries a natural Poisson structure \cite{LiBland:2012vo,LiBland:2013ue} which can be described in terms of a Goldman-type bracket \cite{Goldman:1986eh,Goldman:1984hr}. 

Suppose that we equip the (geometric) coloured surface $(\Sigma,\mbf{Walls},\mf{c}_\cdot)$ with a skelata whose bones are  $\mbf{Bones}$, 
 then \cite{LiBland:2012vo} shows that we have a canonical identificiation
 \begin{equation}\mc{M}(\Sigma,\mbf{Walls},\mf{c}_\cdot; G)\cong\mc{M}_{K(\Sigma)},\end{equation}
 where $K(\Sigma)$ is as in \eqref{eq:CombFunc}. Moreover,
 \cite{LiBland:2014da} proves that the resulting inclusion
 $$\mc{M}(\Sigma,\mbf{Walls},\mf{c}_\cdot; G)\xhookrightarrow{\hbar=0}\mc{M}_{K(\Sigma)}^\hbar$$
 is a deformation quantization of the Poisson structure.
 In particular, to first order in $\hbar$, the ringed space $\mc{M}^\hbar_{K(\Sigma)}$ does not depend on either the choice of skeleton or the choice of associator. 

%
%
%
Now let $\Col^{geom}$ denote the category of (geometric) coloured surfaces (as defined in \S~\ref{sec:GeomColSurf}), whose objects are coloured surfaces $(\Sigma,\mbf{Walls},C_\cdot)$, and whose morphisms $$(\Sigma,\mbf{Walls},C_\cdot)\to (\Sigma',\mbf{Walls}',C_\cdot')$$ are embeddings $\Sigma\to\Sigma'$ compatible with the colouring. Then we have a forgetful functor $$F:\sCol^{geom}\xrightarrow{(\Sigma,\mbf{Walls},C_\cdot;\mbf{Bones})\to (\Sigma,\mbf{Walls},C_\cdot)}\Col^{geom}$$ from the category of (geometric) coloured surfaces equipped with skeleta to the category of (geometric) coloured surfaces. 

\begin{Theorem}\label{thm:CohQuant}
The following diagram of functors commutes
$$\begin{tikzpicture}
	\node (sC) at (-1,2) {$\sCol$};
	\node (sCg) at (-4,0) {$\sCol^{geom}$};
	\node (Cg) at (-4,-2) {$\Col^{geom}$};
	\node (kh) at (3,0)  {$\Top_{\K\Alg_\hbar}^{op}$};
	\node (kh2) at (3,-2)  {$\Top_{\K\CPoiss}^{op}$};
	\draw[->] (sC) -- (kh);
	\draw[->] (sCg) -- (kh);
	\draw[->] (Cg) -- (kh2);
	\draw[->] (kh) -- (kh2);
	\draw[->] (sCg) -- node {$K$} (sC);
	\draw[->] (sCg) -- node {$F$} (Cg);
\end{tikzpicture}$$
where the top triangle is as described in \S~\ref{sec:ColSurf}, the rightmost vertical arrow is as in \eqref{eq:ForgToPoiss}, and the bottom arrow sends a (geometric) coloured surface $(\Sigma,\mbf{Walls},C_\cdot)$ to the corresponding moduli space of flat bundles $\mc{M}(\Sigma,\mbf{Walls},\mf{c}_\cdot; G)$ with its Goldman-type Poisson structure \cite{LiBland:2012vo,LiBland:2013ue}.
	To summarize, $\Sigma\to \mc{M}^\hbar_{K(\Sigma)}$ is a functorial deformation quantization of the moduli spaces associated to (geometric) coloured surfaces equipped with skeleta.
	
	In particular, to coherently quantize a diagram $p:J\to \Col^{geom}$, it suffices to find a lift $\tilde p:J\to \sCol^{geom}$ (i.e. $p=F\circ \tilde p$).
\end{Theorem}
\begin{proof}
	This follows by relating Theorem~\ref{thm:AlgFunct} to the results in \cite{LiBland:2014da}.
\end{proof}

As an example, we will use this result to describe the deformation quantization of a Poisson algebraic group $H$. Our procedure is a slight variant of the one first presented in \cite{Severa:2014te} to quantize Lie bialgebras. For now, we make the simplifying assumption that we may identify $G$ with the Drinfel'd double of $H$ (we will treat the general case in the next section). More precisely, we suppose that $H,H^*\subset G$ are closed subgroups where $H^*$ is a dual Poisson algebraic group to $H$, that the product map \begin{equation}\label{eq:invProd}H^*\times H\xrightarrow{h,h'\mapsto h\cdot h'}G\end{equation} is invertible, so that $\g=\h\oplus\h^*$ as a vector space (here $\h$ and $\h^*$ denote the Lie algebras corresponding to $H$ and $H^*$), and that $t\in S^2(\g)\cong S^2(\h\oplus \h^*)$ corresponds to the canonical symmetric pairing between $\h$ and $\h^*$.

Consider the skeletized coloured surface depicted in Figure~\ref{fig:PoissonLie}.
\begin{figure}[h!]
\begin{center}
\begin{subfigure}{.45\textwidth}
\centering
	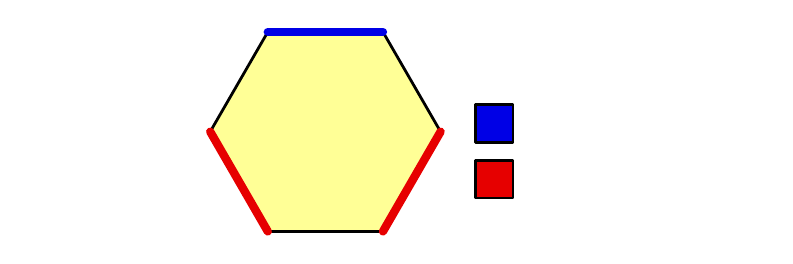
	\caption{\label{fig:PoissonLie}The (geometric) coloured surface for a Poisson Lie group.}
\end{subfigure}
\begin{subfigure}{.45\textwidth}
\centering
	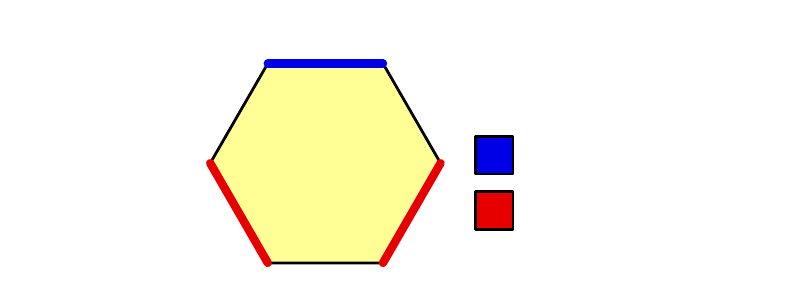
	\caption{\label{fig:PoissonLieSkel}A good choice of skeleton.}
\end{subfigure}
\caption{}
\end{center}
	
\end{figure}
By choosing a skeleton as in Figure~\ref{fig:PoissonLieSkel}
we may use \eqref{eq:AlgFormula} to compute the corresponding algebra:
$$\mc{O}\big((G\times G)/(H^*\times H\times H^*\big)[\![\hbar]\!],$$
where the action is $$(h^*_1,h,h^*_2)\cdot(g_1,g_2)=\big(hg_1(h_1^*)^{-1},hg_2(h_2^*)^{-1}\big),\quad(h^*_1,h,h^*_2)\in H^*\times H\times H^*,\quad(g_1,g_2)\in G\times G.$$
In particular, using the decomposition \eqref{eq:invProd}, we may identify the algebra with $\mc{O}(H)[\![\hbar]\!]$ via the map $$H\xrightarrow{h\mapsto (1,h)}(G\times G)/(H^*\times H\times H^*).$$
The multiplication \eqref{eq:multGamma} is a deformation quantization of the Poisson structure on $H$ \cite{LiBland:2014da}.

We now want to describe a compatible comultiplication on $\mc{O}(H)[\![\hbar]\!]$ corresponding to the group product $H\times H\to H$. For this, consider the following sequence of embeddings:

\begin{equation}\label{eq:MultEmb}
	
\end{equation}
The colours are as before, and we choose the following compatible skeleton on $\Sigma^{(2)}$:
\begin{center}
\begingroup%
  \makeatletter%
  \providecommand\color[2][]{%
    \errmessage{(Inkscape) Color is used for the text in Inkscape, but the package 'color.sty' is not loaded}%
    \renewcommand\color[2][]{}%
  }%
  \providecommand\transparent[1]{%
    \errmessage{(Inkscape) Transparency is used (non-zero) for the text in Inkscape, but the package 'transparent.sty' is not loaded}%
    \renewcommand\transparent[1]{}%
  }%
  \providecommand\rotatebox[2]{#2}%
  \ifx\svgwidth\undefined%
    \setlength{\unitlength}{39.19645386bp}%
    \ifx\svgscale\undefined%
      \relax%
    \else%
      \setlength{\unitlength}{\unitlength * \real{\svgscale}}%
    \fi%
  \else%
    \setlength{\unitlength}{\svgwidth}%
  \fi%
  \global\let\svgwidth\undefined%
  \global\let\svgscale\undefined%
  \makeatother%
  \begin{picture}(1,1.00264102)%
    \put(0,0){\includegraphics[width=\unitlength,page=1]{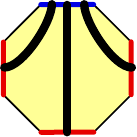}}%
  \end{picture}%
\endgroup%
	
\end{center}

Once again, using \eqref{eq:AlgFormula}, we may identify the algebra corresponding to $\Sigma^{(2)}$ with

$$\mc{O}\big((G\times G\times G)/(H\times H^*\times H^*\times H^*)\big)[\![\hbar]\!]\cong \mc{O}\big((H\times H\times H)/H\big)[\![\hbar]\!]\cong\mc{O}(H\times H)[\![\hbar]\!],$$ (here $H$ acts diagonally on $H\times H\times H$).  Thus, applying the functor $\mc{O}(\mc{M}_{K(-)}^\hbar)$ to the sequence of embeddings \eqref{eq:MultEmb} yields a sequence of 
 algebra morphisms:
\begin{equation}\label{eq:PrecoMult}\mc{O}(H)[\![\hbar]\!]\hat\otimes\mc{O}(H)[\![\hbar]\!]\xrightarrow{\cong}\mc{O}(H\times H)[\![\hbar]\!]\leftarrow \mc{O}(H)[\![\hbar]\!].\end{equation}

\begin{Theorem}\label{thm:HopfQuant1}
The left hand morphism in \eqref{eq:PrecoMult} is invertible, and the corresponding composite \begin{equation}\label{eq:coMult1}\Delta:\mc{O}(H)[\![\hbar]\!]\to \mc{O}(H)[\![\hbar]\!]\hat\otimes\mc{O}(H)[\![\hbar]\!]\end{equation}
is the comultiplication for a Hopf-algebra structure on $\mc{O}(H)[\![\hbar]\!]$ quantizing the Poisson algebraic group $H$.
\end{Theorem}

\begin{proof}

At $\hbar=0$, the left hand morphism in \eqref{eq:PrecoMult} is the canonical identification, and hence invertible. Since all the algebras in \eqref{eq:PrecoMult} are flat $\mbb{K}[\![\hbar]\!]$ algebras, it follows that the left hand morphism in \eqref{eq:PrecoMult} remains invertible away from $\hbar=0$.


 Now coassociativity is the equation
 \begin{subequations}
 \begin{equation}\label{eq:CoAss0}(\id\hat\otimes\Delta)\circ\Delta=(\Delta\hat\otimes\id)\circ\Delta.\end{equation} We claim that both sides of this equations are equal to the morphism  \begin{equation}\label{eq:coAss1}\Delta^{(3)}:\mc{O}(H)[\![\hbar]\!]\to \mc{O}(H)[\![\hbar]\!]\hat\otimes\mc{O}(H)[\![\hbar]\!]\hat\otimes\mc{O}(H)[\![\hbar]\!]\end{equation}
 obtained by applying the functor $\mc{O}(\mc{M}_{K(-)}^\hbar)$ to the diagram
  \begin{equation}\label{eq:CoAss2}
\begingroup%
  \makeatletter%
  \providecommand\color[2][]{%
    \errmessage{(Inkscape) Color is used for the text in Inkscape, but the package 'color.sty' is not loaded}%
    \renewcommand\color[2][]{}%
  }%
  \providecommand\transparent[1]{%
    \errmessage{(Inkscape) Transparency is used (non-zero) for the text in Inkscape, but the package 'transparent.sty' is not loaded}%
    \renewcommand\transparent[1]{}%
  }%
  \providecommand\rotatebox[2]{#2}%
  \ifx\svgwidth\undefined%
    \setlength{\unitlength}{294.09228227bp}%
    \ifx\svgscale\undefined%
      \relax%
    \else%
      \setlength{\unitlength}{\unitlength * \real{\svgscale}}%
    \fi%
  \else%
    \setlength{\unitlength}{\svgwidth}%
  \fi%
  \global\let\svgwidth\undefined%
  \global\let\svgscale\undefined%
  \makeatother%
  \begin{picture}(1,0.70616396)%
    \put(0,0){\includegraphics[width=\unitlength,page=1]{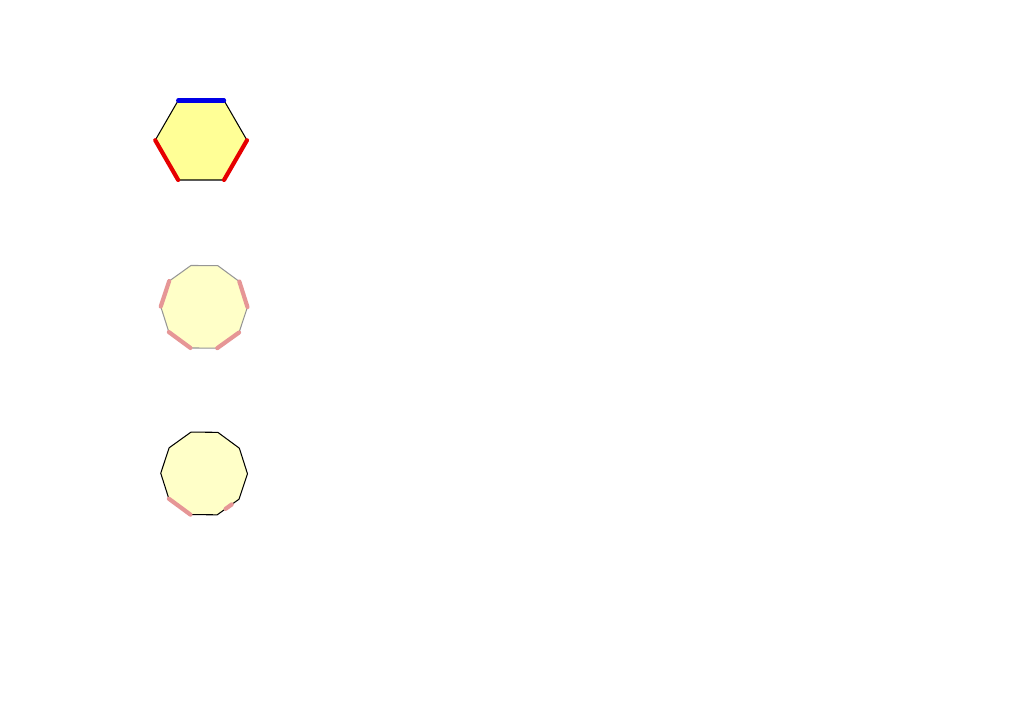}}%
    \put(0.65992826,-0.06026356){\color[rgb]{0,0,0}\makebox(0,0)[rb]{\smash{}}}%
    \put(0,0){\includegraphics[width=\unitlength,page=2]{coAssSh.pdf}}%
    \put(0.16267926,0.52719475){\color[rgb]{0,0,0}\makebox(0,0)[rb]{\smash{$0$}}}%
    \put(0.15723879,0.41294489){\color[rgb]{0,0,0}\makebox(0,0)[rb]{\smash{$0$}}}%
    \put(0.23340538,0.52719475){\color[rgb]{0,0,0}\makebox(0,0)[lb]{\smash{$3$}}}%
    \put(0.24428632,0.41294489){\color[rgb]{0,0,0}\makebox(0,0)[lb]{\smash{$3$}}}%
    \put(0.24428632,0.24973082){\color[rgb]{0,0,0}\makebox(0,0)[lb]{\smash{$3$}}}%
    \put(0.15723879,0.24973082){\color[rgb]{0,0,0}\makebox(0,0)[rb]{\smash{$0$}}}%
    \put(0.17356022,0.34221879){\color[rgb]{0,0,0}\makebox(0,0)[rb]{\smash{$1$}}}%
    \put(0.22796491,0.34221879){\color[rgb]{0,0,0}\makebox(0,0)[lb]{\smash{$2$}}}%
    \put(0.22796491,0.17900472){\color[rgb]{0,0,0}\makebox(0,0)[lb]{\smash{$2$}}}%
    \put(0.17356022,0.17900472){\color[rgb]{0,0,0}\makebox(0,0)[rb]{\smash{$1$}}}%
    \put(0,0){\includegraphics[width=\unitlength,page=3]{coAssSh.pdf}}%
    \put(0.53807166,0.20076659){\color[rgb]{0,0,0}\makebox(0,0)[rb]{\smash{$0$}}}%
    \put(0.60879776,0.20076659){\color[rgb]{0,0,0}\makebox(0,0)[lb]{\smash{$1$}}}%
    \put(0.66320245,0.20076659){\color[rgb]{0,0,0}\makebox(0,0)[rb]{\smash{$1$}}}%
    \put(0.73392855,0.20076659){\color[rgb]{0,0,0}\makebox(0,0)[lb]{\smash{$2$}}}%
    \put(0.78289269,0.20076659){\color[rgb]{0,0,0}\makebox(0,0)[rb]{\smash{$2$}}}%
    \put(0.85361879,0.20076659){\color[rgb]{0,0,0}\makebox(0,0)[lb]{\smash{$3$}}}%
  \end{picture}%
\endgroup%
 
 \end{equation} 
 \end{subequations}
 and composing from top left to bottom right (after applying the functor $\mc{O}(\mc{M}_{K(-)}^\hbar)$, the horizontal edge become invertible). In more detail, the skeleta are as in\footnote{These fit together to define a cosimplicial object \begin{equation}\label{eq:coSimpl}\mbf{\Delta}\xrightarrow{[n]\mapsto \Sigma^{(n)}} \sCol^{geom}\xrightarrow{\mc{O}(\mc{M}_{K(-)}^\hbar)}\mbb{K}\Alg_\hbar,\end{equation}
 (here $\mbf{\Delta}$ is the simplex category, or equivalently, the category of linearly ordered finite sets).}:
 \begin{center}
\begingroup%
  \makeatletter%
  \providecommand\color[2][]{%
    \errmessage{(Inkscape) Color is used for the text in Inkscape, but the package 'color.sty' is not loaded}%
    \renewcommand\color[2][]{}%
  }%
  \providecommand\transparent[1]{%
    \errmessage{(Inkscape) Transparency is used (non-zero) for the text in Inkscape, but the package 'transparent.sty' is not loaded}%
    \renewcommand\transparent[1]{}%
  }%
  \providecommand\rotatebox[2]{#2}%
  \ifx\svgwidth\undefined%
    \setlength{\unitlength}{426.61171875bp}%
    \ifx\svgscale\undefined%
      \relax%
    \else%
      \setlength{\unitlength}{\unitlength * \real{\svgscale}}%
    \fi%
  \else%
    \setlength{\unitlength}{\svgwidth}%
  \fi%
  \global\let\svgwidth\undefined%
  \global\let\svgscale\undefined%
  \makeatother%
  \begin{picture}(1,0.17381617)%
    \put(0,0){\includegraphics[width=\unitlength,page=1]{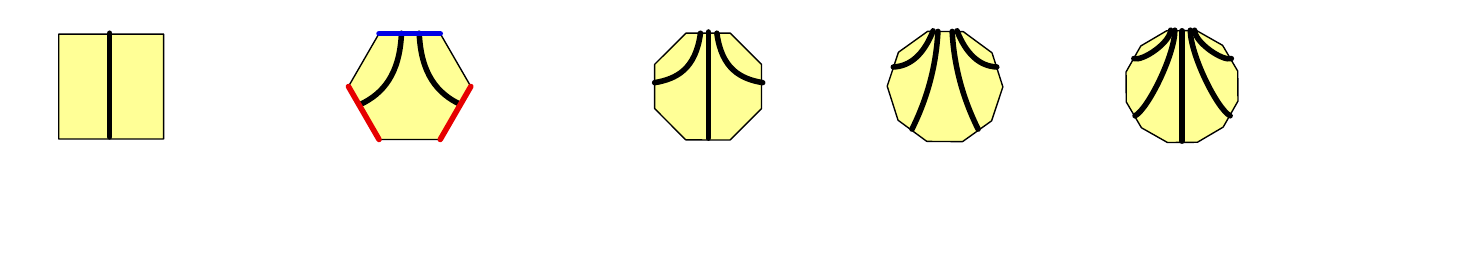}}%
    \put(0.27647103,0.0226256){\color[rgb]{0,0,0}\makebox(0,0)[b]{\smash{$\Sigma=\Sigma^{(1)}$}}}%
    \put(0.47794153,0.0226256){\color[rgb]{0,0,0}\makebox(0,0)[b]{\smash{$\Sigma^{(2)}$}}}%
    \put(0,0){\includegraphics[width=\unitlength,page=2]{CoSimpl.pdf}}%
    \put(0.79788848,0.0226256){\color[rgb]{0,0,0}\makebox(0,0)[b]{\smash{$\Sigma^{(4)}$}}}%
    \put(0,0){\includegraphics[width=\unitlength,page=3]{CoSimpl.pdf}}%
    \put(0.0750005,0.0226256){\color[rgb]{0,0,0}\makebox(0,0)[b]{\smash{$\Sigma^{(0)}$}}}%
    \put(0,0){\includegraphics[width=\unitlength,page=4]{CoSimpl.pdf}}%
    \put(0.637915,0.0226256){\color[rgb]{0,0,0}\makebox(0,0)[b]{\smash{$\Sigma^{(3)}$}}}%
    \put(0.64958256,-0.01275831){\color[rgb]{0,0,0}\makebox(0,0)[rb]{\smash{}}}%
    \put(0.94136202,0.09763526){\color[rgb]{0,0,0}\makebox(0,0)[b]{\smash{$\cdots$}}}%
  \end{picture}%
\endgroup%

 \end{center}
 and we have numbered the $H^*$-coloured (red) domain walls to emphasize where they are mapped to. 

Indeed the diagram \eqref{eq:CoAss2} fits into the following commutative diagram:
  \begin{center}
\begingroup%
  \makeatletter%
  \providecommand\color[2][]{%
    \errmessage{(Inkscape) Color is used for the text in Inkscape, but the package 'color.sty' is not loaded}%
    \renewcommand\color[2][]{}%
  }%
  \providecommand\transparent[1]{%
    \errmessage{(Inkscape) Transparency is used (non-zero) for the text in Inkscape, but the package 'transparent.sty' is not loaded}%
    \renewcommand\transparent[1]{}%
  }%
  \providecommand\rotatebox[2]{#2}%
  \ifx\svgwidth\undefined%
    \setlength{\unitlength}{412.68603516bp}%
    \ifx\svgscale\undefined%
      \relax%
    \else%
      \setlength{\unitlength}{\unitlength * \real{\svgscale}}%
    \fi%
  \else%
    \setlength{\unitlength}{\svgwidth}%
  \fi%
  \global\let\svgwidth\undefined%
  \global\let\svgscale\undefined%
  \makeatother%
  \begin{picture}(1,0.78657601)%
    \put(0,0){\includegraphics[width=\unitlength,page=1]{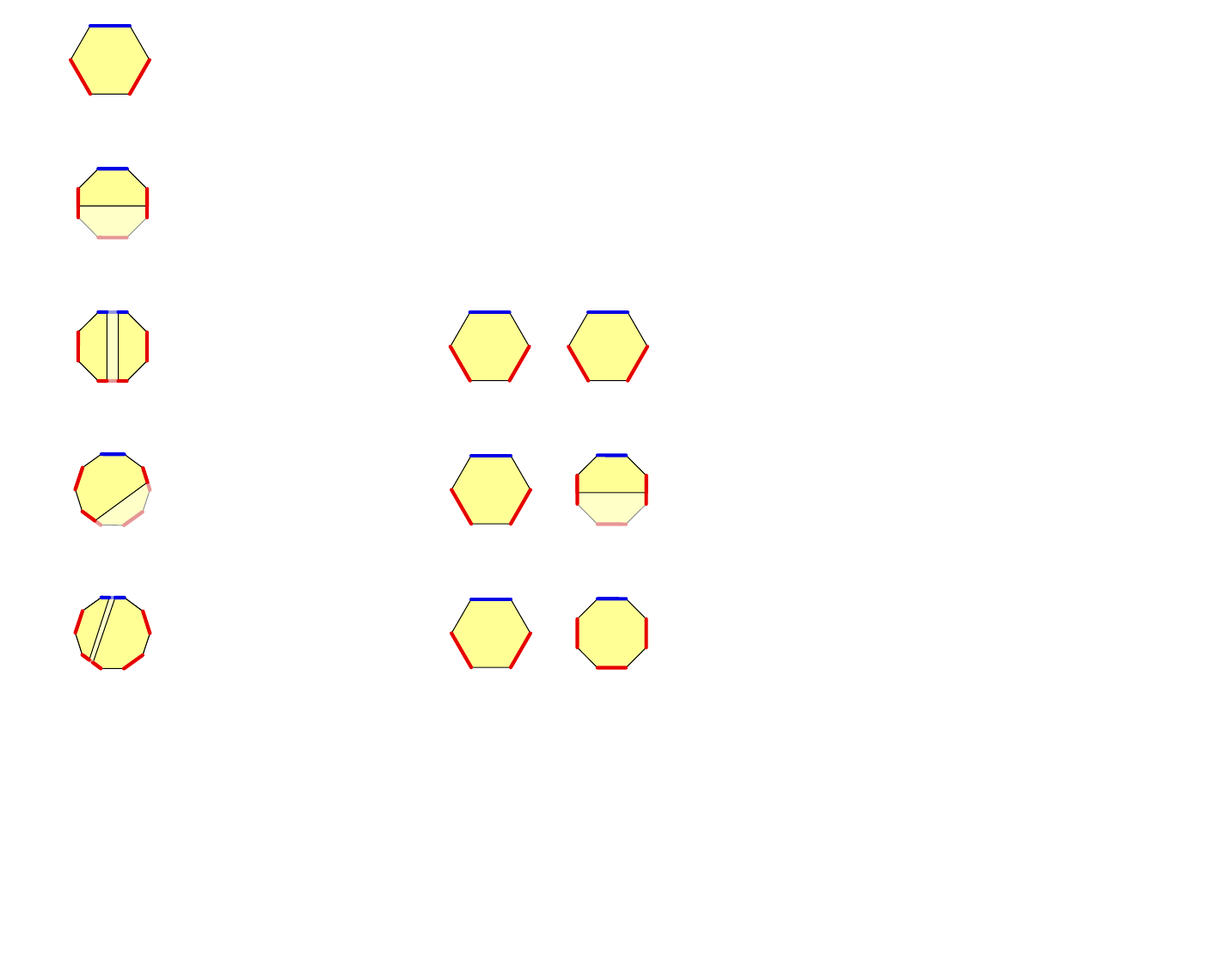}}%
    \put(0.41925336,0.05707722){\color[rgb]{0,0,0}\makebox(0,0)[rb]{\smash{}}}%
    \put(0,0){\includegraphics[width=\unitlength,page=2]{coAss.pdf}}%
    \put(0.06489901,0.70833976){\color[rgb]{0,0,0}\makebox(0,0)[rb]{\smash{$0$}}}%
    \put(0.06102197,0.61141377){\color[rgb]{0,0,0}\makebox(0,0)[rb]{\smash{$0$}}}%
    \put(0.06102197,0.49897963){\color[rgb]{0,0,0}\makebox(0,0)[rb]{\smash{$0$}}}%
    \put(0.06102197,0.39429957){\color[rgb]{0,0,0}\makebox(0,0)[rb]{\smash{$0$}}}%
    \put(0.37118514,0.47571739){\color[rgb]{0,0,0}\makebox(0,0)[rb]{\smash{$0$}}}%
    \put(0.37118514,0.35940621){\color[rgb]{0,0,0}\makebox(0,0)[rb]{\smash{$0$}}}%
    \put(0.37118514,0.24309503){\color[rgb]{0,0,0}\makebox(0,0)[rb]{\smash{$0$}}}%
    \put(0.11530054,0.70833976){\color[rgb]{0,0,0}\makebox(0,0)[lb]{\smash{$3$}}}%
    \put(0.12305462,0.61141377){\color[rgb]{0,0,0}\makebox(0,0)[lb]{\smash{$3$}}}%
    \put(0.12305462,0.49897963){\color[rgb]{0,0,0}\makebox(0,0)[lb]{\smash{$3$}}}%
    \put(0.12305462,0.39429957){\color[rgb]{0,0,0}\makebox(0,0)[lb]{\smash{$3$}}}%
    \put(0.12305462,0.27798839){\color[rgb]{0,0,0}\makebox(0,0)[lb]{\smash{$3$}}}%
    \put(0.06102197,0.27798839){\color[rgb]{0,0,0}\makebox(0,0)[rb]{\smash{$0$}}}%
    \put(0.53014375,0.38266845){\color[rgb]{0,0,0}\makebox(0,0)[lb]{\smash{$3$}}}%
    \put(0.52238967,0.47571739){\color[rgb]{0,0,0}\makebox(0,0)[lb]{\smash{$4$}}}%
    \put(0.42546369,0.47571739){\color[rgb]{0,0,0}\makebox(0,0)[lb]{\smash{$1$}}}%
    \put(0.46811112,0.47571739){\color[rgb]{0,0,0}\makebox(0,0)[rb]{\smash{$1$}}}%
    \put(0.42546369,0.35940621){\color[rgb]{0,0,0}\makebox(0,0)[lb]{\smash{$1$}}}%
    \put(0.0920383,0.56876634){\color[rgb]{0,0,0}\makebox(0,0)[b]{\smash{$1$}}}%
    \put(0.0920383,0.45245516){\color[rgb]{0,0,0}\makebox(0,0)[b]{\smash{$1$}}}%
    \put(0.07265311,0.34389806){\color[rgb]{0,0,0}\makebox(0,0)[rb]{\smash{$1$}}}%
    \put(0.1114235,0.34389806){\color[rgb]{0,0,0}\makebox(0,0)[lb]{\smash{$2$}}}%
    \put(0.1114235,0.22758688){\color[rgb]{0,0,0}\makebox(0,0)[lb]{\smash{$2$}}}%
    \put(0.07265311,0.22758688){\color[rgb]{0,0,0}\makebox(0,0)[rb]{\smash{$1$}}}%
    \put(0.49525039,0.33614398){\color[rgb]{0,0,0}\makebox(0,0)[b]{\smash{$2$}}}%
    \put(0.46811112,0.38266845){\color[rgb]{0,0,0}\makebox(0,0)[rb]{\smash{$1$}}}%
    \put(0.53014375,0.26635727){\color[rgb]{0,0,0}\makebox(0,0)[lb]{\smash{$3$}}}%
    \put(0.42546369,0.24309503){\color[rgb]{0,0,0}\makebox(0,0)[lb]{\smash{$1$}}}%
    \put(0.49525039,0.2198328){\color[rgb]{0,0,0}\makebox(0,0)[b]{\smash{$2$}}}%
    \put(0.46811112,0.26635727){\color[rgb]{0,0,0}\makebox(0,0)[rb]{\smash{$1$}}}%
    \put(0,0){\includegraphics[width=\unitlength,page=3]{coAss.pdf}}%
    \put(0.37118514,0.12678385){\color[rgb]{0,0,0}\makebox(0,0)[rb]{\smash{$0$}}}%
    \put(0.53014375,0.15004609){\color[rgb]{0,0,0}\makebox(0,0)[lb]{\smash{$3$}}}%
    \put(0.42546369,0.12678385){\color[rgb]{0,0,0}\makebox(0,0)[lb]{\smash{$1$}}}%
    \put(0.49525039,0.10352162){\color[rgb]{0,0,0}\makebox(0,0)[b]{\smash{$2$}}}%
    \put(0.46811112,0.15004609){\color[rgb]{0,0,0}\makebox(0,0)[rb]{\smash{$1$}}}%
    \put(0,0){\includegraphics[width=\unitlength,page=4]{coAss.pdf}}%
    \put(0.72011868,0.12678385){\color[rgb]{0,0,0}\makebox(0,0)[rb]{\smash{$0$}}}%
    \put(0.77052019,0.12678385){\color[rgb]{0,0,0}\makebox(0,0)[lb]{\smash{$1$}}}%
    \put(0.80929058,0.12678385){\color[rgb]{0,0,0}\makebox(0,0)[rb]{\smash{$1$}}}%
    \put(0.85969209,0.12678385){\color[rgb]{0,0,0}\makebox(0,0)[lb]{\smash{$2$}}}%
    \put(0.89458539,0.12678385){\color[rgb]{0,0,0}\makebox(0,0)[rb]{\smash{$2$}}}%
    \put(0.9449869,0.12678385){\color[rgb]{0,0,0}\makebox(0,0)[lb]{\smash{$3$}}}%
  \end{picture}%
\endgroup%
 
 \end{center}
 After applying the functor $\mc{O}(\mc{M}_{K(-)}^\hbar)$, the horizontal edges again become invertible, 
 and composing the diagram from top-left to bottom right along the top/right-most path yields the left hand side of \eqref{eq:CoAss0}. On the other hand composing the diagram from top-left to bottom right along the bottom/left-most path yields the right hand side of \eqref{eq:CoAss0}, which shows that it too is equal to \eqref{eq:coAss1}.
 
 The counit is obtained by applying the functor $\mc{O}(\mc{M}_{K(-)}^\hbar)$ to the unique embedding $\Sigma\cong \Sigma^{(1)}\to \Sigma^{(0)}$.
 
 Now, at $\hbar=0$, \eqref{eq:coMult1} 
  is just the comultiplication dual to the group product $H\times H\to H$, so the bialgebra structure on $\mc{O}(H)$ is the standard one; in particular it is a Hopf-algebra. It follows that $\mc{O}(H)[\![\hbar]\!]$ remains a Hopf-algebra away from $\hbar=0$ which deformation quantizes the Poisson algebraic group $H$.
  
%
%
\end{proof}
\begin{remark}
	We should mention that the sequence of embeddings \eqref{eq:MultEmb} corresponds to vertical-time-slices of the following cobordism
	\begin{center}
\begingroup%
  \makeatletter%
  \providecommand\color[2][]{%
    \errmessage{(Inkscape) Color is used for the text in Inkscape, but the package 'color.sty' is not loaded}%
    \renewcommand\color[2][]{}%
  }%
  \providecommand\transparent[1]{%
    \errmessage{(Inkscape) Transparency is used (non-zero) for the text in Inkscape, but the package 'transparent.sty' is not loaded}%
    \renewcommand\transparent[1]{}%
  }%
  \providecommand\rotatebox[2]{#2}%
  \ifx\svgwidth\undefined%
    \setlength{\unitlength}{157.71936035bp}%
    \ifx\svgscale\undefined%
      \relax%
    \else%
      \setlength{\unitlength}{\unitlength * \real{\svgscale}}%
    \fi%
  \else%
    \setlength{\unitlength}{\svgwidth}%
  \fi%
  \global\let\svgwidth\undefined%
  \global\let\svgscale\undefined%
  \makeatother%
  \begin{picture}(1,1.01868871)%
    \put(0,0){\includegraphics[width=\unitlength,page=1]{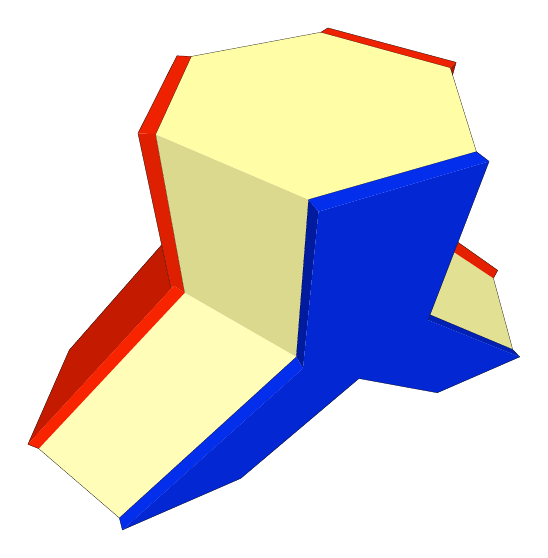}}%
  \end{picture}%
\endgroup%

	\end{center}
	In particular, there should be a deeper TQFT proof of Theorem~\ref{thm:HopfQuant1}. In future work we hope to define this TQFT rigorously using factorization homology.
\end{remark}

As a second example, we will use this result to describe an equivariant deformation quantization procedure for Lu-Yakimov Poisson Homogeneous spaces \cite{Lu06}. In more detail, suppose that $H,H^*\subset G$ are as before and $C\subset G$ is a closed subgroup whose Lie algebra $\mf{c}\subset \g$ is coisotropic. Then Lu and Yakimov described a natural Poisson structure on $G/C$ for which the action map $H\times G/C\to G/C$ is Poisson. 

Consider the skeletized coloured surface 
\begin{center}
\begingroup%
  \makeatletter%
  \providecommand\color[2][]{%
    \errmessage{(Inkscape) Color is used for the text in Inkscape, but the package 'color.sty' is not loaded}%
    \renewcommand\color[2][]{}%
  }%
  \providecommand\transparent[1]{%
    \errmessage{(Inkscape) Transparency is used (non-zero) for the text in Inkscape, but the package 'transparent.sty' is not loaded}%
    \renewcommand\transparent[1]{}%
  }%
  \providecommand\rotatebox[2]{#2}%
  \ifx\svgwidth\undefined%
    \setlength{\unitlength}{199.53677979bp}%
    \ifx\svgscale\undefined%
      \relax%
    \else%
      \setlength{\unitlength}{\unitlength * \real{\svgscale}}%
    \fi%
  \else%
    \setlength{\unitlength}{\svgwidth}%
  \fi%
  \global\let\svgwidth\undefined%
  \global\let\svgscale\undefined%
  \makeatother%
  \begin{picture}(1,0.44330154)%
    \put(0,0){\includegraphics[width=\unitlength,page=1]{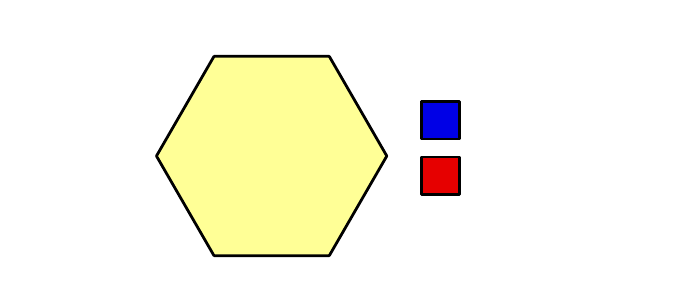}}%
    \put(0.69419407,0.25246893){\color[rgb]{0,0,0}\makebox(0,0)[lb]{\smash{$H\subset G$}}}%
    \put(0.69419407,0.17228321){\color[rgb]{0,0,0}\makebox(0,0)[lb]{\smash{$H^*\subset G$}}}%
    \put(0.61400835,0.33265465){\color[rgb]{0,0,0}\makebox(0,0)[lb]{\smash{Domain walls}}}%
    \put(0.18781782,0.20458195){\color[rgb]{0,0,0}\makebox(0,0)[rb]{\smash{$\Sigma_C=$}}}%
    \put(0,0){\includegraphics[width=\unitlength,page=2]{PoissonHomo.pdf}}%
    \put(0.69419407,0.09209749){\color[rgb]{0,0,0}\makebox(0,0)[lb]{\smash{$C\subset G$}}}%
    \put(0,0){\includegraphics[width=\unitlength,page=3]{PoissonHomo.pdf}}%
  \end{picture}%
\endgroup%

\end{center}
The corresponding moduli space is $\mc{M}_{K(\Sigma_C)}=G/C$ with the Lu-Yakimov Poisson structure \cite{LiBland:2012vo}, and Theorem~\ref{thm:CohQuant} implies that $G/C\xhookrightarrow{\hbar=0}\mc{M}_{K(\Sigma_C)}^\hbar$ is a deformation quantization of the Poisson scheme $G/C$ (cf. \cite{LiBland:2014da}).  
 
 Applying the functor $\mc{M}_{K(-)}^\hbar$ to following sequence of embeddings
 \begin{center}
\begingroup%
  \makeatletter%
  \providecommand\color[2][]{%
    \errmessage{(Inkscape) Color is used for the text in Inkscape, but the package 'color.sty' is not loaded}%
    \renewcommand\color[2][]{}%
  }%
  \providecommand\transparent[1]{%
    \errmessage{(Inkscape) Transparency is used (non-zero) for the text in Inkscape, but the package 'transparent.sty' is not loaded}%
    \renewcommand\transparent[1]{}%
  }%
  \providecommand\rotatebox[2]{#2}%
  \ifx\svgwidth\undefined%
    \setlength{\unitlength}{444.86015625bp}%
    \ifx\svgscale\undefined%
      \relax%
    \else%
      \setlength{\unitlength}{\unitlength * \real{\svgscale}}%
    \fi%
  \else%
    \setlength{\unitlength}{\svgwidth}%
  \fi%
  \global\let\svgwidth\undefined%
  \global\let\svgscale\undefined%
  \makeatother%
  \begin{picture}(1,0.23917552)%
    \put(0,0){\includegraphics[width=\unitlength,page=1]{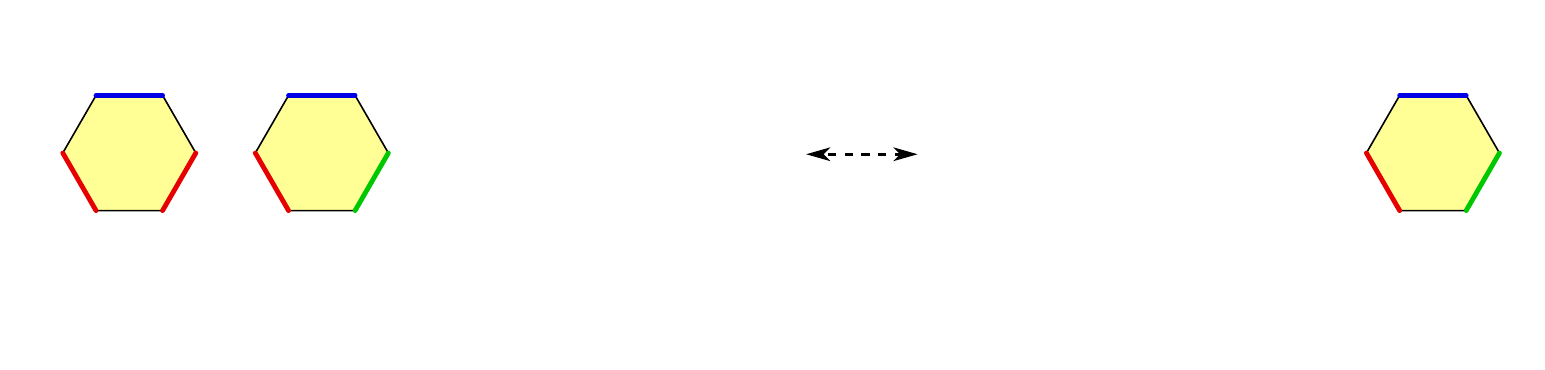}}%
    \put(0.55897301,0.14604063){\color[rgb]{0,0,0}\makebox(0,0)[b]{\smash{$=$}}}%
    \put(0.43118712,0.02169749){\color[rgb]{0,0,0}\makebox(0,0)[b]{\smash{$\Sigma_C^{(1)}$}}}%
    \put(0.68654827,0.02169749){\color[rgb]{0,0,0}\makebox(0,0)[b]{\smash{$\Sigma_C^{(1)}$}}}%
    \put(0.92752287,0.02169749){\color[rgb]{0,0,0}\makebox(0,0)[b]{\smash{$\Sigma_C$}}}%
    \put(0.14705288,0.02169749){\color[rgb]{0,0,0}\makebox(0,0)[b]{\smash{$\Sigma\sqcup\Sigma_C$}}}%
    \put(0,0){\includegraphics[width=\unitlength,page=2]{PoissonLieAct.pdf}}%
  \end{picture}%
\endgroup%

 \end{center}
 yields the following sequence of morphisms of spaces equipped with sheaves of rings in $\K\Alg_\hbar$:
 \begin{equation}\label{eq:QuantAct}\mc{M}^\hbar_{K(\Sigma)}\times\mc{M}^\hbar_{K(\Sigma_C)}\xleftarrow{\cong}\mc{M}^\hbar_{K(\Sigma_C^{(1)})}\rightarrow\mc{M}^\hbar_{K(\Sigma_C)},\end{equation}
 which restricts at $\hbar=0$ to the morphism of schemes
 $$H\times G/C\cong (H\times H\times G/C)/H \to  G/C$$
 where the leftmost map is invertible (the $H$-action in the middle term is diagonal) and the composite is the action map. As before, it follows that the leftmost map in \eqref{eq:QuantAct} is invertible, and the composite defines 
 an action map of the monoid $\mc{M}^\hbar_{K(\Sigma)}$ on $\mc{M}^\hbar_{K(\Sigma_C)}$$$\mc{M}^\hbar_{K(\Sigma)}\times\mc{M}^\hbar_{K(\Sigma_C)}\to \mc{M}^\hbar_{K(\Sigma_C)}$$ in the category $\Top_{\K\Alg_\hbar}$. That is, we have an equivariant deformation quantization: 
 $$\begin{tikzpicture}
\node (m-1-1) at (0,0) {$\mc{M}^\hbar_{K(\Sigma)}$};
\node (m-1-2) at (2.75cm,0) {$ \mc{M}^\hbar_{K(\Sigma_C)}$};
\node (m-2-1) at (0,-4cm) {$H$};
\node (m-2-2) at (2.75cm,-4cm) {$G/C$};
	\draw[<-] (m-1-2.center) ++(-.5cm,.5cm) arc [start angle=45, end angle=315, radius=0.75cm];
	\draw[<-] (m-2-2.center) ++(-.5cm,.5cm) arc [start angle=25, end angle=335, radius=1cm];
	\draw[right hook->] (m-2-2) -- node[swap] {$\hbar=0$} (m-1-2);
	\draw[right hook->] (m-2-1) -- node {$\hbar=0$} (m-1-1);
\end{tikzpicture}$$ 
 of the Poisson homogeneous space $G/C$ for the Poisson Lie group $H$. 

\subsection{Quantization of Poisson Algebraic groups (the general case)}\label{sec:QPAG}
Suppose that $H$ is a Poisson algebraic group.\footnote{Recall: by an \emph{algebraic group} we mean an affine group scheme of finite type over $\mbb{K}$; in particular, since $\mbb{K}$ is of characteristic zero, an algebraic group over $\mbb{K}$ is smooth.} Let $\h$ denote the corresponding Lie algebra, $(\g;\h,\h^*)$ the corresponding Manin triple,\footnote{We do not assume that $(\g;\h;\h^*)$ integrate to a compatible triple of algebraic groups.} and $t:=\langle\cdot,\cdot\rangle^{-1}\in S^2(\g)^\g$ the inverse of the non-degenerate invariant symmetric quadratic form on $\g$. As proven in the appendix of \cite{LiBland:2011vq}, there is a unique action of $\g$ on $H$ which extends the dressing action $\rho:\h^*\to \mf{X}(H)$ of $\h^*\subseteq \g$ and the action of $\h\subseteq \g$ by left invariant vector fields; and the stabilizers for this $\g$-action are coisotropic. 
Moreover, it will be significant to us that $t\in S^2(\g)^\g$ is $H$ invariant. Here the action of $H$ on $\g$ arises from the Poisson structure on $H$: it extends the adjoint action on $\h\subset\g$, and for $\xi\in \h^*\subset\g$ is given by $$h\cdot\xi=\ad^*_h\xi+\iota_{\rho(\xi)}\theta^R_h,\quad h\in H,\quad \xi\in \h^*,$$
where $\theta^R$ is the right Maurer-Cartan form (see \cite[Appendix~B]{LiBland:2011vq} for details).
%
%

Let $\Gamma_H^{(n)}$ denote the ciliated graph which has one (positive) vertex $\mbf{V}_{\Gamma_H^{(n)}}=\mbf{V}^+_{\Gamma_H^{(n)}}=\{\ast\}$, no full-edges $\mbf{E}_{\Gamma_H^{(n)}}=\emptyset$, and $n+1$ (positive) half-edges $\mbf{H}_{\Gamma_H^{(n)}}=\mbf{H}^+_{\Gamma_H^{(2)}}=\{0,1,2,\dots,n\}$ with $0<1<\dots<n$. We colour $\Gamma_H^{(n)}$ as follows: $C_\ast=H$, and $X_0=\dots=X_n=H$. 
We may picture $\Gamma_H^{(n)}$ in terms of the corresponding surfaces $\Sigma_{{\Gamma_H^{(n)}}}$ as follows (cf. Remark~\ref{rem:GeomColSurf}):
\begin{center}
  \def\svgwidth{\textwidth}
\begingroup%
  \makeatletter%
  \providecommand\color[2][]{%
    \errmessage{(Inkscape) Color is used for the text in Inkscape, but the package 'color.sty' is not loaded}%
    \renewcommand\color[2][]{}%
  }%
  \providecommand\transparent[1]{%
    \errmessage{(Inkscape) Transparency is used (non-zero) for the text in Inkscape, but the package 'transparent.sty' is not loaded}%
    \renewcommand\transparent[1]{}%
  }%
  \providecommand\rotatebox[2]{#2}%
  \ifx\svgwidth\undefined%
    \setlength{\unitlength}{420.47944336bp}%
    \ifx\svgscale\undefined%
      \relax%
    \else%
      \setlength{\unitlength}{\unitlength * \real{\svgscale}}%
    \fi%
  \else%
    \setlength{\unitlength}{\svgwidth}%
  \fi%
  \global\let\svgwidth\undefined%
  \global\let\svgscale\undefined%
  \makeatother%
  \begin{picture}(1,0.32180138)%
    \put(0.30842844,0.14240418){\color[rgb]{0,0,0}\makebox(0,0)[b]{\smash{$\Sigma_{\Gamma_H^{(1)}}$}}}%
    \put(0.55267605,0.14240418){\color[rgb]{0,0,0}\makebox(0,0)[b]{\smash{$\Sigma_{\Gamma_H^{(2)}}$}}}%
    \put(0.0641808,0.14240418){\color[rgb]{0,0,0}\makebox(0,0)[b]{\smash{$\Sigma_{\Gamma_H^{(0)}}$}}}%
    \put(0.74661579,0.14240418){\color[rgb]{0,0,0}\makebox(0,0)[b]{\smash{$\Sigma_{\Gamma_H^{(3)}}$}}}%
    \put(0.76076065,0.0995074){\color[rgb]{0,0,0}\makebox(0,0)[rb]{\smash{}}}%
    \put(0,0){\includegraphics[width=\unitlength,page=1]{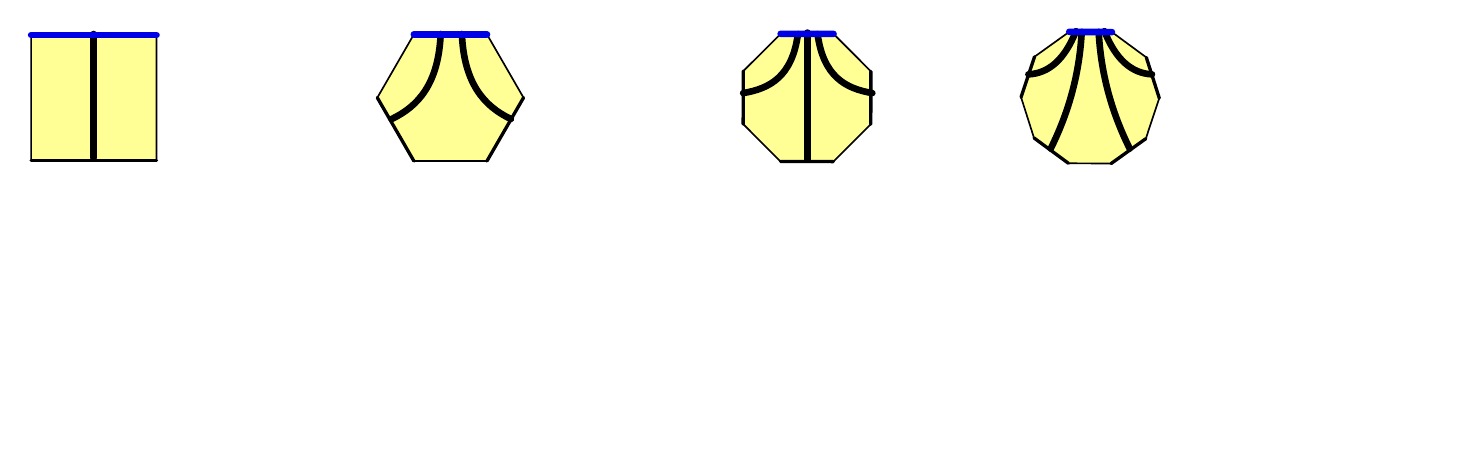}}%
    \put(0.93261993,0.23334022){\color[rgb]{0,0,0}\makebox(0,0)[b]{\smash{$\cdots$}}}%
    \put(0,0){\includegraphics[width=\unitlength,page=2]{CoSimpl2.pdf}}%
    \put(0.06350993,0.20993916){\color[rgb]{0,0,0}\makebox(0,0)[b]{\smash{$H$}}}%
    \put(0.26811602,0.22357956){\color[rgb]{0,0,0}\makebox(0,0)[b]{\smash{$H$}}}%
    \put(0.34748738,0.22357956){\color[rgb]{0,0,0}\makebox(0,0)[b]{\smash{$H$}}}%
    \put(0.50662545,0.24631357){\color[rgb]{0,0,0}\makebox(0,0)[b]{\smash{$H$}}}%
    \put(0.5975615,0.24631357){\color[rgb]{0,0,0}\makebox(0,0)[b]{\smash{$H$}}}%
    \put(0.55209347,0.20084555){\color[rgb]{0,0,0}\makebox(0,0)[b]{\smash{$H$}}}%
    \put(0.72032515,0.20993915){\color[rgb]{0,0,0}\makebox(0,0)[b]{\smash{$H$}}}%
    \put(0.77293467,0.20993915){\color[rgb]{0,0,0}\makebox(0,0)[b]{\smash{$H$}}}%
    \put(0.79319762,0.25995398){\color[rgb]{0,0,0}\makebox(0,0)[b]{\smash{$H$}}}%
    \put(0.70226158,0.25995398){\color[rgb]{0,0,0}\makebox(0,0)[b]{\smash{$H$}}}%
    \put(0,0){\includegraphics[width=\unitlength,page=3]{CoSimpl2.pdf}}%
    \put(0.47278294,0.02467858){\color[rgb]{0,0,0}\makebox(0,0)[lb]{\smash{$H$}}}%
    \put(0.43473109,0.06273038){\color[rgb]{0,0,0}\makebox(0,0)[lb]{\smash{Domain walls}}}%
  \end{picture}%
\endgroup%

\end{center}

\begin{remark}
In the case that the subalgebras $\h,\h^*\subset\g$ integrate to algebraic subgroups $H,H^*\subseteq G$ such that the group product \eqref{eq:invProd} is an isomorphism, then $H\to G\to G/H^*$ is a $\g$-equivariant isomorphism of $\K$-schemes.

In particular, $$\mc{O}(\mc{M}_{\Gamma_H^{(n)}}^\hbar)\cong \mc{O}\big((\overset{n+1}{\overbrace{H\times\cdots\times H}})/H\big)[\![\hbar]\!]\cong \mc{O}\big((G^{n+1})/(H\times (H^*)^{n+1})\big)[\![\hbar]\!].$$
More precisely, we have a canonical isomorphism $\mc{M}_{\Gamma_H^{(n)}}^\hbar\cong \mc{M}_{K(\Sigma_n)}^\hbar$, where 
the (geometric) coloured surface $\Sigma^{(n)}$ is pictured in the proof of Theorem~\ref{thm:HopfQuant1}.
\end{remark}

As in \cite{LiBland:2012vo}, we have a canonical isomorphism $\mc{M}_{\Gamma_H^{(1)}}= (H\times H)/H\cong H$ as Poisson manifolds (where $H$ is equipped with the Lie Poisson structure); and the inclusion $H\xhookrightarrow{\hbar=0}\mc{M}_{\Gamma_H^{(1)}}^\hbar$ is a deformation quantization of the Lie Poisson structure on $H$ (cf. \cite[Theorem~3]{LiBland:2014da}).
We will now show that $\mc{M}_{\Gamma_H^{(1)}}^\hbar$ is a monoid in $\Top_{\K\Alg_\hbar}$ deforming the group product on $H$. Equivalently, $\mc{O}(\mc{M}_{\Gamma_H^{(1)}}^\hbar)\cong \mc{O}(H)[\![\hbar]\!]$ has a Hopf algebra structure deforming the standard one on $\mc{O}(H)$.

Any monotone map $\tau:\{0,\dots,n\}\to \{0,\dots, m\}$ between the linearly ordered sets $0<\dots<n$ and $0<\dots<m$, defines a morphism of skeletized coloured surfaces $\tau_!:\Gamma^{(n)}_H\to \Gamma^{(m)}_H$ which sends the half edge $i\in \{0,\dots,n\}=\mbf{H}_{\Gamma^{(n)}_H}$ to the half edge $\tau(i)\in \{0,\dots,m\}=\mbf{H}_{\Gamma^{(m)}_H}$. This defines morphisms of the corresponding (non-commutative) ringed spaces:
$$\tau^*:\mc{M}_{\Gamma^{(m)}_H}^\hbar\to \mc{M}_{\Gamma^{(n)}_H}^\hbar$$


Let $\delta^i:\{0,\dots,n-1\}\to \{0,1,\dots,n\}$ denote the injective monotone map which `misses' $i\in \{0,1,\dots,n\}$. Now there is a unique monotone morphism of ciliated graphs $$P^{(2)}:=\delta^2_!\sqcup\delta^0_!:\Gamma_H^{(1)}\sqcup\Gamma_H^{(1)}\to \Gamma_H^{(2)}$$
which restricts to $\delta^2_!$ on the left copy of $\Gamma_H^{(1)}$ and $\delta^0_!$ on the right copy (this can be pictured as the left hand emebedding in Remark~\ref{rem:MultFig}). So we have a morphism 
\begin{equation}\label{eq:P2}(P^{(2)})^*:\mc{M}_{\Gamma_H^{(2)}}^\hbar\to \mc{M}_{\Gamma_H^{(1)}}^\hbar\times\mc{M}_{\Gamma_H^{(1)}}^\hbar.\end{equation}

\begin{Theorem}\label{thm:HopfQuant3}
	The morphism \eqref{eq:P2} is invertible, and the composite morphism \begin{equation}\label{eq:coMult3}\nabla:=(\delta^1)^*\circ\big((P^{(2)})^*\big)^{-1}:\mc{M}_{\Gamma_H^{(1)}}^\hbar\times\mc{M}_{\Gamma_H^{(1)}}^\hbar\to \mc{M}_{\Gamma_H^{(1)}}^\hbar\end{equation} is an associative product  on $\mc{M}_{\Gamma_H^{(1)}}^\hbar$ quantizing the Poisson algebraic group $H$; that is $H\xhookrightarrow{\hbar=0}\mc{M}_{\Gamma_H^{(1)}}^\hbar$ is both a morphism of monoids and a deformation quantization of the Poisson structure. 
	Dually, $\mc{O}(\mc{M}_{\Gamma_H^{(1)}}^\hbar)\cong\mc{O}(H)[\![\hbar]\!]$ is a Hopf-algebra which deformation quantizes the Poisson algebraic group $H$.
\end{Theorem}
\begin{remark}\label{rem:MultFig}
We can picture \eqref{eq:coMult3} as follows:
\begin{center}
  \def\svgwidth{\textwidth}
\begingroup%
  \makeatletter%
  \providecommand\color[2][]{%
    \errmessage{(Inkscape) Color is used for the text in Inkscape, but the package 'color.sty' is not loaded}%
    \renewcommand\color[2][]{}%
  }%
  \providecommand\transparent[1]{%
    \errmessage{(Inkscape) Transparency is used (non-zero) for the text in Inkscape, but the package 'transparent.sty' is not loaded}%
    \renewcommand\transparent[1]{}%
  }%
  \providecommand\rotatebox[2]{#2}%
  \ifx\svgwidth\undefined%
    \setlength{\unitlength}{444.86015625bp}%
    \ifx\svgscale\undefined%
      \relax%
    \else%
      \setlength{\unitlength}{\unitlength * \real{\svgscale}}%
    \fi%
  \else%
    \setlength{\unitlength}{\svgwidth}%
  \fi%
  \global\let\svgwidth\undefined%
  \global\let\svgscale\undefined%
  \makeatother%
  \begin{picture}(1,0.23917552)%
    \put(0,0){\includegraphics[width=\unitlength,page=1]{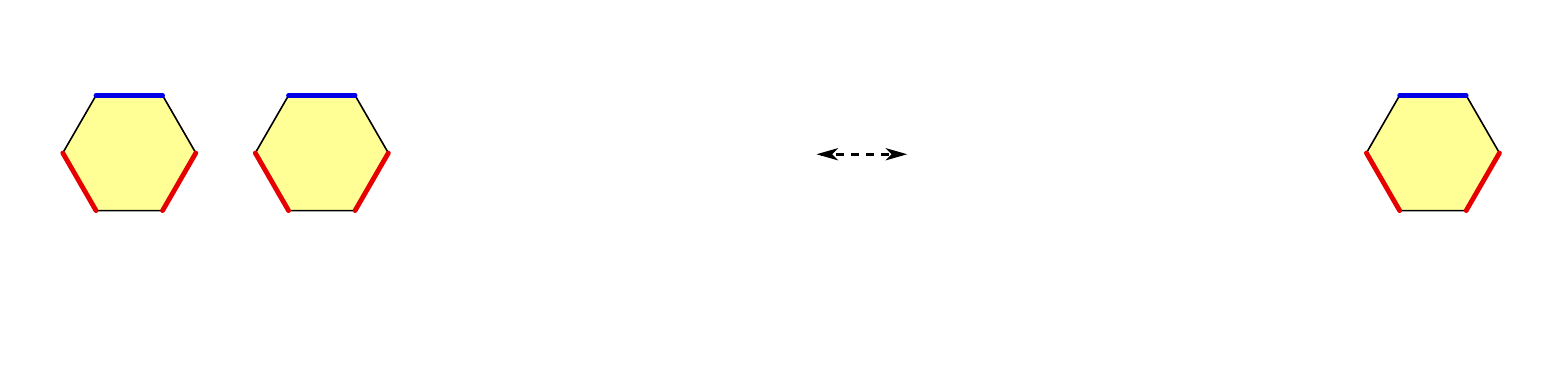}}%
    \put(0.55897301,0.14604063){\color[rgb]{0,0,0}\makebox(0,0)[b]{\smash{$=$}}}%
    \put(0.43118712,0.02169749){\color[rgb]{0,0,0}\makebox(0,0)[b]{\smash{$\Sigma_{\Gamma_H^{(2)}}$}}}%
    \put(0.68654827,0.02169749){\color[rgb]{0,0,0}\makebox(0,0)[b]{\smash{$\Sigma_{\Gamma_H^{(2)}}$}}}%
    \put(0.92752287,0.02169749){\color[rgb]{0,0,0}\makebox(0,0)[b]{\smash{$\Sigma_{\Gamma_H^{(1)}}$}}}%
    \put(0.14705288,0.02169749){\color[rgb]{0,0,0}\makebox(0,0)[b]{\smash{$\Sigma_{\Gamma_H^{(1)}}\sqcup\Sigma_{\Gamma_H^{(1)}}$}}}%
    \put(0,0){\includegraphics[width=\unitlength,page=2]{PoissonLieMult2.pdf}}%
    \put(0.05018867,0.11332926){\color[rgb]{0,0,0}\makebox(0,0)[b]{\smash{$H$}}}%
    \put(0,0){\includegraphics[width=\unitlength,page=3]{PoissonLieMult2.pdf}}%
    \put(0.1163661,0.11332926){\color[rgb]{0,0,0}\makebox(0,0)[b]{\smash{$H$}}}%
    \put(0,0){\includegraphics[width=\unitlength,page=4]{PoissonLieMult2.pdf}}%
    \put(0.17463295,0.11332926){\color[rgb]{0,0,0}\makebox(0,0)[b]{\smash{$H$}}}%
    \put(0,0){\includegraphics[width=\unitlength,page=5]{PoissonLieMult2.pdf}}%
    \put(0.24081038,0.11332926){\color[rgb]{0,0,0}\makebox(0,0)[b]{\smash{$H$}}}%
    \put(0,0){\includegraphics[width=\unitlength,page=6]{PoissonLieMult2.pdf}}%
    \put(0.89396016,0.11332926){\color[rgb]{0,0,0}\makebox(0,0)[b]{\smash{$H$}}}%
    \put(0,0){\includegraphics[width=\unitlength,page=7]{PoissonLieMult2.pdf}}%
    \put(0.96013752,0.11332926){\color[rgb]{0,0,0}\makebox(0,0)[b]{\smash{$H$}}}%
    \put(0,0){\includegraphics[width=\unitlength,page=8]{PoissonLieMult2.pdf}}%
    \put(0.75010229,0.12895651){\color[rgb]{0,0,0}\makebox(0,0)[b]{\smash{$H$}}}%
    \put(0,0){\includegraphics[width=\unitlength,page=9]{PoissonLieMult2.pdf}}%
    \put(0.61343011,0.12895651){\color[rgb]{0,0,0}\makebox(0,0)[b]{\smash{$H$}}}%
    \put(0,0){\includegraphics[width=\unitlength,page=10]{PoissonLieMult2.pdf}}%
    \put(0.49833772,0.12895651){\color[rgb]{0,0,0}\makebox(0,0)[b]{\smash{$H$}}}%
    \put(0,0){\includegraphics[width=\unitlength,page=11]{PoissonLieMult2.pdf}}%
    \put(0.36166554,0.12895651){\color[rgb]{0,0,0}\makebox(0,0)[b]{\smash{$H$}}}%
    \put(0,0){\includegraphics[width=\unitlength,page=12]{PoissonLieMult2.pdf}}%
    \put(0.42991322,0.0630322){\color[rgb]{0,0,0}\makebox(0,0)[b]{\smash{$H$}}}%
    \put(0,0){\includegraphics[width=\unitlength,page=13]{PoissonLieMult2.pdf}}%
    \put(0.68167774,0.0630322){\color[rgb]{0,0,0}\makebox(0,0)[b]{\smash{$H$}}}%
  \end{picture}%
\endgroup%

\end{center}	
\end{remark}

\begin{proof}
The proof is essentially the same as the proof of Theorem~\ref{thm:HopfQuant1}, so we omit many of the details.

For $j\in \{0,1,\dots,n-1\}$, we define $\epsilon^j_n:\{0,1\}\xrightarrow{i\to i+j} \{0,1,\dots,n\}$. Now there is a unique monotone morphism of ciliated graphs $$P^{(n)}:=(\epsilon^0_n)_!\sqcup\cdots \sqcup (\epsilon^{n-1}_n)_!:\overset{n}{\overbrace{\Gamma_H^{(1)}\sqcup\cdots\sqcup\Gamma_H^{(1)}}}\to \Gamma_H^{(n)}$$
which restricts to $\epsilon^j_n$ on the $i^{th}$ copy of $\Gamma_H^{(1)}$. Thus we have morphisms  
\begin{equation}\label{eq:Pn}(P^{(n)})^*:\mc{M}_{\Gamma_H^{(n)}}^\hbar\to \overset{n}{\overbrace{\mc{M}_{\Gamma_H^{(1)}}^\hbar\times\cdots \times\mc{M}_{\Gamma_H^{(n)}}^\hbar}}.\end{equation}

	As in the proof of Theorem~\ref{thm:HopfQuant1}, at $\hbar=0$, \eqref{eq:Pn} is the canonical identification
	$$\mc{M}_{\Gamma_H^{(n)}}= H^{n+1}/H \to \overset{n}{\overbrace{(H^2/H)\times\cdots\times (H^2/H)}}= \overset{n}{\overbrace{\mc{M}_{\Gamma_H^{(1)}}\times\cdots\times \mc{M}_{\Gamma_H^{(n)}}}},$$ (where all the $H$ actions are diagonal) and hence invertible. Since all the sheaves of algebras in \eqref{eq:Pn} are flat $\mbb{K}[\![\hbar]\!]$ algebras, it follows that \eqref{eq:Pn} remains invertible away from $\hbar=0$. 

The coassociativity of \eqref{eq:coMult3} follows from the fact that 
\begin{multline*} \nabla\circ (\id\times \nabla)=\big((\delta^1)^*\circ((P^{(2)})^*)^{-1}\big)\circ \bigg( (\id\times (\delta^2)^*)\circ \big(\id\times ((P^{(2)})^*)^{-1}\big)\bigg)
=\tau^*\circ((P^{(3)})^*)^{-1}\\
=\nabla\circ(\nabla\times\id)
\end{multline*}
where $\tau:\{0,1\}\to \{0,1,2,3\}$ is defined by $\tau(0)=0$ and $\tau(1)=3$ (this equation is the analogue of \eqref{eq:CoAss2}).

Note that $\mc{M}_{\Gamma_H^{(0)}}=H/H\cong\ast$ is just a point, and the inclusion $\mc{M}_{\Gamma_H^{(0)}}^\hbar\to \mc{M}_{\Gamma_H^{(1)}}^\hbar$ corresponding to the unique map $\{0,1\}\to \{0\}$ is the unit for the monoid.

 Now, at $\hbar=0$, \eqref{eq:coMult3} 
  is just the group product $H\times H\to H$ (cf. \cite[Example~7]{LiBland:2012vo}). Dually, it follows that $\mc{O}(\mc{M}_{\Gamma_H^{(1)}}^\hbar)\cong\mc{O}(H)[\![\hbar]\!]$ is a Hopf-algebra which deformation quantizes the Poisson algebraic group $H$.
\end{proof}

\subsection{Equivariant quantization}
Suppose that $H$ is a Poisson algebraic group, $X$ a scheme over $\K$ equipped with a Poisson structure, and $H\times X\to X$ an action which is also a Poisson morphism; in particular $H\in \Top_{\K\CPoiss}$ is a monoid and $X\in\Top_{\K\CPoiss}$ is a module for $H$. Then an \emph{equivariant deformation quantization} of $X$ is a monoid $H^\hbar\in \Top_{\K\Alg_\hbar}$ over $H$, together with a module $X^\hbar \in \Top_{\K\Alg_\hbar}$ for $H^\hbar$ over $X$ (cf. \cite{Etingof:1998vf}). The Groenewold - Van Hove No-Go theorem implies that equivariant quantizations need not exist in general \cite{Groenewold:1946tf,Gotay:2008wr,VanHove:1951te}.

Let $(\g,\h,\h^*)$ denote the Manin triple corresponding to $H$. Suppose that $(H,\g)$  acts on a scheme $M$ with coisotropic stabilizers. Then $M$ carries a natural Poisson structure whose bivector field is given by the formula
\begin{equation}\label{eq:PoissonStrOnHgqpc}
\pi=\frac{1}{2}\sum_i(\xi^i)_M\wedge (\xi_i)_M
\end{equation}
where for any $\xi\in \g$, $\xi_M\in \mf{X}(M)$ denotes the corresponding vector field on $M$, and $\{\xi_i\}\subset \h$ and $\{\xi^i\}\subset\h^*$ are basis in duality; moreover the action map $H\times M\to M$ is a Poisson morphism (cf. \cite{Lu06,LiBland:2009hx}).

In this section we will construct an equivariant quantization of $M$. First, however, we give some examples:

\begin{Example}[Variety of Lagrangian (coisotropic) subalgebras]
Let $\mc{L}(\g)\subseteq \mc{C}\mathrm{iso}(\g)$ denote the variety (resp. coisotropic) Lie subalgebras of $\g$. The Lie algebra $\g$ both act on $\mc{L}(\g)$ (resp. $\mc{C}\mathrm{iso}(\g)$) by adjunction, and the stabilizer of any point $\mf{l}\subseteq\mc{L}(\g)$ (resp. $\mf{l}\subseteq \mc{C}\mathrm{iso}(\g)$) contains $\mf{l}$; in particular it is coisotropic. The algebraic group $H$ also acts compatibly by adjunction. The Poisson structure \eqref{eq:PoissonStrOnHgqpc} on $\mc{L}(\g)$ (resp. $\mc{C}\mathrm{iso}(\g)$) is the one appearing in \cite{Evens:2001ue,Evens:2006kk,LiBland:2009hx}.

Note that if $N$ is any Poisson homogeneous space for $H$, the $H$-equivariant Drinfeld' map $M\to \mc{L}(\g)$ is Poisson (cf. \cite{Evens:2001ue}).
\end{Example}

\begin{Example}[Wonderful Compactification]
As explained in \S~\ref{sec:QPAG} the action of $\g$ on $H$ (which extends both the dressing action of $\h^*$ and the action of $\h$ by left-invariant vector fields) has coisotropic stabilizers. Therefore any $(H,\g)$-equivariant compactification $H\subset \bar H$ of $H$ will also have coisotropic stabilizers. 

	Suppose that $G$ is an algebraic group with Lie algebra $\g$, which is equipped with an algebraic morphism $H\to G$ over the inclusion $\h\subseteq \g$. Suppose further that $t:=\langle\cdot,\cdot\rangle^{-1}\in S^2(\g)^\g$ is $G$-invariant. Then the action of $G\times G$ on $G$ given by $$(g_1,g_2)\cdot g= g_1gg_2^{-1},\quad (g_1,g_2)\in G\times G, g\in G$$ has coisotropic stabilizers with respect to $t\oplus-t\in S^2(\g\oplus\g)$. Therefore any $G\times G$-equivariant compactification $G\subset \bar G$ of $G$ will also have $\g\oplus\bar\g$-coisotropic stabilizers; where $\g\oplus\bar\g=\g\oplus\g$ as a Lie algebra but comes equipped with $t_{\g\oplus\bar\g}:=t\oplus-t\in S^2(\g\oplus\g)$.
\end{Example}

\begin{Example}[Canonical Embeddings of Poisson Homogeneous Spaces]Suppose that $N$ is a Poisson homogeneous space for $H$.
 Then \cite{Drinfeld:1993il} associated to each point $x\in N$ a Lagrangian Lie subalgebra $\mf{l}_x:=\{\subset \g$ such that the stabilizer Lie subalgebra of $x\in N$ is $\mf{h}_x=\h\cap\mf{l}_x$.

Suppose that $G$ is an algebraic group with Lie algebra $\g$, which is equipped with an algebraic morphism $H\to G$ over the inclusion $\h\subseteq \g$. Suppose further that $t:=\langle\cdot,\cdot\rangle^{-1}\in S^2(\g)^\g$ is $G$-invariant. Finally, suppose that Drinfel'd's Lie subalgebra $\mf{l}_x\subset\g$ integrates to an algebraic subgroup $L_x\subset G$ whose preimage in $H$ is the stabilizer subgroup $_x=L_x\times_G H\subset H$ of $x\in N$. (Note that all these conditions are automatically satisfied if we work with formal rather than algebraic groups).

Then $G/L_x$ the action of $\g$ on $G/L_x$ has coisotropic stabilizers, and the natural embedding $N\cong H/H_x\hookrightarrow G/L_x$ is a Poisson morphism with respect to the Lu-Yakimov Poisson structure \eqref{eq:PoissonStrOnHgqpc} on $G/L_x$ (cf. \cite[Remark 2.10]{Lu07} and \cite{Enriquez:r0OMuB1e}).

In particular, Theorem~\ref{thm:EqQuant3} implies that given any (formal) Poisson group $H$ and any (formal) Poisson homogeneous space $N$, there exists a canonical embedding of $N$ into a (formal) Poisson space $G/L_x$ which can be equivariantly quantized. The idea of equivariantly quantizing $G/L_x$ as a means of equivariantly quantizing $N$ was first explored by Enriquez and Kosmann-Schwarzbach in \cite{Enriquez:r0OMuB1e}.
\end{Example}

Let $\Gamma_M^{(n)}$ denote the ciliated graph which has one (positive) vertex $\mbf{V}_{\Gamma_M^{(n)}}=\mbf{V}^+_{\Gamma_M^{(n)}}=\{\ast\}$, no full-edges $\mbf{E}_{\Gamma_M^{(n)}}=\emptyset$, and $n+2$ (positive) half-edges $\mbf{H}_{\Gamma_M^{(n)}}=\mbf{H}^+_{\Gamma_M^{(n)}}=\{0,1,2,\dots,n,\clubsuit\}$ with $0<1<2<\dots<n<\clubsuit$. We colour $\Gamma_M^{(n)}$ as follows: $C_\ast=H$, $X_0=\dots=X_{n}=H$, and $X_\clubsuit=M$. 
We may picture $\Gamma_M^{(n)}$ in terms of the corresponding surfaces $\Sigma_{{\Gamma_M^{(n)}}}$ as follows (cf. Remark~\ref{rem:GeomColSurf}):
\begin{center}
\begingroup%
  \makeatletter%
  \providecommand\color[2][]{%
    \errmessage{(Inkscape) Color is used for the text in Inkscape, but the package 'color.sty' is not loaded}%
    \renewcommand\color[2][]{}%
  }%
  \providecommand\transparent[1]{%
    \errmessage{(Inkscape) Transparency is used (non-zero) for the text in Inkscape, but the package 'transparent.sty' is not loaded}%
    \renewcommand\transparent[1]{}%
  }%
  \providecommand\rotatebox[2]{#2}%
  \ifx\svgwidth\undefined%
    \setlength{\unitlength}{420.44243164bp}%
    \ifx\svgscale\undefined%
      \relax%
    \else%
      \setlength{\unitlength}{\unitlength * \real{\svgscale}}%
    \fi%
  \else%
    \setlength{\unitlength}{\svgwidth}%
  \fi%
  \global\let\svgwidth\undefined%
  \global\let\svgscale\undefined%
  \makeatother%
  \begin{picture}(1,0.20313652)%
    \put(0,0){\includegraphics[width=\unitlength,page=1]{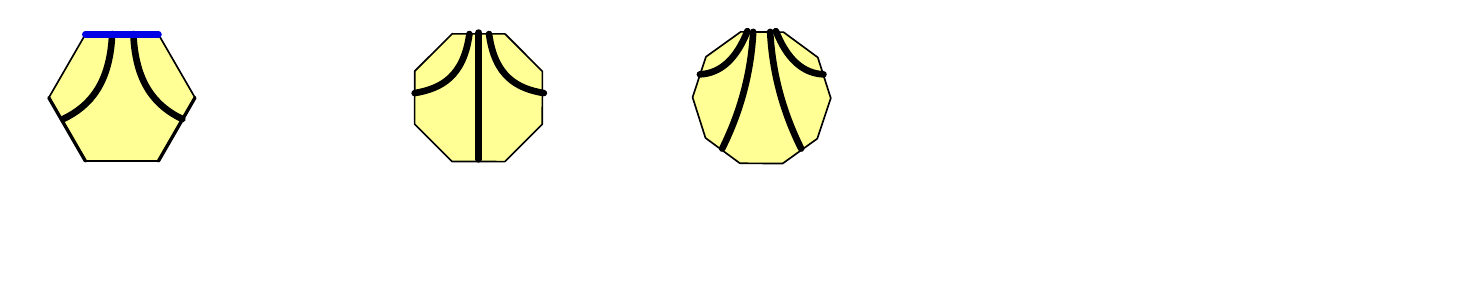}}%
    \put(0.0834222,0.02372354){\color[rgb]{0,0,0}\makebox(0,0)[b]{\smash{$\Sigma_{\Gamma_M^{(0)}}$}}}%
    \put(0.3276913,0.02372354){\color[rgb]{0,0,0}\makebox(0,0)[b]{\smash{$\Sigma_{\Gamma_M^{(1)}}$}}}%
    \put(0,0){\includegraphics[width=\unitlength,page=2]{CoSimplEQ.pdf}}%
    \put(0.52164812,0.02372354){\color[rgb]{0,0,0}\makebox(0,0)[b]{\smash{$\Sigma_{\Gamma_M^{(2)}}$}}}%
    \put(0.53579422,-0.01917702){\color[rgb]{0,0,0}\makebox(0,0)[rb]{\smash{}}}%
    \put(0.70766863,0.11466758){\color[rgb]{0,0,0}\makebox(0,0)[b]{\smash{$\cdots$}}}%
    \put(0,0){\includegraphics[width=\unitlength,page=3]{CoSimplEQ.pdf}}%
    \put(0.04310623,0.10490606){\color[rgb]{0,0,0}\makebox(0,0)[b]{\smash{$H$}}}%
    \put(0.12248458,0.10490606){\color[rgb]{0,0,0}\makebox(0,0)[b]{\smash{$M$}}}%
    \put(0.28163666,0.12764207){\color[rgb]{0,0,0}\makebox(0,0)[b]{\smash{$H$}}}%
    \put(0.3725807,0.12764207){\color[rgb]{0,0,0}\makebox(0,0)[b]{\smash{$M$}}}%
    \put(0.32710868,0.08217005){\color[rgb]{0,0,0}\makebox(0,0)[b]{\smash{$H$}}}%
    \put(0.49535517,0.09126445){\color[rgb]{0,0,0}\makebox(0,0)[b]{\smash{$H$}}}%
    \put(0.54796931,0.09126445){\color[rgb]{0,0,0}\makebox(0,0)[b]{\smash{$H$}}}%
    \put(0.56823405,0.14128368){\color[rgb]{0,0,0}\makebox(0,0)[b]{\smash{$M$}}}%
    \put(0.47729,0.14128368){\color[rgb]{0,0,0}\makebox(0,0)[b]{\smash{$H$}}}%
    \put(0,0){\includegraphics[width=\unitlength,page=4]{CoSimplEQ.pdf}}%
    \put(0.87331985,0.09573842){\color[rgb]{0,0,0}\makebox(0,0)[lb]{\smash{$H$}}}%
    \put(0.83526471,0.13379354){\color[rgb]{0,0,0}\makebox(0,0)[lb]{\smash{Domain walls}}}%
  \end{picture}%
\endgroup%

\end{center}

As in \S~\ref{sec:QPAG} any monotone map $\tau:\{0,\dots,n\}\to \{0,\dots, m\}$ defines two morphisms of skeletized coloured surfaces $\tau_!:\Gamma^{(n)}_H\to \Gamma^{(m)}_M$ and $\tau_!:\Gamma^{(n)}_M\to \Gamma^{(m)}_M$ via the corresponding map of half edges 
$$\begin{tikzpicture}
\mmat{m}{\mbf{H}_{\Gamma^{(n)}_H}&\mbf{H}_{\Gamma^{(m)}_M},\text{ and }\\
\mbf{H}_{\Gamma^{(n)}_M}&\mbf{H}_{\Gamma^{(m)}_M}.\\};
\draw[->] (m-1-1) -- node {$i\mapsto \tau(i)$} (m-1-2);
\draw[->] (m-2-1) -- node {$i\mapsto \tau(i)$} node[swap] {$\clubsuit\mapsto\clubsuit$} (m-2-2);
\end{tikzpicture}$$
 This, in turn, yields morphisms of the corresponding (non-commutative) ringed spaces:
\begin{align*}\tau^*:\mc{M}_{\Gamma^{(m)}_M}^\hbar&\to \mc{M}_{\Gamma^{(n)}_H}^\hbar,\\	
\tau^*:\mc{M}_{\Gamma^{(m)}_M}^\hbar&\to \mc{M}_{\Gamma^{(n)}_M}^\hbar\\	
\end{align*}


Consider the monotone morphism of ciliated graphs $$P_M^{(2)}:=\id_!\sqcup\delta^0_!:\Gamma_H^{(1)}\sqcup\Gamma_M^{(0)}\to \Gamma_M^{(1)}$$
which restricts to $\id_!$ on $\Gamma_H^{(1)}$ and $\delta^0_!$ on $\Gamma_M^{(1)}$ (this can be pictured as the left hand emebedding in Remark~\ref{rem:ActFig}). So we have a morphism 
\begin{equation}\label{eq:PM2}(P^{(2)})^*:\mc{M}_{\Gamma_M^{(1)}}^\hbar\to \mc{M}_{\Gamma_H^{(1)}}^\hbar\times\mc{M}_{\Gamma_M^{(0)}}^\hbar.\end{equation}
\begin{Theorem}\label{thm:EqQuant3}
	The morphism \eqref{eq:PM2} is invertible, and the composite morphism \begin{equation}\label{eq:coAct3}\nabla_M:=(\delta^1)^*\circ\big((P_M^{(2)})^*\big)^{-1}:\mc{M}_{\Gamma_H^{(1)}}^\hbar\times\mc{M}_{\Gamma_M^{(0)}}^\hbar\to \mc{M}_{\Gamma_M^{(0)}}^\hbar\end{equation} defines an action of the monoid $\mc{M}_{\Gamma_H^{(1)}}^\hbar$ on $\mc{M}_{\Gamma_M^{(0)}}^\hbar$, equivariantly quantizing the action of Poisson algebraic group $H$ on $M$. 
	
	Moreover, this equivariant quantization depends functorially on the  scheme $M$.
\end{Theorem}
\begin{remark}\label{rem:ActFig}
We can picture \eqref{eq:coAct3} as follows:
\begin{center}
\begingroup%
  \makeatletter%
  \providecommand\color[2][]{%
    \errmessage{(Inkscape) Color is used for the text in Inkscape, but the package 'color.sty' is not loaded}%
    \renewcommand\color[2][]{}%
  }%
  \providecommand\transparent[1]{%
    \errmessage{(Inkscape) Transparency is used (non-zero) for the text in Inkscape, but the package 'transparent.sty' is not loaded}%
    \renewcommand\transparent[1]{}%
  }%
  \providecommand\rotatebox[2]{#2}%
  \ifx\svgwidth\undefined%
    \setlength{\unitlength}{444.86015625bp}%
    \ifx\svgscale\undefined%
      \relax%
    \else%
      \setlength{\unitlength}{\unitlength * \real{\svgscale}}%
    \fi%
  \else%
    \setlength{\unitlength}{\svgwidth}%
  \fi%
  \global\let\svgwidth\undefined%
  \global\let\svgscale\undefined%
  \makeatother%
  \begin{picture}(1,0.23917552)%
    \put(0,0){\includegraphics[width=\unitlength,page=1]{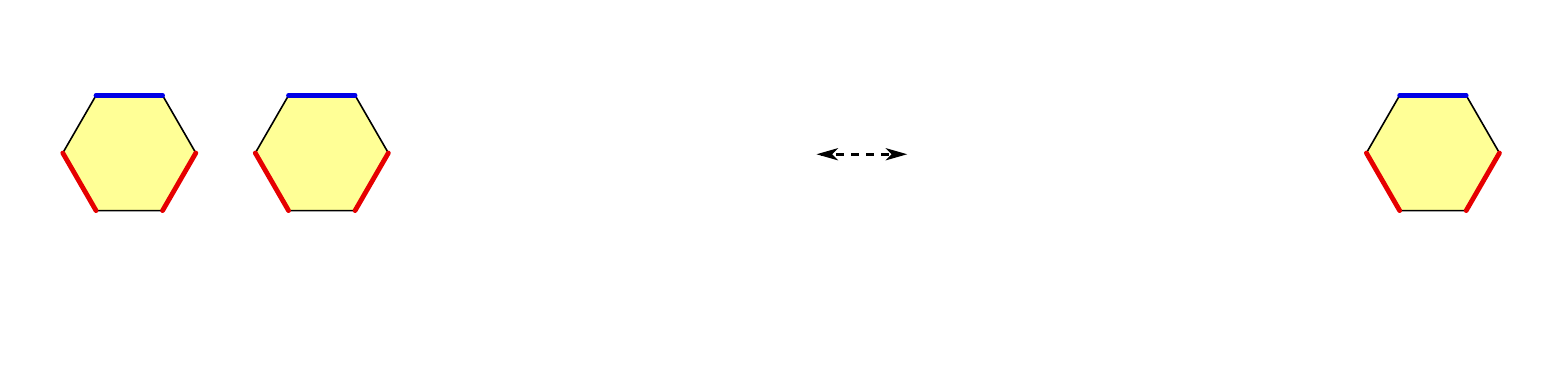}}%
    \put(0.55897301,0.14604063){\color[rgb]{0,0,0}\makebox(0,0)[b]{\smash{$=$}}}%
    \put(0.43118712,0.02169749){\color[rgb]{0,0,0}\makebox(0,0)[b]{\smash{$\Sigma_{\Gamma_M^{(1)}}$}}}%
    \put(0.68654827,0.02169749){\color[rgb]{0,0,0}\makebox(0,0)[b]{\smash{$\Sigma_{\Gamma_M^{(1)}}$}}}%
    \put(0.92752287,0.02169749){\color[rgb]{0,0,0}\makebox(0,0)[b]{\smash{$\Sigma_{\Gamma_M^{(0)}}$}}}%
    \put(0.14705288,0.02169749){\color[rgb]{0,0,0}\makebox(0,0)[b]{\smash{$\Sigma_{\Gamma_H^{(1)}}\sqcup\Sigma_{\Gamma_M^{(0)}}$}}}%
    \put(0,0){\includegraphics[width=\unitlength,page=2]{PoissonLieAct2.pdf}}%
    \put(0.05018867,0.11332926){\color[rgb]{0,0,0}\makebox(0,0)[b]{\smash{$H$}}}%
    \put(0,0){\includegraphics[width=\unitlength,page=3]{PoissonLieAct2.pdf}}%
    \put(0.1163661,0.11332926){\color[rgb]{0,0,0}\makebox(0,0)[b]{\smash{$H$}}}%
    \put(0,0){\includegraphics[width=\unitlength,page=4]{PoissonLieAct2.pdf}}%
    \put(0.17463295,0.11332926){\color[rgb]{0,0,0}\makebox(0,0)[b]{\smash{$H$}}}%
    \put(0,0){\includegraphics[width=\unitlength,page=5]{PoissonLieAct2.pdf}}%
    \put(0.24081038,0.11332926){\color[rgb]{0,0,0}\makebox(0,0)[b]{\smash{$M$}}}%
    \put(0,0){\includegraphics[width=\unitlength,page=6]{PoissonLieAct2.pdf}}%
    \put(0.89396016,0.11332926){\color[rgb]{0,0,0}\makebox(0,0)[b]{\smash{$H$}}}%
    \put(0,0){\includegraphics[width=\unitlength,page=7]{PoissonLieAct2.pdf}}%
    \put(0.96013752,0.11332926){\color[rgb]{0,0,0}\makebox(0,0)[b]{\smash{$M$}}}%
    \put(0,0){\includegraphics[width=\unitlength,page=8]{PoissonLieAct2.pdf}}%
    \put(0.75010229,0.12895651){\color[rgb]{0,0,0}\makebox(0,0)[b]{\smash{$M$}}}%
    \put(0,0){\includegraphics[width=\unitlength,page=9]{PoissonLieAct2.pdf}}%
    \put(0.61343011,0.12895651){\color[rgb]{0,0,0}\makebox(0,0)[b]{\smash{$H$}}}%
    \put(0,0){\includegraphics[width=\unitlength,page=10]{PoissonLieAct2.pdf}}%
    \put(0.49833772,0.12895651){\color[rgb]{0,0,0}\makebox(0,0)[b]{\smash{$M$}}}%
    \put(0,0){\includegraphics[width=\unitlength,page=11]{PoissonLieAct2.pdf}}%
    \put(0.36166554,0.12895651){\color[rgb]{0,0,0}\makebox(0,0)[b]{\smash{$H$}}}%
    \put(0,0){\includegraphics[width=\unitlength,page=12]{PoissonLieAct2.pdf}}%
    \put(0.42991322,0.0630322){\color[rgb]{0,0,0}\makebox(0,0)[b]{\smash{$H$}}}%
    \put(0,0){\includegraphics[width=\unitlength,page=13]{PoissonLieAct2.pdf}}%
    \put(0.68167774,0.0630322){\color[rgb]{0,0,0}\makebox(0,0)[b]{\smash{$H$}}}%
  \end{picture}%
\endgroup%

\end{center}	
\end{remark}
\begin{proof}
The proof that \eqref{eq:coAct3} is a direct analogue of the proof of Theorem~\ref{thm:HopfQuant3}, and so we omit it. 

To see that this equivariant quantization depends naturally on $M$, notice that any $(H,\g)$-equivariant morphism between schemes $M\to N$ on which $\g$ acts with coisotropic stabilizers defines a morphism $\Gamma_N^{(n)}\to \Gamma_M^{(n)}$ of the corresponding skeletized coloured surfaces; that is $\Gamma_M^{(n)}$ depends (contravariantly) functorially on $M$. Similarly, $\mc{M}^\hbar_{\Gamma_M^{(n)}}$ depends (contravariantly) functorially on $\Gamma_M^{(n)}$ (cf. Theorem~\ref{thm:AlgFunct}). Hence $\mc{M}^\hbar_{\Gamma_M^{(0)}}$ (as well as all the structural maps) depends naturally on $M$.
\end{proof}

\bibliography{basicbib}{}
\bibliographystyle{plain}
\end{document}